\documentclass{amsart}

\usepackage{amsmath,amssymb}
\usepackage{amsthm}
\usepackage{amsfonts}
\usepackage{amscd}
\usepackage{comment}

\newtheorem{df}{Definition}[section]
\newtheorem{thm}[df]{Theorem}
\newtheorem{prop}[df]{Proposition}
\newtheorem{lemm}[df]{Lemma}
\newtheorem{cor}[df]{Corollary}

\newtheorem{rem}[df]{Remark}

\newcommand{\Gwenael}{Gw\'{e}na\"{e}l}

\newcommand{\homfly}{HOMFLY-PT\ }

\newcommand{\Q}{\mathbb{Q}}
\newcommand{\R}{\mathbb{R}}
\newcommand{\Z}{\mathbb{Z}}
\newcommand{\Zlarger}[1]{\mathbb{Z}_{\geq #1} }
\newcommand{\C}{\mathbb{C}}
\newcommand{\Liesl}[1]{\mathrm{sl}(#1)}

\newcommand{\surface}{\Sigma}

\newcommand{\geh}{\mathfrak{g}}
\newcommand{\cgeh}{\widehat{\mathfrak{g}}}
\newcommand{\filt}[1]{F^{ #1}}
\newcommand{\filtn}[1]{\{ #1 \}_{n \geq 0}}

\newcommand{\comp}[1]{\underleftarrow{\lim}_{#1 \rightarrow \infty}}
\newcommand{\GL}{\mathbb{Q} \zettaiti{{\pi}_1}}

\newcommand{\GLM}{\mathbb{Q} \pi_1 }

\newcommand{\cGL}{\widehat{\mathbb{Q} \lvert \pi_1 \rvert }}
\newcommand{\cGLM}{\widehat{\mathbb{Q} \pi_1}}

\newcommand{\Uh}{\mathcal{U}_h }
\newcommand{\cUh}{\widehat{\mathcal{U}_h}}

\newcommand{\Loc}{(\sum_{* \in \Zlarger{0}} \frac{1}{h^*} F^{3*})}
\newcommand{\cLoc}{\widehat{(\sum_{* \in \Zlarger{0}} \frac{1}{h^*} F^{3*})}}

\newcommand{\skein}{\mathcal{S}}

\newcommand{\cskein}{\widehat{\mathcal{S}} }
\newcommand{\hskein}{\mathcal{S}^{-1}}
\newcommand{\chskein}{\widehat{\mathcal{S}}^{-1}}

\newcommand{\tskein}{\mathcal{A}}

\newcommand{\ctskein}{\widehat{\mathcal{A}}}
\newcommand{\tzeroskein}{\mathcal{A}_0 }
\newcommand{\ctzeroskein}{\widehat{\mathcal{A}}_0}

\newcommand{\graphset}{\Upsilon}
\newcommand{\graphele}{\gamma}
\newcommand{\ginv}{\natural}
\newcommand{\MCG}{\mathcal{M}}
\newcommand{\torelli}{\mathcal{I}}
\newcommand{\jhom}[1]{\tau_{#1}}
\newcommand{\jkernel}{\mathcal{K}}

\newcommand{\jfd}[1]{\mathcal{I}^{#1} }
\newcommand{\Kfd}[1]{\mathcal{I}^{#1}_{
\mathcal{S}}}
\newcommand{\Hfd}[1]{\mathcal{I}^{#1}_{
\mathcal{A}} }

\newcommand{\zettaiti}[1]{\lvert #1 \rvert}
\newcommand{\shuugou}[1]{\{ #1 \}}

\newcommand{\PPsi}[2]{ \Psi_{#1}^{#2}}

\newcommand{\id}{\mathrm{id}}

\newcommand{\Diff}{\mathrm{Diff}}
\newcommand{\gyaku}[1]{ #1^{-1}}

\newcommand{\defeq}{\stackrel{\mathrm{def.}}{=}}
\newcommand{\tosha}[1]{\stackrel{ #1 }{\to}}
\newcommand{\Aut}{\mathrm{Aut}}
\newcommand{\bch}{\mathrm{bch}}

\renewcommand{\tilde}{\widetilde}
\begin{document}

\title[The Johnson homomorphisms and homology cylinders]{The total Johnson 
homomorphism on the homology cylinder and
the bracket-quantization
\homfly skein algebra}
\author{Shunsuke Tsuji}
\email{tsujish@kurims.kyoto-u.ac.jp}
\date{}
\maketitle 

\begin{abstract}

A homology cylinder of a surface induces an automorphism of the completed group
ring of the fundamental group of the surface.
We introduce a   new method of 
computing the automorphism by
using the Goldman Lie algebra of the surface
or some skein algebra.
In particular, we give a refinement of 
a formula by Kuno and Massuyeau \cite{KM2019}.

\end{abstract}

\tableofcontents

\subsection*{Acknowledgment}

The author is grateful to Kazuo Habiro, Nariya Kawazumi, Yusuke Kuno, \Gwenael \ Massuyeau, Jun Murakami and Tomotada Ohtsuki for helpful comments and discussions. This work was partially supported by JSPS KAKENHI Grant Number 18J00305, the Research Institute for Mathematical Sciences and an International Joint Usage/Research Center located at Kyoto University.

\section{Introduction and the main result}

\subsection{Back ground}
\label{subsection_back_ground}

The mapping class group $\MCG (\surface)$ of a compact connected oriented surface $\surface$ with boundary is the set of the isotopy classes of diffeomorphism fixing the boundary $\partial \surface$ pointwise. The group $\MCG (\surface)$ acts naturally on the fundamental group $\pi_1 (\surface,*)$ of the surface $\surface$ with basepoint $* \in \partial \surface$. The action is faithful by a theorem of Dehn and Nielsen, and so provides us with an algebraic method of studying the mapping class group.


We denote by $\surface_{g,1}$ a compact connected oriented surface of genus $g$ with one boundary component. In the study of the action of $\MCG (\surface_{g,1})$ on $\pi_1 (\surface_{g,1})$, Johnson introduced a series of homomorphisms, which we call the Johnson homomorphisms.
Morita \cite{Morita1989} \cite{Morita_Casson_core} discovered an explicit relationship between the Johnson homomorphisms and the Casson invariant, which is the finite type invariant of $3$-manifold of order $1$. Furthermore, Garoufalidis and Levine \cite{GL} made us understand the Johnson homomorphisms as the full tree parts of finite type invariants for ``homology cobordisms".



Kawazumi, Kuno, Massuyeau, and Turaev \cite{Kawazumi} \cite{KK} \cite{MT} studied the Goldman Lie algebra using a grading defined by an augmentation ideal. The Torelli group $\torelli (\surface_{g,1})$ is the kernel of the action of $\MCG (\surface_{g,1})$ on $H_1 (\surface_{g,1}, \Z)$. Furthermore, Kawazumi and Kuno constructed an injective map from the Torelli group $\torelli (\surface_{g,1})$ to the completed Goldman Lie algebra in terms of the grading, which recovers the Johnson homomorphisms. In other words, we can understand the Johnson homomorphisms as the associated graded quotients.



We denote by $\mathcal{C} (\surface)$ the set of diffeomorphism classes of homology cobordisms of the surface, which will be defined later. By ``stacking sums", we may consider $\mathcal{C} (\surface)$ as a monoid.


Let $\filtn{(\pi_1 (\surface_{g,1}))_n}$ be the lower central series of $\pi_1 (\surface_{g,1})$. For any $N \in \Zlarger{1}$, using the Stallings's theorem \cite{Stallings}, Garoufalidis and Levine \cite{GL} introduced a monoid homomorphism
\begin{equation*}
\Phi'_N:\mathcal{C} (\surface_{g,1}) \to \Aut (\pi_1 (\surface_{g,1})/((\pi_1 (\surface_{g,1}))_N).
\end{equation*}
 We can extend it to 
\begin{equation*}
\Phi =( \cdot )_*:\mathcal{C} (\surface_{g,1}) \to \Aut (\cGLM (\surface_{g,1})),
\end{equation*}
 where $\cGLM (\surface_{g,1})$ is the completed group ring of $\pi_1 (\surface_{g,1})$. The monoid homomorphisms provides us with an algebraic method of studying $\mathcal{C} (\surface_{g,1})$ in terms of the finite type invariants of 3-manifolds and the Goldman Lie algebra. We remark that we can also define a monoid homomorphism 
\begin{equation*}
\Phi =( \cdot )_*:\mathcal{C} (\surface) \to \Aut (\cGLM (\surface)),
\end{equation*}
for the surface $\surface$.



In the case $\surface=\surface_{g,1}$, Garoufalidis and Levine \cite{GL} derived the Johnson homomorphisms of $\mathcal{C} (\surface_{g,1})$ from the action $\Phi'_N$. By their paper, we understand the Johnson homomorphisms as the full tree parts of finite type invariants for homology cobordisms. In terms of the Goldman Lie algebra, Kuno and Massuyeau \cite{KM2019} have recently obtained a formula for the Johnson homomorphisms of $\mathcal{C} (\surface_{g,1})$ using  `` generalized Dehn twists".



We can define a submonoid $\mathcal{IC} (\surface_{g,1})$ of $\mathcal{C} (\surface_{g,1})$ analogous to the Torelli group $\torelli (\surface_{g,1})$ (Definition \ref{df_homology_cylinder}) and construct a map from $\mathcal{IC} (\surface_{g,1})$ to the completed Goldman Lie algebra. The action $\Phi =( \cdot )_*:\mathcal{C} (\surface_{g,1}) \to \Aut (\cGLM (\surface_{g,1}))$ restricted to $\mathcal{IC} (\surface_{g,1})$ can be regarded as the exponential of the map to the completed Goldman Lie algebra (Theorem \ref{thm_main_hom_Johnson}). Furthermore, we can understand the Johnson homomorphisms as the associated graded quotients of the Goldman Lie algebra.



In \S \ref{subsection_notation_for_the_Goldman}  and \S\ref{subsection_notation_for_homology_cylinders}, we introduce the notion of the Goldman Lie algebra and homology cylinders of the surface $\surface$. In \S \ref{subsection_main_result}, we present our main result. In \S 1.5, we explain that we understand the Johnson homomorphisms as the associated graded quotients of the Goldman Lie algebra.


\subsection{The Goldman Lie algebra}
\label{subsection_notation_for_the_Goldman}


We set up notation for the Goldman Lie algebra on the surface $\surface$. By using the Goldman Lie algebra, we will describe the action $\Phi$ of a homology cylinder of $\surface$ on the completed group ring in \S \ref{subsection_main_result}.



Let $\zettaiti{\pi_1} (\surface)$ denote the set of conjugacy classes in $\pi_1 (\surface)$. Goldman defined a Lie bracket on the free $\Q$-vector space $\GL (\surface)$ over the set $\zettaiti{\pi_1} (\surface)$, which we call the Goldman Lie bracket. Kawazumi and Kuno \cite{Kawazumi} introduced Lie actions $\sigma:\GL (\surface) \times \GLM (\surface, *) \to \GLM (\surface, *)$ and $\sigma:\GL (\surface) \times \GLM (\surface, *_1, *_2) \to \GLM (\surface, *_1, *_2)$ for $*, *_1, *_2 \in \partial \surface$. The Lie bracket and the Lie actions satisfy the condition
\begin{align*}
&[ \zettaiti{ (\ker \varepsilon)^n},
\zettaiti{(\ker \varepsilon)^m}]
\subset \zettaiti{ (\ker \varepsilon)^{n+m-2}}, \\
&\sigma (\zettaiti{ (\ker \varepsilon)^n})
((\ker \varepsilon)^m \GLM (\surface, *_1, *_2))
\subset 
(\ker \varepsilon)^{n+m-2} \GLM (\surface, *_1, *_2),
\end{align*}
where $\zettaiti{\cdot} : \GLM (\surface) \to \GL (\surface)$
is the quotient map. Here, for any group $G$, $\varepsilon :\Q G \to
\Q$ is the augmentation map defined by $g\in G \mapsto 1$.


To define the logarithm and the exponential on the group ring, we define filtrations
\begin{align*}
&\filt{n} \GLM (\surface, *)
\defeq (\ker \varepsilon)^n, \\ 
&\filt{n} \GLM (\surface, *_1, *_2), 
\defeq (\ker \varepsilon)^n \GLM (\surface, *_1, *_2), 
 \\
&\filt{n} \GL (\surface) 
\defeq \zettaiti{ (\ker \varepsilon)^n}. 
\end{align*}
and completions
\begin{align*}
&\cGLM (\surface, *) \defeq
\comp{i} \GLM (\surface, *)/ \filt{i} \GLM (\surface,*), \\
&\cGLM (\surface, *_1, *_2) \defeq
\comp{i} \GLM (\surface, *_1, *_2)/ \filt{i}
\GLM (\surface, *_1, *_2), 
 \\
&\cGL (\surface ) \defeq
\comp{i} \GL (\surface )/
\filt{i} \GL (\surface ).
\end{align*}
The completions also have filtrations such that 
\begin{equation*}
F^n \widehat{V}= \ker ( \widehat{V} \to V/\filt{n} V)
\end{equation*}
where 
\begin{equation*}
V = \GLM (\surface, *), \GLM (\surface, *_1, *_2),  \ \mathrm{and} \
\GL (\surface)
\end{equation*}
 and $\widehat{V}$ is the completion of $V$. Using the Baker-Campbell-Hausdorff series, we consider $\filt{3} \cGL (\surface)$ as a group.



Using the filtrations and the completions, we can discuss a relationship between the Goldman Lie algebra and the mapping class group. Here we denote by $t_c$ the Dehn twist along a simple closed curve $c$. Let $\torelli' (\surface)$ be the subgroup of the mapping class group $\MCG (\surface)$ generated by 
\begin{align*}
\shuugou{t_{c_1} t_{c_2}^{-1}|
(c_1, c_2) \mathrm{\ bounds \ a  \ surface}} \ \mathrm{and} \ 
\shuugou{t_{c_0}|
c_0 \mathrm{\ bounds \ a \ surface}}.
\end{align*}
 Then there is an embedding $\zeta$ from $\torelli' (\surface)$ to $\filt{3} \cGLM (\surface)$. The map $\zeta$ satisfies
\begin{equation*}
\xi_* = \exp (\zeta (\xi)) \defeq \sum_{i=0}^\infty \frac{1}{i!} (\sigma (\zeta (\xi)))^i:
\cGLM (\surface, *_1, *_2) \to \cGLM (\surface, *_1, *_2).
\end{equation*}
For details, see \S \ref{section_Goldman} or \cite{Kawazumi} \cite{KK} \cite{MT}.


\subsection{Homology cylinders}
\label{subsection_notation_for_homology_cylinders}


We recall the definition of homology cobordisms. Let $(M, \alpha)$ be a pair of a compact connected oriented $3$-manifold and a diffeomorphism $\partial (\surface \times I) \to \partial M$, where $I$ is the unit interval $[0,1]$. If the embedding maps
\begin{equation*}
\alpha_0:\surface \to M, p \mapsto \alpha (p,0) \ \mathrm{and} \ 
\alpha_1: \surface \to M, p \mapsto \alpha (p,1)
\end{equation*}
 induce the isomorphisms
\begin{equation*}
\alpha_{0*}: H_1 (\surface, \Z) \to H_1 (M, \Z) \ \mathrm{and} \
\alpha_{1*}: H_1 (\surface, \Z) \to H_1 (M, \Z),
\end{equation*}
we call it a homology cobordism.



The set of the diffeomorphism classes of homology cobordisms
\begin{equation*}
\mathcal{C} (\surface) \defeq \shuugou{\mathrm{homology \ cobordisms}}/
\mathrm{diffeomorpic}
\end{equation*}
is a monoid by ``stacking sums". The action $\mathcal{C}(\surface) \to \Aut (\cGLM (\surface, *_1, *_2))$ and the Johnson homomorphisms introduced by Garoufalidis and Levine are monoid homomorphism. 



We define the action $\Phi(\cdot) =( \cdot)_*:\mathcal{C} (\surface) \to \Aut (\cGLM (\surface, *_1, *_2))$ as follows. Let $(M, \alpha)$ be a homology cobordism of the surface
$\surface$. By Stallings' theorem, the embedding maps $\alpha_0, \alpha_1$ induce the isomorphisms
\begin{align*}
&\alpha_{0*}:\cGLM (\surface, *_1, *_2) \to \cGLM (M, \alpha (*_1,0), \alpha (*_2,0)), \\
&\alpha_{1*}:\cGLM (\surface, *_1, *_2) \to \cGLM (M, \alpha (*_1,1), \alpha (*_2,1)). \\
\end{align*}
 Here, for any $i=0,1$, we denote by $\cGLM (M, \alpha (*_1 ,i), \alpha (*_2 ,i))$ the completion
\begin{equation*}
\comp{n} \GLM (M, \alpha (*_1 ,i), \alpha (*_2 ,i))/
(\ker \varepsilon)^n \GLM (M, \alpha (*_1 ,i), \alpha (*_2 ,i)).
\end{equation*}
Let $\Diamond:\cGLM (M, \alpha (*_1,0), \alpha (*_2,0))
\to \cGLM (M, \alpha (*_1,1), \alpha (*_2,1))
$ be the automorphism defined as $\Diamond (\gamma) \defeq (\gamma_{*01})^{-1} \gamma 
\gamma_{*01}$, where the continuous map $I \to M, t \mapsto \alpha (*, t)$ represents the path
$\gamma_{*01}$. The composite $\alpha_{1*}^{-1} \circ  \Diamond \circ \alpha_{0*}$ defines a monoid homomorphism
\begin{equation*}
\Phi =( \cdot )_*:\mathcal{C} (\surface) \to \Aut (\cGLM (\surface, *_1, *_2)),
(M, \alpha) \mapsto (M, \alpha)_* \defeq \alpha_{1*}^{-1} \circ \Diamond \circ \alpha_{0*},
\end{equation*}
which we call the action $\Phi$ of $\mathcal{C} (\surface)$ on 
$\cGLM (\surface, *_1, *_2)$.


\begin{df}
\label{df_homology_cylinder}
We denote by $\mathcal{IC} (\surface)$ the subset of $\mathcal{C}(\surface)$ consisting of all elements $\Xi$ satisfying 
\begin{equation*}
(\id- \Xi_*)(\cGLM (\surface,*_1, *_2)) \subset
F^2 \cGLM (\surface, *_1, *_2)
\end{equation*}
for any $*_1, *_2 \in \partial \surface$.
We call a homology cobordism representing an element of 
$\mathcal{IC} (\surface)$ a homology cylinder.


\end{df}


To state our main result, we introduce a homology cylinder version of a Heegaard splitting of a 3-manifold. We choose disjoint closed disks $D_1, \cdots, D_N$ on $\surface$ where $N$ is at least $1$. Let $\surface_\mathrm{st}$ be the closure of $\surface \backslash (D_1 \sqcup \cdots \sqcup D_N)$. We denote the extended surface $\partial (\surface_\mathrm{st} \times I) \backslash (\partial \surface \times  (0,1)) = \surface_\mathrm{st} \times \shuugou{0,1} \cup (\partial (D_1 \sqcup \cdots \sqcup D_N) \times I)$ by $\tilde{\surface}_\mathrm{st}$. For an element $\xi \in \torelli' (\tilde{\surface}_\mathrm{st})$, we take a representative of $\xi$ and denote it by the same symbol $\xi$. Here $e_\mathrm{st}:\surface_\mathrm{st} \times I  \to \surface \times I,(p,t) \mapsto (p, \frac{1+t}{3})$ is the standard embedding map. We define a homology cylinder $(\surface \times I) (e_\mathrm{st}, \xi)$  by 
\begin{equation*}
( (\mathrm {the \ closure \ of \ }(\surface \times I \backslash
e_\mathrm{st} (\surface_\mathrm{st} \times I)))\cup_{e_\mathrm{st} \circ \xi} (\surface_\mathrm{st} \times I), \id_{\partial(\surface \times I)}).
\end{equation*}


\begin{prop}[Proposition \ref{prop_standard_embedding}]

For any homology cylinder $(M, \alpha)$ of the surface $\surface$, there exist some disjoint closed disks $D_1 , \cdots, D_N$ and an element $\xi \in \torelli' (\tilde{\surface_\mathrm{st}})$ such that the diffeomorphism class of $(M, \alpha)$ equals that of 
$(\surface \times I) (e_\mathrm{st}, \xi)$.

\end{prop}
For details, see \S \ref{section_homology_cylinders}.

To state our main result, we need to define a map $v:\filt{3} \cGL (\tilde{\surface}_\mathrm{st}) \to \filt{3} \cGL (\surface_\mathrm{st})$ as follows. The maps
\begin{equation*}
\tilde{\surface_\mathrm{st}} \to \surface_\mathrm{st} \times I 
\twoheadrightarrow \surface_\mathrm{st}, \ 
\surface_\mathrm{st} \to \tilde{\surface}_\mathrm{st}, p \mapsto (p,1), \ \mathrm{and} \
\surface_\mathrm{st} \to \tilde{\surface}_\mathrm{st}, p \mapsto (p,0)
\end{equation*}
induce linear maps
\begin{align*}
&\kappa_*:\cGL (\tilde{\surface_\mathrm{st}}) \to \cGL (\surface_\mathrm{st}), \
\iota_{0*}: \cGL (\surface_\mathrm{st}) \to \cGL (\tilde{\surface}_\mathrm{st}), \\ 
&\mathrm{and} \
\iota_{1*}:\cGL (\surface_\mathrm{st}) \to \cGL (\tilde{\surface_\mathrm{st}}).
\end{align*} 
We remark that the second map is a Lie algebra homomorphisms, but the first one and the third one are not.


\begin{df}
For $x \in \filt{3} \cGL (\tilde{\surface}_\mathrm{st})$,
a sequence $\shuugou{v_n (x)}_{n \in \Zlarger{1}} \subset 
\filt{3} \cGL (\surface_\mathrm{st})$ is defined
by 
\begin{align*}
v_1 (x) \defeq \kappa_* (x), \ 
v_{n+1} (x) \defeq v_{n} (x) +\kappa_* (\bch (-\iota_{1*} (v_{n} (x)), x)).
\end{align*}
Furthermore, $v(x)$ is defined by
\begin{equation*}
v(x) \defeq \lim_{n \to \infty} v_n(x).
\end{equation*}
\end{df}

The element $v(x)$ is a unique one satisfying  $\kappa_* (\bch(-\iota_{1*} (v(x)), x))=0$. In other words, $v(x)$ is the unique solution of $\kappa_* (\bch(-\iota_{1*} ( \cdot), x))=0$. There is a topological definition of $v$ using some skein algebra in \S \ref{section_qbtskein}. Using this map, we state our main theorem proved in \S \ref{section_proof_main_theorem_qbtskein}.


\subsection{Main result}
\label{subsection_main_result}


The paper aims to construct a monoid homomorphism 
\begin{equation*}
\tilde{\zeta}:\mathcal{IC} (\surface) \to \filt{3}\cGL (\surface)
\end{equation*}
 having the following property 
\begin{equation*}
\Phi (\cdot) =\exp (\sigma (\tilde{\zeta} (\cdot)))\in \Aut (\cGLM (\surface,*_1, *_2))
\end{equation*}
 as follows.


\begin{thm}
\label{thm_main_hom_Johnson}
We have
\begin{equation*}
(\surface \times I )(e_\mathrm{st} , \xi)_*
=\exp (e'_*(v(\zeta  (\xi )))):
\cGLM (\surface, *_1, *_2) \to \cGLM (\surface, *_1, *_2)
\end{equation*}
for $*_1, *_2 \in \partial \surface$,
where $e'_* :\cGL (\surface_\mathrm{st}) \to \cGL (\surface)$
is induced by $\surface_\mathrm{st} \hookrightarrow \surface.$ 
Using the formula, the map
\begin{equation*}
\tilde{\zeta}:\mathcal{IC} (\surface) \to
(F^3 \cGL (\surface), \bch),
\end{equation*}
defined by
\begin{equation*}
(\surface \times I) (e_\mathrm{st} , \xi) \mapsto
e'_* (v (\zeta (\xi))),
\end{equation*}
is a well-defined monoid homomorphism.
\end{thm}


We construct the map $\tilde{\zeta}$ in this section in an algebraic method. On the other hand, in \S \ref{section_qbtskein}, we reconstruct the map in $\tilde{\zeta}$ a topological one. In other words,  ``a surgery formula" of some skein algebra explained later defines $\tilde{\zeta}$ in a similar way to the WRT invariant. See, for details, Theorem \ref{thm_qbtskein_main_boundary_link}.


\subsection{An application}
\label{subsection_an_application}


In this subsection, we assume that $\surface = \surface_{g,1}$. By the monoid homomorphism
$\tilde{\zeta}: \mathcal{IC} (\surface_{g,1}) \to \filt{3} \cGL (\surface_{g,1})$, we recover the Johnson homomorphisms of homology cobordisms defined by Garoufalidis and Levine. 



We reformulate the Johnson-like filtration $\filtn{\filt{n} \mathcal{C} (\surface_{g,1})}$ of $\mathcal{C} (\surface_{g,1})$ introduced by Garoufalidis and Levine. The subset $\filt{n} \mathcal{C} (\surface_{g,1})$ consists of all elements $\Xi$ satisfying
\begin{equation*}
(\Xi_*- \id ) (\cGLM (\surface_{g,1})) \subset
\filt{n+1} \cGLM (\surface_{g,1}).
\end{equation*}
 Here we set the linear maps
\begin{align*}
\lambda_{n+2} : &\filt{n+2} \GL (\surface_{g,1})/\filt{n+3} \GL (\surface_{g,1})
\to (H_1 (\surface_{g,1}, \Q))^{\otimes n+2}, \\
&\zettaiti{(\gamma_1-1) \cdots (\gamma_{n+2}-1)} \mapsto c_{n+2} ([\gamma_1] \otimes \cdots
\otimes [\gamma_{n+2}]), \\
c_{n+2}: &(H_1 (\surface_{g,1}, \Q))^{\otimes n+2} \to (H_1 (\surface_{g,1}, \Q))^{\otimes n+2}, \\
&a_1 \otimes \cdots \otimes a_{n+2} \mapsto \sum_{i=1}^{n+2}
a_i \otimes \cdots \otimes a_{n+2} \otimes a_1 \otimes \cdots \otimes a_{i-1}.
\end{align*}
For any $n\in \Zlarger{1}$, the composite
\begin{equation*}
\filt{n}\mathcal{C} (\surface_{g,1}) \tosha{\tilde{\zeta}}
\filt{n+2}\GL (\surface_{g,1}) /\filt{n+3} \GL (\surface_{g,1}) \tosha{\lambda_{n+2}}
c_{n+2} (H_1 (\surface_{g,1}, \Q)^{\otimes n+2})
\end{equation*} 
is the $n$-th Johnson homomorphism introduced by Garoufalidis and Levine. For details, see \cite{GL}. Since we can recover each $\tau_n$ from $\tilde{\zeta}$, we call $\tilde{\zeta}$ the total Johnson homomorphism. 


\section{Review of the Goldman Lie algebra}
\label{section_Goldman}


We denote by $\surface$ a compact connected oriented surface as in \S \ref{subsection_notation_for_the_Goldman}. In this section, we review some facts on the Goldman Lie algebra of $\surface$. In particular, we discuss a relationship between the Goldman Lie algebra and the mapping class group. We use the propositions in this section to prove our main theorem in \S \ref{section_main_theorem_qbtskein}.



Now we recall the Lie bracket and the Lie action introduced by Goldman and Kawazumi-Kuno, respectively. Let $\delta$ be a free loop in $\surface$ and $\gamma$ a path from $*_1 \in \partial \surface$ to $*_2 \in \partial \surface$. We assume $\delta$ and $\gamma$ are in general position. For an intersection $p \in \delta \cap \gamma$, we denote 
\begin{itemize}
\item by $\gamma_{*,p}$ and $\gamma_{p,*}$ a path from $*$ to $p$
and one from $p$ to $*$ along $\gamma$.
\item by $\delta_p$ a based loop along $\delta$ whose basepoint is $p$, and
\item by $\epsilon (p, \delta, \gamma)$ the local intersection number of $\delta$ and $\gamma$
at $p$.
\end{itemize}
We set $\sigma (\delta )(\gamma)$ by
\begin{equation*}
\sigma (\delta) (\gamma) \defeq \sum_{p \in \delta \cap \gamma}
\epsilon (p, \delta, \gamma) \gamma_{*, p} \delta_p \gamma_{p,*}.
\end{equation*}
For $\delta \in \zettaiti{\pi_1} (\surface)$ and $\delta' \in \pi_1 (\surface, *)$, Goldman defined the bracket as $[ \delta, \zettaiti{\delta'}] \defeq \zettaiti{ \sigma (\delta) (\delta')}$. Kawazumi and Kuno proved that the action $\sigma$ makes $\GLM (\surface, *_1, *_2)$ a Lie module of the Goldman Lie algebra.



To discuss the relationship between the Goldman Lie algebra and the mapping class group, Kawazumi, Kuno, Massuyeau, and Turaev used the filtration and the completion of the Goldman Lie algebra defined as the equations in \S \ref{subsection_notation_for_the_Goldman}.
 They obtained a formula for the action of the Dehn twist $t_c$ along a simple closed curve $c$ in terms of the one of the Goldman Lie algebra. 


\begin{thm}[\cite{Kawazumi} \cite{KK} \cite{MT}]
\label{thm_Dehn_formula}

For a simple closed curve $c$, we set an element 
\begin{equation*}
L_{\GL} (c) \defeq \zettaiti{\frac{1}{2} (\log \gamma)^2} \in \cGL (\surface)
\end{equation*}
of the completed Goldman Lie algebra where $\gamma \in \pi_1 (\surface)$ satisfies $\zettaiti{\gamma} =c$. We have 
\begin{equation*}
t_{c*} = \exp (\sigma (L_{\GL} (c))) \defeq \sum_{i=0}^\infty \frac{1}{i!} (\sigma (L_{\GL} (c)))^i:
\cGLM (\surface, *_1, *_2) \to \cGLM (\surface, *_1, *_2)
\end{equation*}
 for any $*_1, *_2 \in \partial \surface$.

\end{thm}


We need the following two propositions Proposition \ref{proposition_center_of_the_Goldman} and
Proposition \ref{prop_zeta} to state our main theorem Theorem \ref{thm_main_hom_Johnson}. We use the first one to prove that the homomorphism $\tilde{\zeta}$ is well-defined. We describe the automorphism induced by a homology cylinder using the second one.



We can prove the following proposition by ``the symplectic expansion". For details, see \cite{Kawazumi} \cite{KK} \cite{MT}. In this paper, we use it without proof.


\begin{prop}
\label{proposition_center_of_the_Goldman}
\begin{enumerate}
\item For $x \in \filt{1} \cGL (\surface)$, if $\sigma (x) (\cGLM (\surface, *_1, *_2)) =\shuugou{0}$
for any $*_1, *_2 \in \partial \surface$, we have $x=0$.
\item We fix a subset $B\subset \partial \surface$ such that the map
$B \to \pi_0 (\partial \surface)$ is bijective.  Let $\xi$ be a bijection $\coprod_{\bullet, * \in B} 
\cGLM (\surface, \bullet, *)\to \coprod_{\bullet, * \in B} 
\cGLM (\surface, \bullet, *)
$ satisfying  the following
\begin{itemize}
\item For any $\bullet, * \in B$, $\xi (\cGLM (\surface, \bullet, *))
=\cGLM (\surface, \bullet, *)$.
\item For any $\star_1, \star_2, \star_3 \in B$ and $x \in \cGLM (\surface, \star_1, \star_2)$,
$x' \in \cGLM (\surface, \star_2, \star_3)$, $\xi (xx') =\xi (x) \xi (x')$.
\item  For any $\bullet, * \in B$,
\begin{align*}
& \xi (I_{\GLM (\surface, \bullet)}^n \cGLM (\surface, \bullet,*))
=I_{\GLM (\surface, \bullet)}^n \cGLM (\surface, \bullet,*), \\
&(\xi-\id )(\cGLM (\surface,\bullet,*)) \subset
I_{\GLM (\surface, \bullet)}^2 \cGLM (\surface, \bullet,\star).
\end{align*}
\item For any $\gamma \in \pi_1 (\surface, \bullet, *)$ represented by
a continuous map $I \to \partial \surface$,
\begin{equation*}
\xi (\gamma)=\gamma
\end{equation*}
\end{itemize}
Then there exists a unique element $\zeta_{\Aut} (\xi) \in \filt{3} \cGL (\surface)$
such that 
\begin{equation*}
\xi =\exp (\zeta_{\Aut} (\xi)):
\cGLM (\surface, \star, \star') \to \cGLM (\surface, \star, \star')
\end{equation*}
for any $\star,\star' \in \partial \surface$.

\end{enumerate}
\end{prop}


We use Theorem \ref{thm_Dehn_formula}  and Proposition \ref{proposition_center_of_the_Goldman} to prove the following.


\begin{prop}
\label{prop_zeta}
The homomorphism $\zeta: \torelli' (\surface) \to (F^3 \cGL (\surface), \bch)$
is defined by 
\begin{align*}
\zeta (t_{c_0}) = L (c_0) \ \mathrm{and} \ 
\zeta (t_{c_1} t_{c_2}^{-1}) = L_{\GL} (c_1)- L_{\GL} (c_2)
\end{align*}
where a simple closed curve $c_0$ and the pair $(c_1,c_2)$ of simple closed curves
bound surfaces. Then the homomorphism is well-defined. Furthermore, $\zeta$ satisfies
the property
\begin{equation*}
\xi_* = \exp (\sigma (\zeta (\xi))):
\cGLM (\surface, *_1, *_2) \to \cGLM (\surface, *_1, *_2)
\end{equation*}
for any $\xi \in \torelli' (\surface)$ and $*_1, *_2 \in \partial \surface$.

\end{prop}

\begin{proof}
Let 
\begin{equation*}
\zeta': \shuugou{t_{c_0}|c_0 \mathrm{ \ bounds \ a \ surface.}} \cup
\shuugou{t_{c_1} t_{c_2}^{-1}| (c_1, c_2) \mathrm{ \ bpounds \ s \ surface.}}
\to \filt{3} \cGL (\surface)
\end{equation*}
be a map defined by 
\begin{equation*}
t_{c_0} \mapsto  L_{\GL} (c_0) \mathrm{ \ and \ } t_{c_1}t_{c_2}^{-1} \mapsto L_{\GL}(c_1) -L_{\GL}(c_2).
\end{equation*}
 We prove the proposition by two steps. We use Theorem \ref{thm_Dehn_formula} and Proposition \ref{proposition_center_of_the_Goldman} in the first and second steps, respectively.



In the first step, we prove that, for any 
\begin{equation*}
\xi_1, \cdots, \xi_j \in \shuugou{t_{c_0}|c_0 \mathrm{ \ bounds \ a \ surface.}} \cup
\shuugou{t_{c_1} t_{c_2}^{-1}| (c_1, c_2) \mathrm{ \ bpounds \ s \ surface.}},
\end{equation*}
\begin{equation*}
(\xi_1 \cdots \xi_j)_* =\exp (\sigma (\bch (\zeta' (\xi_1), \cdots, \zeta' (\xi_j)))):
\cGLM (\surface, *_1, *_2) \to \cGLM (\surface, *_1, *_2)
\end{equation*}
holds. Here we denote $\bch (x_1, \bch (x_2, \cdots, \bch(x_{n-1}, x_n)))$ by
$\bch (x_1, \cdots, x_n)$ for any $x_1, \cdots, x_n \in 
\filt{3} \cGL (\surface)$.
By Theorem \ref{thm_Dehn_formula}, we have 
\begin{align*}
&\exp (\sigma (\bch (\zeta' (\xi_1), \cdots, \zeta' (\xi_j)))) \\ 
&=\exp (\sigma (\zeta' (\xi_1))) \circ \cdots \circ \exp (\sigma (\zeta' (\xi_j))) \\
&=\xi_{1*} \circ \cdots \circ \xi_{j*} =(\xi_1 \cdots \xi_j)_*.
\end{align*}



In the second step, we prove that, for any 
\begin{equation*}
\xi_1, \cdots, \xi_j, \xi'_1, \cdots, \xi'_{j'} \in \shuugou{t_{c_0}|c_0 \mathrm{ \ bounds \ a \ surface.}} \cup
\shuugou{t_{c_1} t_{c_2}^{-1}| (c_1, c_2) \mathrm{ \ bpounds \ s \ surface.}},
\end{equation*}
if 
\begin{equation*}
\xi =\xi_1 \cdots \xi_j = \xi'_1 \cdots \xi'_{j'},
\end{equation*}
\begin{equation*}
\bch (\zeta' (\xi_1), \cdots, \zeta' (\xi_j))=\bch (\zeta' (\xi'_1), \cdots, \zeta'
(\xi'_{j'}))
\end{equation*}
holds. By the first statement, we have
\begin{align*}
&\sigma (\bch (\zeta' (\xi_1), \cdots, \zeta' (\xi_j))
=\sigma (\bch (\zeta' (\xi'_1), \cdots, \zeta'
(\xi'_{j'}))) = \sum_{i=1}^\infty \frac{(-1)^i}{i} (\xi_*-\id )^i \\
&:\cGLM (\surface, *_1, *_2) \to \cGLM (\surface, *_1, *_2)
\end{align*}
for any $*_1, *_2 \in \partial \surface$. Using it, by Proposition \ref{proposition_center_of_the_Goldman}, we obtain
\begin{equation*}
\bch (\zeta' (\xi_1), \cdots, \zeta' (\xi_j))=\bch (\zeta' (\xi'_1), \cdots, \zeta'
(\xi'_{j'})).
\end{equation*}


The above two statements prove the proposition.


\end{proof}

\section{The bracket-quantization \homfly skein algebra}
\label{section_qbtskein}

In this section, we set a skein module $\tskein' (M)$ of a oriented 3-manifold $M$. We call $\tskein'(M)$ the bracket-quantization HOMFLY-PT skein module in this paper. If the 3-manifold is a cylinder $S \times I$ of a compact oriented surface $S$, we consider an algebraic structure of the module and call $\tskein' (S) \defeq \tskein' (S \times I)$ the quantum bracket-quantization HOMFLY-PT skein module.

For a compact oriented 3-manifold $M$, let $\mathcal{E}' (M)$ be a set consisting of all embeddings $E:(\coprod S^1) \to M$ satisfying the condition. 
\begin{itemize}
\item $ E(\coprod (S^1))
\subset M \backslash \partial M.$
\end{itemize}
For two elements $E_1$ and $E_2$ of $\mathcal{E}' (M)$, $E_1$ and $E_2$ are isotopic, if and only if there exists an isotopy $Y:(\coprod S^1) \times I \to M$ having the properties. 
\begin{itemize}
\item For  any $t \in I$, $Y ( \cdot, t) \in \mathcal{E}' (M).$
\item $E_0=Y (\cdot, 0)$, $E_1= Y( \cdot, 1)$.
\end{itemize}
We denote by $\mathcal{T}'(M)$ the set of all tangles, which are isotopy classes of elements of $\mathcal{E}' (M)$.

In this paper, we also use tangles. For two points $\star_0, \star_1 \in \partial \surface$, let $\mathcal{E}' (M,\star_1,\star_2)$ be a set consisting of all embeddings $E:I \sqcup (\coprod S^1) \to M$ satisfying the condition.
\begin{enumerate}
\item $E((0,1) \sqcup \coprod (S^1))
\subset M \backslash \partial M.$
\item $E (\shuugou{0 \in I}) = \shuugou{\star_0}$.
\item $E (\shuugou{1 \in I})=\shuugou{\star_1}$.
\end{enumerate}
 For two elements $E_1$ and $E_2$ of $\mathcal{E}' (M,\star_1,\star_2)$, $E_1$ and $E_2$ are isotopic, if and only if there exists an isotopy $Y: (I \sqcup  \coprod (S^1)) \times I$ having the properties. 
\begin{itemize}
\item For any $t \in I$, $Y ( \cdot, t) \in \mathcal{E}' (M,\star_0,\star_1).$
\item $E_0=Y (\cdot, 0)$, $E_1= Y( \cdot, 1)$.
\end{itemize}
We denote by $\mathcal{T}' (M, \star_0, \star_1)$ the set of all tangles in $M$ with basepoint set $\star_0, \star_1$, which are isotopy classes of $\mathcal{E}' (M, \star_0, \star_1)$.

To define the quantum-bracket skein module, we will consider a Conway triple who are elements of $\mathcal{T}' (M)$ or $\mathcal{T}' (M, \star_0, \star_1)$ and maps
\begin{align*}
&\zettaiti{ \cdot}:
\mathcal{T}' (M, \star_0, \star_1) \to \Zlarger{0},
T \mapsto (\mathrm{the \ number \ of  \ components \ of \ }T), \\
&\zettaiti{ \cdot }:
\mathcal{T}' (M) \to \Zlarger{0},
L \mapsto (\mathrm{the \ number \ of  \ components \ of \ }L). \\
\end{align*}
 Conway triples need to define skein relations. There exist two types of the skein relations of the quantum-bracket skein module in this paper.




\begin{df}

Let $(E_{(C+)}, E_{(C-)}, E_{(C0)}) \in \mathcal{E}'(M, \star_0, \star_1)^{\times 3} 
\mathrm{\ or \ } \mathcal{E}'(M)^{\times 3}$ be embeddings having the two properties.
\begin{itemize}
\item The image of $E_{(C+)}, E_{(C-)}, E_{(C0)}$ are equals except for a closed ball
as oriented submanifolds.
\item In the ball, the images of them are two lines presented by Figure\ref{fig_conway_plus_unframed},
Figure\ref{fig_conway_minus_unframed},
and Figure\ref{fig_conway_zero_unframed}.
\end{itemize}
Then the three elements 
$(T_{(C+)}, T_{(C-)}, T_{(C0)}) \in \mathcal{T}' (M, \star_0, \star_1)^{\times 3}
\mathrm{\ or \ } \mathcal{T}'(M)^{\times 3}$
represented by $(E_{(C+)}, E_{(C-)}, E_{(C0)})$ is a Conway triple.


\end{df}

	\begin{figure}[htbp]
	\begin{tabular}{ccc}
	\begin{minipage}{0.3\hsize}
\begin{center}
\begin{picture}(20,20)
\put(0,0){
{\unitlength 0.1in%
\begin{picture}(3.2000,3.2400)(1.1800,-4.4000)%
%
\special{pn 8}%
\special{ar 278 276 160 160 0.0000000 6.2831853}%
%
\special{pn 8}%
\special{pa 160 280}%
\special{pa 200 240}%
\special{fp}%
\special{pa 160 280}%
\special{pa 200 320}%
\special{fp}%
\special{pa 280 160}%
\special{pa 240 200}%
\special{fp}%
\special{pa 280 160}%
\special{pa 320 200}%
\special{fp}%
%
\special{pn 8}%
\special{pa 278 116}%
\special{pa 278 116}%
\special{fp}%
\special{pa 278 120}%
\special{pa 278 440}%
\special{fp}%
\special{pa 438 280}%
\special{pa 318 280}%
\special{fp}%
\special{pa 238 280}%
\special{pa 118 280}%
\special{fp}%
\end{picture}}%
}
\end{picture}
\end{center}
\caption{$T_{(C+)}$}
\label{fig_conway_plus_unframed}
	\end{minipage}
	\begin{minipage}{0.3\hsize}
\begin{center}
\begin{picture}(20,20)
\put(0,0){
{\unitlength 0.1in%
\begin{picture}(3.2000,3.2400)(1.1800,-4.4000)%
%
\special{pn 8}%
\special{ar 278 276 160 160 0.0000000 6.2831853}%
%
\special{pn 8}%
\special{pa 160 280}%
\special{pa 200 240}%
\special{fp}%
\special{pa 160 280}%
\special{pa 200 320}%
\special{fp}%
\special{pa 280 160}%
\special{pa 240 200}%
\special{fp}%
\special{pa 280 160}%
\special{pa 320 200}%
\special{fp}%
%
\special{pn 8}%
\special{pa 278 120}%
\special{pa 278 120}%
\special{fp}%
\special{pa 278 120}%
\special{pa 278 240}%
\special{fp}%
\special{pa 278 320}%
\special{pa 278 440}%
\special{fp}%
\special{pa 438 280}%
\special{pa 118 280}%
\special{fp}%
\end{picture}}%
}
\end{picture}
\end{center}
\caption{$T_{(C-)}$}
\label{fig_conway_minus_unframed}
	\end{minipage}
	\begin{minipage}{0.3\hsize}
\begin{center}
\begin{picture}(20,20)
\put(0,0){
{\unitlength 0.1in%
\begin{picture}(3.2200,3.2400)(1.1800,-4.4000)%
%
\special{pn 8}%
\special{ar 278 276 160 160 0.0000000 6.2831853}%
%
\special{pn 8}%
\special{ar 200 360 80 80 4.7123890 6.2831853}%
%
\special{pn 8}%
\special{ar 360 200 80 80 1.5707963 3.1415927}%
%
\special{pn 8}%
\special{pa 280 200}%
\special{pa 280 200}%
\special{fp}%
\special{pa 280 120}%
\special{pa 280 200}%
\special{fp}%
\special{pa 360 280}%
\special{pa 440 280}%
\special{fp}%
\special{pa 280 360}%
\special{pa 280 440}%
\special{fp}%
\special{pa 200 280}%
\special{pa 120 280}%
\special{fp}%
%
\special{pn 8}%
\special{pa 160 280}%
\special{pa 200 240}%
\special{fp}%
\special{pa 160 280}%
\special{pa 200 320}%
\special{fp}%
\special{pa 280 160}%
\special{pa 240 200}%
\special{fp}%
\special{pa 280 160}%
\special{pa 320 200}%
\special{fp}%
\end{picture}}%
}
\end{picture}
\end{center}
\caption{$T_{(C0)}$}
\label{fig_conway_zero_unframed}
	\end{minipage}
	\end{tabular}
\end{figure}

\begin{df}
%

For a compact oriented 3-manifold $M$ and $\star_0, \star_1 \in \partial M$, we set 
$\tskein' (M)$ and $\tskein' (M, \star_0, \star_1)$ as the quotients of the free $\Q [h]$-modules $\Q [h] \mathcal{T}' (M)$ and $\Q [h]\mathcal{T}' (M, \star_0, \star_1)$ by the relations
\begin{itemize}
\item For a Conway triple $(T_{(C+)}, T_{(C-)}, T_{(C0)}) \in \mathcal{T}' (M)^{ \times 3}
\mathrm{ \ or \ } \mathcal{T}' (M, \star_0, \star_1)^{\times 3}$
such that
\begin{equation*}
\zettaiti{T_{(C+)}}=\zettaiti{T_{(C-)}}=\zettaiti{T_{C0}}-1,
\end{equation*}
we consider the relation
\begin{equation*}
T_{(C+)}-T_{(C-)}=h T_{(C0)}.
\end{equation*}
\item  For a Conway triple $(T_{(C+)}, T_{(C-)}, T_{(C0)}) \in \mathcal{T}' (M)^{ \times 3}
\mathrm{ \ or \ } \mathcal{T}' (M, \star_0, \star_1)^{\times 3}$
such that
\begin{equation*}
\zettaiti{T_{(C+)}}=\zettaiti{T_{(C-)}}=\zettaiti{T_{C0}}+1,
\end{equation*}
we consider the relation
\begin{equation*}
T_{(C+)}-T_{(C-)}=0.
\end{equation*}
\end{itemize}
 In other words, the figure visualizes relations defining the skein modules
$\tskein' (M)$ and $\tskein' (M, \star_0, \star_1)$.

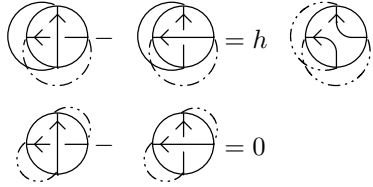
\begin{center}
\input{fig_tskein_prime_easy_relation}
\end{center}

\end{df}

Let $S$ be a compact oriented surface. We denote by $\tskein' (S)$ and $\tskein' (S, \star_0, \star_1)$ the skein modules $\tskein( (\surface \times I)$ and $\tskein' (\surface \times I, (\star_0,1), (\star_1, 0))$ for $\star_0, \star_1 \in \partial \surface$. For a tangle $T \in \mathcal{T}' (S \times I, (\star_0,1), (\star_1,0))$ represented by $E_\gamma \sqcup E:I \sqcup  \coprod (S^1) \to S \times I$ and $t_0,t_1 \in I$, we set $E_{t_0 \cdot \gamma \cdot t_1}$ as 
\begin{equation*}
E_{t_0 \cdot \gamma \cdot t_1}(t) 
\begin{cases}
(\star_0,t_0(1-3t)+3t) \mathrm{ \ if \ } t \in [0,\frac{1}{3}] \\
E_{\gamma} (3t-1) \mathrm{ \ if \ } t \in [\frac{1}{3}, \frac{2}{3}] \\
(\star_1, t_1(3t-2) \mathrm{\ if \ } t \in [\frac{2}{3}, 1]
\end{cases}
\end{equation*}
 and $\diamondsuit_{(1,0)}^{(t_0,t_1)}(T)$ as an element of $\mathcal{T}' (S \times I,(\star_0,t_0),(\star_1, t_1))$ represented by one of $\mathcal{E}' (S \times I,(\star_0,t_0),(\star_1, t_1))$ isotopic to the embedding $E_{t_0, \gamma, t_1} \sqcup E$. The map induces a bijection $\diamondsuit_{(1,0)}^{(t_0,t_1)}: \tskein (S, \star_0, \star_1) \to \tskein (S\times I, (\star_0, t_0), (\star_1, t_1))$, whose reverse map $\diamondsuit_{(t_0,t_1)}^{(1,0)}$ denotes. We will define a multiple of $\tskein' (S)$ and right and left actions of $\tskein' (S)$ on $\tskein' (S, \star_0,\star_1)$. We set two embeddings $e_\mathrm{over},e_\mathrm{under}
: S \times I \to S \times I$ as 
\begin{equation*}
e_\mathrm{over} (p,t)=(p,\frac{1+t}{2}), \
e_\mathrm{under} (p,t)=(p,\frac{t}{2}).
\end{equation*}
Let $L_1$ and $L_2$ be two links in $S\times I$ and $T$ be a tangle represented by $E_1$, $E_2$, and $E$. We set the multiple $L_1 L_2$ of $L_1$ and $L_2$, the right and left actions $L_1 T,TL_1$ of $L_1$ on $T$ as the following. The multiple $L_1 L_2$ is a link represented by the embedding $e_\mathrm{over} \circ E_1 \sqcup e_\mathrm{under} \circ E_2$. The right and left actions $L_1 T, T L_1$ are the tangles as
\begin{equation*}
L_1 T \defeq \diamondsuit_{(\frac{1}{2},0)}^{(1,0)} ((L_1 T)'), \ 
T L_1 \defeq \diamondsuit_{(1, \frac{1}{2})}^{(1,0)}((TL_1)')
\end{equation*}
 where the embeddings $e_\mathrm{over} \circ E_1 \sqcup e_\mathrm{under} \circ
E$ and $e_\mathrm{over} \circ E \sqcup e_\mathrm{under} \circ
E_1$ represent $(L_1T)'$ and $(TL_1)'$.

Furthermore, we consider a Lie bracket $[ \cdot, \cdot]:\tskein' (\surface) \times \tskein' (\surface) \to \tskein' (\surface)$ and a Lie action $\sigma ( \cdot )(\cdot ): \tskein' (\surface) \times \tskein' (\surface, \star_0, \star_1)$ as the following. By Theorem \ref{thm_psi_Uh_tskein'}, the maps
\begin{align*}
&h (\cdot):\tskein' (\surface) \to \tskein' (\surface),x \mapsto h x \\
&h (\cdot):\tskein' (\surface, \star_1, \star_2) \to \tskein' (\surface
, \star_1, \star_2),y \mapsto h y \\
\end{align*}
 are injective, and by the skein relation, we have 
\begin{align*}
&x_1x_2-x_2x_1 \in h \tskein' (\surface), \\
&x_1 y-yx_1 \in h \tskein' (\surface, \star_1, \star_2). \\
\end{align*}
for $x_1,x_2 \in \tskein' (\surface)$ and $y \in \tskein' (\surface ,\star_1, \star_2).$
  We can define 
\begin{align*}
&[x_1, x_2] \defeq \frac{1}{h} (x_1 x_2-x_2x_1), \\
&\sigma (x_1)(y) \defeq \frac{1}{h} (x_1 y-yx_1). \\
\end{align*}

Let $\surface$ be a compact connected oriented surface. We set $\Uh (\surface)$ as the quotient of the tensor algebra $\oplus_{i=0}^\infty (\Q [h] 
\zettaiti{\pi_1} (\surface))^{\otimes_{\Q [h]} i}$ by the relation
\begin{itemize}
\item For any $x, y \in \GL (\surface)$,
\begin{equation*}
x \otimes y - y \otimes x =h [x,y].
\end{equation*}
\end{itemize}
We define the Lie bracket $[ \cdot, \cdot]:\Uh (\surface) \times \Uh (\surface) \to \Uh (\surface)$ by the Leibniz rule, which means that we have
\begin{align*}
[x_1\otimes \cdots \otimes x_i, y_1 \otimes \cdots \otimes y_j]
&= \\
\sum_{i' \in \shuugou{1, \cdots, i}, j' \in \shuugou{1, \cdots, j}}
&x_1 \otimes \cdots \otimes x_{i'-1} \otimes 
y_1 \otimes \cdots \otimes y_{j'-1}  \\
&\otimes[x_{i'}, y_{j'}] \otimes y_{j'+1} \otimes \cdots \otimes y_j \otimes
x_{i'+1} \otimes \cdots \otimes x_i
\end{align*}
for $x_1 \otimes \cdots \otimes x_i \in (\Q [h] \zettaiti{\pi_1} (\surface))^{\otimes i}$
and
$y_1 \otimes \cdots \otimes y_j \in (\Q
 [h] \zettaiti{\pi_1} (\surface))^{\otimes j}$. 
It satisfies the formula
\begin{equation*}
[x,y]=\frac{1}{h} (x \otimes y -y \otimes x)
\end{equation*}
for $x$ and $y \in \Uh (\surface)$. 
We denote by $\PPsi{\GL}{\Uh}$ the natural injection $\GL (\surface) \to \Uh (\surface)$
and by $\PPsi{\Uh}{\GL}$ the natural surjection
\begin{equation*}
\Uh (\surface) \to \GL (\surface), x \mapsto
\begin{cases}
x \mathrm{ \ if \ } x \in \GL (\surface)\\
0 \mathrm{ \ if \ } x \in \GL (\surface)^{\otimes k}, k \neq 1.
\end{cases}
\end{equation*}

We also set $\Uh (\surface, \star_1, \star_2)$ as the quotient of the tensor module 
$\Uh (\surface) \otimes_{\Q [ h]}\Q [h] \pi_1 (\surface, \star_1, \star_2)\otimes_{\Q [h]}\Uh (\surface)$
by the relation
\begin{itemize}
\item For $x_1, \cdots,  x_i \in \Q [h] \zettaiti{\pi_1 } (\surface)$
and $y \in \Q[h] \pi_1 (\surface, \star_1, \star_2)$,
\begin{align*}
&x_1 \otimes  \cdots \otimes x_{i'} \otimes y \otimes 
x_{i'+1} \otimes \cdots \otimes x_{i} \\
&-
x_1 \otimes  \cdots \otimes x_{i'-1} \otimes y \otimes 
x_{i'} \otimes \cdots \otimes x_{i} \\
&=x_1 \otimes  \cdots \otimes x_{i'-1} \otimes \sigma (x_{i'})(y) \otimes 
x_{i'+1} \otimes \cdots \otimes x_{i}
\end{align*}
\end{itemize}
 We define the Lie action $\sigma ( \cdot) ( \cdot ): \Uh (\surface) \to \Uh (\surface, \star_1, \star_2)$ by the Leibniz rule, which means that we have 
\begin{align*}
&\sigma (x_1 \otimes \cdots \otimes x_i) (x' \otimes y \otimes x'')
=[x_1 \otimes \cdots \otimes x_i, x']\otimes y \otimes x'' \\
&+\sum_{i' \in \shuugou{1, \cdots, i}} x' \otimes x_1 \otimes  \cdots \otimes
x_{i-1} \otimes \sigma (x_i) (y)
\otimes x_{i'+1} \otimes \cdots \otimes x_i \otimes x \\
&+x' \otimes y \otimes [x_1 \otimes \cdots \otimes x_i, x'']
\end{align*}
 for $x_1, \cdots, x_i \in \GL (\surface)$, $x,x' \in \Uh (\surface)$, and $y \in \GLM (\surface)$. It also satisfies 
\begin{equation*}
\sigma (x)(y)=\frac{1}{h} (x \otimes y-y \otimes x).
\end{equation*}
 We denote by $\PPsi{\GLM}{\Uh}$ the natural injection 
$\GLM (\surface, \star_1, \star_2) \to \Uh (\surface, \star_1, \star_2)$
and by $\PPsi{\Uh}{\GLM}$ the natural surjection
\begin{align*}
&\Uh (\surface, \star_1, \star_2) \to \GLM (\surface, \star_1, \star_2), \\
&x \otimes y \otimes x' \mapsto
\begin{cases}
(x)(x')y \mathrm{ \ if \ } x,x' \in \GL (\surface)^{\otimes 0} =\Q \\
0 \mathrm{ \ if \ } y \in \GL (\surface)^{\otimes k},y' \in \GL (\surface)^{\otimes k'}, k+k' \geq 1
\end{cases}
\end{align*}
for $y \in\GLM (\surface, \star_1, \star_2)$.



We set the $\Q [h]$-algebra homomorphism $ \PPsi{\Uh}{\tskein'}$ 
and the $\Q [h]$-module one $\PPsi{\Uh}{\tskein}$ by
\begin{align*}
&\PPsi{\Uh}{\tskein'}: \Uh (\surface ) \to \tskein' (\surface),
\zettaiti{\gamma}\in \zettaiti{\pi_1} (\surface) \mapsto K_{\zettaiti{\gamma}}, \\
&\PPsi{\Uh}{\tskein'}:\Uh (\surface, \star_1, \star_2) \to
\tskein' (\surface, \star_1, \star_2),
x \otimes \gamma \otimes x' \mapsto \PPsi{\Uh}{\tskein'} (x)T_\gamma \PPsi{\Uh}{\tskein'} (x)
\end{align*}
for $\gamma \in \pi_1 (\surface)$ and $x,x' \in \Uh (\surface)$,
where $K_{\zettaiti{\gamma}}$ and $T_\gamma$ are a knot and a a tangle whose homotopy class
is $\zettaiti{\gamma}$ and $\gamma$, respectively.
We can prove the following theorem in the same way as the proof of Turaev
\cite[Theorem 3.3]{Turaev}.


\begin{thm}
\label{thm_psi_Uh_tskein'}
For a compact connected oriented surface $\surface$ and $\star_1, \star_2 \in \partial \surface$, the algebra homomorphism 
$\PPsi{\Uh}{\tskein'}: \Uh (\surface ) \to \tskein' (\surface)$
and module one
$\PPsi{\Uh}{\tskein'}:\Uh (\surface, \star_1, \star_2) \to
\tskein' (\surface, \star_1, \star_2)$
 are bijective. Furthermore, we have 
\begin{align*}
&[\PPsi{\Uh}{\tskein'} (x), \PPsi{\Uh}{\tskein'}(x')]
=\PPsi{\Uh}{\tskein'} ([x,x']), \\
&\sigma (\PPsi{\Uh}{\tskein'} (x))(\PPsi{\Uh}{\tskein'} (y))
=\PPsi{\Uh}{\tskein'} (\sigma (x)(y)) \\
\end{align*}
 for $x, x' \in \Uh (\surface)$ and $y \in \Uh (\surface, \star_1, \star_2)$.

\end{thm}


We denote by $\PPsi{\tskein'}{\Uh}$ the reverse map of $\PPsi{\Uh}{\tskein'}$, and consider the composites
\begin{align*}
&\PPsi{\GLM}{\tskein'} \defeq \PPsi{\GLM}{\Uh} \circ \PPsi{\Uh}{\tskein'}, \ 
\PPsi{\GL}{\tskein'} \defeq \PPsi{\GL}{\Uh} \circ \PPsi{\Uh}{\tskein'}, \\
&\PPsi{\tskein'}{\GLM} \defeq \PPsi{\tskein'}{\Uh} \circ \PPsi{\Uh}{\GLM}, \ 
\PPsi{\tskein'}{\GL} \defeq \PPsi{\tskein'}{\Uh} \circ \PPsi{\Uh}{\GL}. \\
\end{align*}

Let $\filtn{\filt{n} \Uh (\surface)}$ and $\filtn{\filt{n} \Uh (\surface, \star_1, \star_2)}$ be the filtrations of $\Uh (\surface)$ and $\Uh (\surface, \star_1, \star_2)$ such as 
\begin{align*}
&\filt{n} \Uh (\surface) 
\defeq \sum_{2i_0+i_1 +\cdots +i_j \geq n}
h^{i_0} (\filt{i_1} \GL (\surface)) \otimes \cdots \otimes 
(\filt{i_j} \GL (\surface)), \\
&\filt{n} \Uh (\surface, \star_1, \star_2) \defeq
\sum_{i_1+i_2+i_3 \geq n}
\filt{i_1}\Uh (\surface) \otimes
(\filt{i_2} \GLM (\surface, \star_1, \star_2))
\otimes \filt{i_3} \Uh (\surface). \\
\end{align*}
 We consider the completed algebra and module 
\begin{align*}
&\cUh (\surface) \defeq \comp{i} \Uh (\surface )/
\filt{i} \Uh (\surface), \\
&\cUh (\surface, \star_1, \star_2) \defeq
\comp{i} \Uh (\surface, \star_1, \star_2)/
\filt{i} \Uh (\surface, \star_1, \star_2)
\end{align*}
of $\Uh (\surface)$ and $\Uh (\surface, \star_1, \star_2)$ in terms of the filtrations. Furthermore, they also define
$\filtn{\filt{n} \tskein' (\surface)}$
 and $\filtn{\filt{n} \tskein' (\surface, \star_1, \star_2)}$
of $\tskein' (\surface)$ and $\tskein' (\surface, \star_1, \star_2)$ such as 
\begin{align*}
&\filt{n} \tskein' (\surface) 
\defeq \PPsi{\Uh}{\tskein'} (\filt{n} \Uh (\surface)), \\
&\filt{n} \Uh (\surface, \star_1, \star_2) \defeq
\PPsi{\Uh}{\tskein'} (\filt{n} (\Uh (\surface , \star_1, \star_2)). \\
\end{align*}
We also consider the completed algebra and module
\begin{align*}
&\ctskein' (\surface) \defeq \comp{i} \tskein' (\surface )/
\filt{i} \tskein' (\surface), \\
&\ctskein' (\surface, \star_1, \star_2) \defeq
\comp{i} \tskein' (\surface, \star_1, \star_2)/
\filt{i} \tskein' (\surface, \star_1, \star_2),
\end{align*}
 of $\tskein' (\surface)$ and $\tskein' (\surface, \star_1, \star_2)$. For a compact oriented surface $S$ having components $\surface^{(1)}, \cdots, \surface^{(k)}$, using the isomorphism
\begin{equation*}
e_{\coprod_{i} \surface^{(i)}, S*}:
\tskein' (\surface^{(1)}) \otimes \cdots \otimes \tskein' (\surface^{(k)})
\to
\tskein' (S),
\end{equation*}
 we define the filtration 
$\filtn{\filt{n} \tskein' (S)}$
as 
\begin{equation*}
\filt{n} \tskein' (S) \defeq \sum_{i_1+ \cdots+i_k \geq n}
e_{\coprod_{i} \surface^{(i)}, S*}(\filt{i_1}\tskein' (\surface^{(1)}) \otimes \cdots \otimes \filt{i_k}
(\surface^{(i_k)})).
\end{equation*}
We also set the filtrations $\filtn{\filt{n} \widehat{V}}$ of completions by
\begin{equation*}
\filt{n} \widehat{V} \defeq \ker  (\widehat{V} \to V/\filt{n} V)
\end{equation*}
for $V=\Uh (\surface),\Uh (\surface, \star_1, \star_2), \tskein' (\surface),
\tskein' (\surface, \star_1, \star_2)$ and $\tskein' (S)$.
 The filtration has the following property.


\begin{prop}

For compact connected surfaces $S$ and $S'$, the induced map 
$e_*:\tskein' (S') \to \tskein' (S)$
of an embedding 
$e :S' \times I \to S \times I$
satisfies 
\begin{equation*}
e_* (\filt{N} \tskein' (S')) \subset
\filt{N} \tskein' (S)
\end{equation*}
 for any $N \in \Zlarger{0}.$

\end{prop}

\begin{proof}

We can prove the statement in the same way as our paper \cite[\S 4]{TsujiHOMFLY-PTskein}. Furthermore, it is easy to prove it by Lemma \ref{thm_qbtskein_filtration_assumption} in this paper in case the sequence is trivial, which means $\mathbf{b}=\shuugou{0,0,0, \cdots }$.

\end{proof}

\section{Homology cylinders}
\label{section_homology_cylinders}


Let $\surface$ be a compact connected oriented surface with boundary. We define a homology cobordism or cylinder not only in the case that $\surface$ has one boundary component but in the general case.


\begin{df}

Let $M$ be a compact connected oriented 3-manifold and
 $\alpha:\partial (\surface \times I) \to
\partial M $ 
a diffeomorphism. We call the pair $(M, \alpha )$ a homology cylinder, if and only if the embeddings 
\begin{align*}
&\alpha_{0}: \surface \to M, p \mapsto \alpha (p,0), \ 
\alpha_{1}: \surface \to M, p \mapsto \alpha (p,1)
\end{align*}
induce isomorphisms 
$\alpha_{0*}: H_ 1(\surface, \Z) \to H_1 (M , \Z)$ and
$\alpha_{1*}: H_1 (\surface, \Z) \to H_1 (M, \Z)$
in homology groups over the integers.


\end{df}


Let $(M^1,\alpha^1)$ and $(M^2, \alpha^2)$ be two homology cobordisms. If there exists a diffeomorphism $\chi: M^1 \to M^2$ satisfying $\chi \circ \alpha^1= \alpha^2$, such pairs are called diffeomorphic. We denote by $\mathcal{C}(\surface)$ the set of diffeomorphic classes of homology cobordisms of the surface $\surface$.



We consider the stacking sum of homology cobordisms defined as follows, which makes $\mathcal{C} (\surface)$ a monoid. For two homology cobordisms $(M^1, \alpha^1)$ and $(M^2, \alpha^2)$, we set a $3$-manifold $M^1 \circ_{\alpha^1, \alpha^2} M^2$ and a diffeomorphism
$\alpha^1 \sqcup \alpha^2: \partial (\surface \times I) \to \partial (M^1 \circ_{\alpha^1, \alpha^2} M^2)$
as the following. The $3$-manifold $M^1 \circ_{\alpha^1, \alpha^2} M^2$ is the quotient of 
$M^1 \sqcup M^2$
 by the relation
\begin{equation*}
\alpha^1_0 (p) \sim \alpha^2_0 (p) \mathrm{ \ for \ }
p \in \surface.
\end{equation*}
 We define the diffeomorphism $\alpha^1 \sqcup \alpha^2$ as 
\begin{equation*}
(\alpha^1 \sqcup \alpha^2)(p,t)=
\begin{cases}
\alpha^1 (p,1) \ & \\
\alpha^1 (p, 2t-1) \ & \mathrm{if} \ t \in [\frac{1}{2},1] \\
\alpha^2 (p, 2t) \ & \mathrm{if} \ t \in [0, \frac{1}{2}] \\
\alpha^2 (p,0). \ &
\end{cases}
\end{equation*}
 Then the pair $(M^1 \circ_{\alpha^1, \alpha^2} M^2, \alpha^1 \sqcup \alpha^2)$ is also a homology cobordism. The operator
\begin{equation*}
( \cdot ) \circ (\cdot):\mathcal{C} (\surface) \times \mathcal{C} (\surface)
\to \mathcal{C} (\surface),
((M^1, \alpha^1),(M^2, \alpha^2)) \mapsto
(M^1 \circ_{\alpha^1, \alpha^2} M^2, \alpha^1 \sqcup \alpha^2)
\end{equation*}
 makes $\mathcal{C} (\surface)$ a monoid, and we call it the stacking sum.

We will consider the action 
\begin{equation*}
\Phi: \mathcal{C} (\surface) \to
\Aut (\cGLM (\surface, \bullet, \star))
\end{equation*}
of $\mathcal{C}(\surface)$ on $\GLM (\surface, \star_1, \star_2)$
 as follows. We set isomorphisms 
\begin{align*}
&\diamondsuit_{0}^{1} :\pi_1 (M, \alpha (\bullet,0), \alpha (*,0)) \to 
\pi_1 (M, \alpha (\bullet ,1), \alpha(* ,1)),
\gamma \mapsto \gamma_{10\bullet} \gamma \gamma_{01*}, \\
&\diamondsuit_{1}^{0} :\pi_1 (M, \alpha (\bullet ,1), \alpha(* ,1)) \to 
\pi_1 (M, \alpha (\bullet,0), \alpha (*,0)),
\gamma \mapsto \gamma_{01\bullet} \gamma \gamma_{10*}, \\
\end{align*}
where the continuous maps 
\begin{align*}
&[0,1] \to M, t \mapsto \alpha (\star ,t), \ \  
[0,1] \to M, t \mapsto \alpha (\star, 1-t), \\
\end{align*}
represent the paths
$\gamma_{01 \star} \in \pi_1 (M, \alpha (\star ,0), \alpha(\star ,1)),$
$\gamma_{10 \star} \in \pi_1 (M, \alpha (\star ,1), \alpha(\star ,0))$
, respectively, for any $\star \in \partial \surface$. By Stallings's theorem \cite{Stallings}, the embeddings $\alpha_0$ and $\alpha_1$ induced isomorphisms
\begin{align*}
&\alpha_{0*}:\cGLM (\surface, \bullet , *)
\to \cGLM (M,\pi_1 (M, \alpha (\bullet,0), \alpha (*,0)), \\
&\alpha_{1*}:\cGLM (\surface, \bullet , *)
\to \cGLM (M, \alpha (\bullet ,1), \alpha(* ,1)). \\
\end{align*}
 Then the action 
$\Phi ((M, \alpha))=(M, \alpha)_* \in \Aut (\cGLM (\surface, \bullet, *))$
is defined as the composite map
\begin{equation*}
\Phi ((M, \alpha))=(M, \alpha)_* \defeq \alpha_{1 *}^{-1} \circ
\diamondsuit_{0}^{1} \circ \alpha_{0*}.
\end{equation*}
 The map 
\begin{equation*}
\Phi=
( \cdot )_*: \mathcal{C} (\surface)
\to \Aut (\cGLM (\surface, \bullet, *))
\end{equation*}
 is a monoid homomorphism.



To make the definition of homology cylinders, we will check that the two conditions in the following proposition are equivalent.


\begin{prop}
\label{prop_jouken_homology_cylinder}

For a homology cobordism $(M,\alpha)$, the following two conditions are equivalent.
\begin{enumerate}
\item We have 
\begin{equation*}
\ker ( \iota_*: H_1 (\partial (\surface \times I) , \Z) \to
H_1 (\surface \times I, \Z))
= \ker (\alpha_*:H_1 (\partial (\surface \times I), \Z) \to
H_1 (M,\Z)),
\end{equation*}
 where the natural embedding $\iota: \partial (\surface \times I) \to \surface \times I$
and $\alpha:\partial (\surface \times I) \to M$ induce the group homomorphisms
\begin{align*}
&\iota_*: H_1 (\partial (\surface \times I) , \Z) \to
H_1 (\surface \times I, \Z), \ 
\alpha_*:H_1 (\partial (\surface \times I), \Z) \to
H_1 (M,\Z).
\end{align*}
\item For any $\star_1, \star_2 \in \partial \surface$, the action $(M,\alpha)_*$
satisfies
\begin{equation*}
((M, \alpha)_*- \id_{\cGLM (\surface, \bullet, *)} )
 ( \cGLM (\surface , \star_1, \star_2)) \subset
\filt{2}
\cGLM (\surface , \star_1, \star_2).
\end{equation*}

\end{enumerate}


\end{prop}


\begin{proof}


To prove it, we need some notation. We choose some base points $\star, \star_1,\star_2 \in \partial \surface$. Let 
\begin{align*}
&\gamma^{\partial (\surface \times I)}_{01\star} \in \pi_1 (\partial (\surface \times I), (\star, 0), (\star,1)), 
&\gamma^{\partial (\surface \times I)}_{10\star} \in \pi_1 (\partial (\surface \times I), (\star, 1), (\star,0)), \\ 
\end{align*}
be the paths represented by the continuous maps
\begin{align*}
&I \to \partial (\surface \times I), t \mapsto (\star, t), \ 
&I \to \partial (\surface \times I), t \mapsto (\star, 1-t). 
\end{align*}
 We set the embeddings
$\iota'_0, \iota'_1 :\surface \to \partial (\surface \times I)$
 as 
\begin{align*}
\iota'_0: \surface \to \partial (\surface \times I), p \mapsto (p,0), \ 
\iota'_1: \surface \to \partial (\surface \times I), p \mapsto (p,1). 
\end{align*} 
They induce 
\begin{align*}
&\iota'_{0*}: \pi_1 (\surface, \star_1, \star_2)
\to \pi_1 (\partial (\surface \times I), (\star_1, 0), (\star_2, 0)),  \\
&\iota'_{1*}: \pi_1 (\surface, \star_1, \star_2)
\to \pi_1 (\partial (\surface \times I), (\star_1, 1), (\star_2, 1)).  \\
\end{align*}

First, we prove $(1)\Rightarrow (2)$. Let $\gamma $ be an element of  $\pi_1 (\surface, \star_1, \star_2)$, and $\bar{\gamma}$ the reverse of $\gamma$. Then, the homology class represented by the path 
\begin{equation*}
\iota'_{0*} (\gamma) \gamma^{\partial (\surface \times I)}_{01\star_2}
\iota'_{1, *} (\bar{\gamma}) \gamma^{\partial (\surface \times I)}_{10 \star_1}
\end{equation*}
satisfies 
\begin{equation*}
\iota_*([\iota'_{0*} (\gamma) \gamma^{\partial (\surface \times I)}_{01\star_2}
\iota'_{1, *} (\bar{\gamma}) \gamma^{\partial (\surface \times I)}_{10 \star_1}])
=0
\in H_1 (\partial (\surface \times I), \Z).
\end{equation*}
 By the first statement, we have 
\begin{align*}
&\alpha_*([\iota'_{0*} (\gamma) \gamma^{\partial (\surface \times I)}_{01\star_2}
\iota'_{1, *} (\bar{\gamma}) \gamma^{\partial (\surface \times I)}_{10 \star_1}])
=0
\in H_1 (M, \Z), \\
&\alpha_{0*} (\gamma) \gamma_{01\star_2} \alpha_{1*} (\bar{\gamma}) \gamma_{10\star_1}
\in [ \pi_1 (M , \alpha (\star_1,0)), \pi_1 (M, \alpha (\star_1, 0))]
\subset 1+\filt{2} \GLM (M, \alpha (\star_1,0)).
\end{align*}
 So we have
\begin{align*}
&\alpha_{0*} (\gamma)-\gamma_{01\star_1} \alpha_{1*} ({\gamma}) \gamma_{10\star_2} \\
&=(\alpha_{0*} (\gamma) \gamma_{01\star_2} \alpha_{1*} (\bar{\gamma}) \gamma_{10\star_1}-1)\gamma_{01\star_1} \alpha_{1*} ({\gamma}) \gamma_{10\star_2}\\
&\in \filt{2} \GLM (M , \alpha (\star_1,0), \alpha (\star_2,0))
\end{align*}
 Hence we obtain $(1)\Rightarrow (2)$.



Next, we prove $(2) \Rightarrow (1)$. We will prove it in five steps. 
\begin{itemize}
\item[(Step 1)] The homology group $H_1 (\partial (\surface \times I), \Z )$ is the direct sum $V \oplus \iota'_{1*} (H_1 (\surface, \Z))$, where, for a path  $\gamma\in \pi_1 (\surface,\star_1, \star_2)$ and the reverse path $\bar{\gamma}$, the homology classes of
\begin{equation*}
\iota'_0 (\gamma) \gamma_{01 \star_2}^{\partial (\surface \times I)} \iota'_1 (\bar{\gamma})
\gamma_{10 \star_1}^{\partial (\surface \times I)}
\end{equation*}
 generate $V$.
\item[(Step 2)] We have $\ker \iota_* =V.$
\item[(Step 3)] The composite $\alpha_* \circ \iota'_{1*}:H_1 (\surface ,\Z)\to H_1 (M, \Z )$ is an isomorphism.
\item[(Step 4)] We have $\alpha_* (V)=\shuugou{0}.$
\item[(Step 5)] We  have $\ker \alpha_* =V$.
\end{itemize}



(Step 1) The basis in the figure generates the homology group $H_1 (\partial (\surface \times I), \Z)$ as a free $\Z$-module. Some set of it generates $\iota'_{1*} (H_1(\surface, \Z))$, and one of the others $V$. This proves (Step 1).

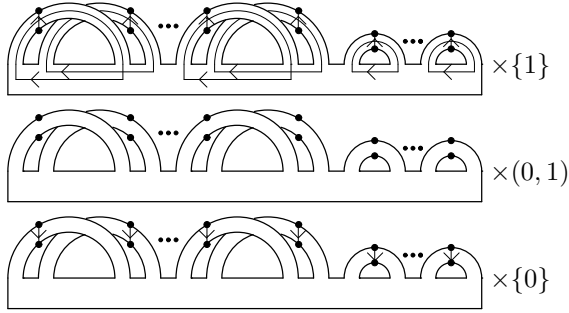
\begin{center}
\input{fig_partial_surface_times_I_homology_basis}
\end{center}



(Step 2) Using the above basis, we have $\iota_* (V) =\shuugou{0}$ and
$\iota_* \circ \iota'_{1*}=\id_{H_1 (\surface, \Z)} :H_1 (\surface, \Z) \to H_1 (\surface, \Z)$. So we obtain (Step 2).



(Step 3) By definition of homology cobordisms, the homomorphism 
$\alpha_{1*} : H_1 (\surface, \Z) \to H_1 (M, \Z)$
is an isomorphism. Since $\alpha_{1} = \alpha \circ \iota'_1$, we have (Step 3).



(Step 4)
Here we use (2). By it, for a path 
$\gamma \in \pi_1 (\surface, \star_1, \star_2)$
 and the reverse $\bar{\gamma}$, we have 
\begin{equation*}
\alpha_{0*}(\gamma)- \gamma_{01\star_1}\alpha_{1*} (\gamma) \gamma_{10 \star_2}
\in \filt{2} 
\GLM (M, \star_{1M,0}, \star_{2M,0}).
\end{equation*}
 So we obtain 
\begin{equation*}
\alpha_{0*}(\gamma)\gamma_{10\star_2}\alpha_{1*} (\bar{\gamma}) \gamma_{10 \star_1}-1
\in \filt{2}\GLM (M, \star_{1M0}).
\end{equation*}
 Since 
\begin{align*}
I_{\GLM (M, \star_{1M0})}/ (I_{\GLM (M ,\star_{1M0})})^2 &\simeq 
H_1 (M, \Q ),
x-1 \mapsto [x],
\end{align*}
 where $I_{\GLM (M, \star_{1M0})} \defeq \shuugou{\sum_{g \in G} q_g g|\sum_{g \in G} q_g=0}$, the homology class 
\begin{equation*}
[\alpha_{0*}(\gamma)\gamma_{10\star_2}\alpha_{1*} (\bar{\gamma}) \gamma_{10 \star_1}]
\end{equation*}
equals $0$ as desired.



(Step 5)
The statement (Step 4) gives $\alpha_* (v_1 + \iota'_{1*} (v_2))=\alpha_* (\iota'_{1*} (v_2))$ for two elements $v_1 \in V$ and $v_2 \in H_1 (\surface, \Z)$. By (Step 3), if $\alpha_* (v_1 + \iota'_{1*} (v_2))=0$, $v_2$ equals $0$. Using (Step 1), we obtain $\ker \alpha_* =V$ as desired.



By the second and fifth steps, we obtain $(2) \Rightarrow (1)$.


\end{proof}

\begin{rem}

For a group $G$, let $I_{\Q G}$ be the augmentation ideal
$\shuugou{\sum_{g \in G} q_g g|\sum_{g \in G} q_g=0}$.
 We can prove $[G, G] \subset 1+ I_{\Q G}^2$ by
\begin{equation*}
\prod_{i=1}^k [\gamma_{i1}, \gamma_{i2}]=
\prod_{i=1}^k (((\gamma_{i1}-1)(\gamma_{i2}-1)-(\gamma_{i2}-1)(\gamma_{i1}-1))\gamma_{i1}^{-1} \gamma_{i2}^{-1}+1)
\end{equation*} 
for elements $\gamma_{11}, \gamma_{21}, \cdots, \gamma_{k1},$
$\gamma_{12}, \gamma_{22}, \cdots ,\gamma_{k2} \in G$.
 So, the linear map 
\begin{align*}
\Q \otimes G/[G,G]  \to I_{\Q G}/ (I_{\Q G})^2,
[x]\mapsto x-1
\end{align*}
is well-defined. It is bijection because 
\begin{equation*}
(x-1)(y-1)= (xy-1)-(x-1)-(y-1).
\end{equation*}
 Hence we have the isomorphism
\begin{align*}
\Q \otimes G/[G,G]  \simeq I_{\Q G}/ (I_{\Q G})^2.
\end{align*}

\end{rem}

\bigskip


For a homology cobordism $(M, \alpha)$, if it has one of the properties in Proposition \ref{prop_jouken_homology_cylinder}, we call it a homology cylinder. Then, it has both of them. 

We introduce two ways to construct homology cylinders. The first way is to use boundary links. The second way is to use standard embeddings, which is an analogy of Heegaard splittings. Both are equivalent to each other, and we can obtain any diffeomorphic class of a homology cylinder in one of both.


\subsection*{``A construction using a boundary link''}


First, we prepare notations about boundary links.


\begin{df}

For a framed unoriented link $L$ in $\surface \times I$, we call it a boundary link if and only if there exists an embedded surface $S$ satisfying the two conditions. 
\begin{itemize}
\item Any component of $S$ has one or two boundary components.
\item The tubular neighborhood of the boundary is isotopic to $L$.
\end{itemize}
 Here we call $S$ a Seifert surface. A label of $S$ is a map $\pi_0 (\partial S)\to \shuugou{-1, +1}$ satisfying 
\begin{equation*}
\partial_1, \partial_2 \subset (\mathrm{a \ component \ of \ } S),
[\partial_1 ] \neq [ \partial_2 ] \Rightarrow
\lambda ([\partial_1]) \lambda ([\partial_2])=-1.
\end{equation*}
 Let a labeled boundary link $L(\lambda)$ be a framed unoriented link having the two properties.
\begin{itemize}
\item The unframed link of $L$ equals that of $L(\lambda)$. 
\item If $l$ and $l'$ are knot components of $L$ and $L(\lambda)$ corresponding to $[\partial] \in \pi_0 (\partial S)$, respectively, we have 
\begin{equation*}
w (l')-w(l)=\lambda ([\partial]).
\end{equation*}
 Here $w (\cdot)$ is the framing number, which is the difference in the number of positive crossings and that of negative ones.
\end{itemize}


\end{df}


Let $L$ be a boundary link in $\surface \times I$ and $\lambda$ a label of a Seifert surface of $L$. We choose a 3-manifold representing the diffeomorphic class obtained by the Dehn surgery of $L (\lambda)$ and denote it by $(\surface \times I)(L (\lambda))$. Then the pair $((\surface \times I) (L(\lambda)), \id_{\partial (\surface \times I)})$ is a homology cylinder. Here we write the same symbol $(\surface \times I)(L(\lambda))$ for the diffeomorphic class of $((\surface \times I) (L(\lambda)), \id_{\partial (\surface \times I)})$. Let $\mathcal{IC}_1 (\surface)$ be the subset of $\mathcal{IC} (\surface)$ consisting of all elements obtained in this way.


\begin{prop}[\cite{Habegger2000} p.4 Theorem 2.5, p.17 Lemma 6.1]
\label{prop_Habegger2000}

We have $\mathcal{IC}_1 (\surface) =\mathcal{IC} (\surface)$. In other words, for any homology cylinder $(M, \alpha)$, there exists a boundary link $L\subset \surface \times I$ and exists a label $\lambda$ of a Seifert surface of $L$ such that the diffeomorphic class 
$(\surface \times I)(L (\lambda))$
is that of $(M, \alpha)$.


\end{prop}

\subsection*{``A construction using a standard embedding''}


We introduce the notion of standard embeddings.


\begin{df}
\label{df_standard_embedding}

We call $e_\mathrm{st}: \surface_\mathrm{st} \times I
\hookrightarrow \surface \times I$ a standard embedding if and only if there exists an embedding $e'_\mathrm{st}: \surface_\mathrm{st} \hookrightarrow
\surface$ satisfying the two properties.
\begin{itemize}
\item The image $e'_\mathrm{st} (\surface)$ is the closure of $\surface$ except for closed disks and includes $\partial \surface$.
\item We have $e_\mathrm{st} (p,t)=(e'_\mathrm{st} (p), \frac{1+t}{3})$ for $p \in \surface_\mathrm{st}$ and $t\in I$.
\end{itemize}

\end{df}


Let $e'_\mathrm{st}; \surface_\mathrm{st} \hookrightarrow
\surface$ be an embedding such that $e_\mathrm{st}: \surface_\mathrm{st} \times I \hookrightarrow
\surface \times I,e(p,t)=e' (p,\frac{1+t}{3})$
 is a standard embedding. Here we denote by $\tilde{\surface}_\mathrm{st}$ the closure of the subsurface $\partial (\surface_\mathrm{st} \times I) \backslash
e_\mathrm{st}^{-1} (\partial \surface \times I)$.

For a diffeomorphism 
$\chi' \in \Diff^+ (\tilde{\surface}_\mathrm{st}, 
\partial \tilde{\surface}_\mathrm{st})$, 
the quotient space of the set
\begin{equation*}
(\mathrm{the \ closure \ of \ }
(\surface \times I) \backslash e ({\surface}_\mathrm{st} \times I)) \sqcup
{\surface}_\mathrm{st} \times I
\end{equation*}
 by the relation 
\begin{equation*}
e_\mathrm{st} \circ \chi'(p) \sim p \mathrm{ \ for \ } p\in \tilde{\surface}_\mathrm{st}
\end{equation*}
is a 3-manifold. We denote by $(\surface \times I) (e_\mathrm{st}, \chi')$ the 3-manifold. We recall that $\torelli' (\tilde{\surface}_\mathrm{st}) \subset 
\MCG (\tilde{\surface}_\mathrm{st})$ is a subgroup generated 
\begin{align*}
\shuugou{t_{c_1} \circ t_{c_2}^{-1}|(c_1,c_2) \mathrm{\ bounds \ a \ surface.}}
\cup
\shuugou{t_c|c \mathrm{\ bounds \ a \ surface.}}.
\end{align*}
 If $\chi$ represents an element $\xi$ of $\torelli' (\surface)$, the pair $((\surface \times I) (e_\mathrm{st}, \chi'), \alpha^{e_\mathrm{st}})$ is a homology cylinder, where we set the diffeomorphism 
$\alpha^{e_\mathrm{st}}:\partial (\surface \times I)
\to \partial ((\surface \times I)(e_\mathrm{st}, \chi'))$
as
\begin{equation*}
\alpha^{e_\mathrm{st}} (p,t)=
\begin{cases}
(p,t) \mathrm{\ if \ }p \in \surface, t \in \shuugou{0,1} \\
(p,t) \mathrm{\ if \ }p\in \partial \surface, t \in [0, \frac{1}{3}] \cup
[\frac{2}{3},1] \\
e_\mathrm{st}^{-1} (p,3t-1) \mathrm{\ if \ }
p \in \partial \surface , t \in [\frac{1}{3}, \frac{2}{3}]
\end{cases}
\end{equation*}
 We write $(\surface \times I)(e_\mathrm{st},\xi)$ for the diffeomorphic class of  $((\surface \times I) (e_\mathrm{st}, \chi'), \alpha^{e_\mathrm{st}})$. Let $\mathcal{IC}_2 (\surface)$ be the subset of $\mathcal{IC} (\surface)$ consisting of all elements obtained in this way. To prove $\mathcal{IC}_2 (\surface) = \mathcal{IC}_1 (\surface)$, we need the following lemma.


\begin{lemm}
\label{lemm_standard_embedding}

Let $S$ be a compact oriented surface and $e$ an embedding from $S$ to $\surface \times I$. Then, there exists a standard embedding $e_\mathrm{st}: \surface_\mathrm{st} \times I
\to \surface \times I$ and exists an embedding 
$\tilde{e} :S \to \tilde{\surface}_\mathrm{st}$
satisfying that the composite 
$e_\mathrm \circ \tilde{e}$
is isotopic to $e$.

\end{lemm}

\begin{proof}
%
%
%
%
%
%
%
%

First, we construct a standard embedding $e_\mathrm{st}: \surface_\mathrm{st} \times I
\to \surface \times I$
by the three steps.
\begin{itemize}
\item We split $S$ into $0$-handles and 1-handles. For example, see Figure
\ref{fig_lemm_BL_standard_step1}. 

\begin{figure}
\begin{picture}(300,180)
\put(-40,0){\input{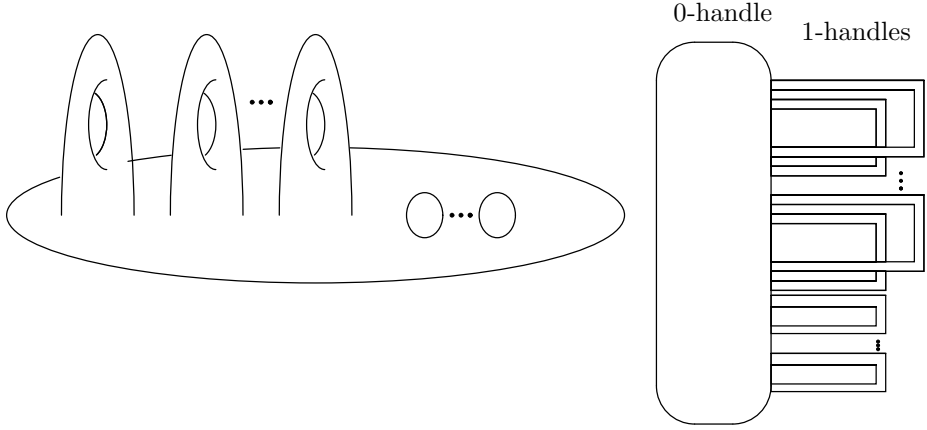}}
\end{picture}
\caption{A $0$-handle and $1$-handles} 
\label{fig_lemm_BL_standard_step1}
\end{figure}

\item We consider $1$-handles as tangles and illustrate a diagram presenting the isotopy type of the embedding $e$ similar to a tangle diagram. For example, see Figure \ref{fig_lemm_BL_standard_step2}.

\begin{figure}
\input{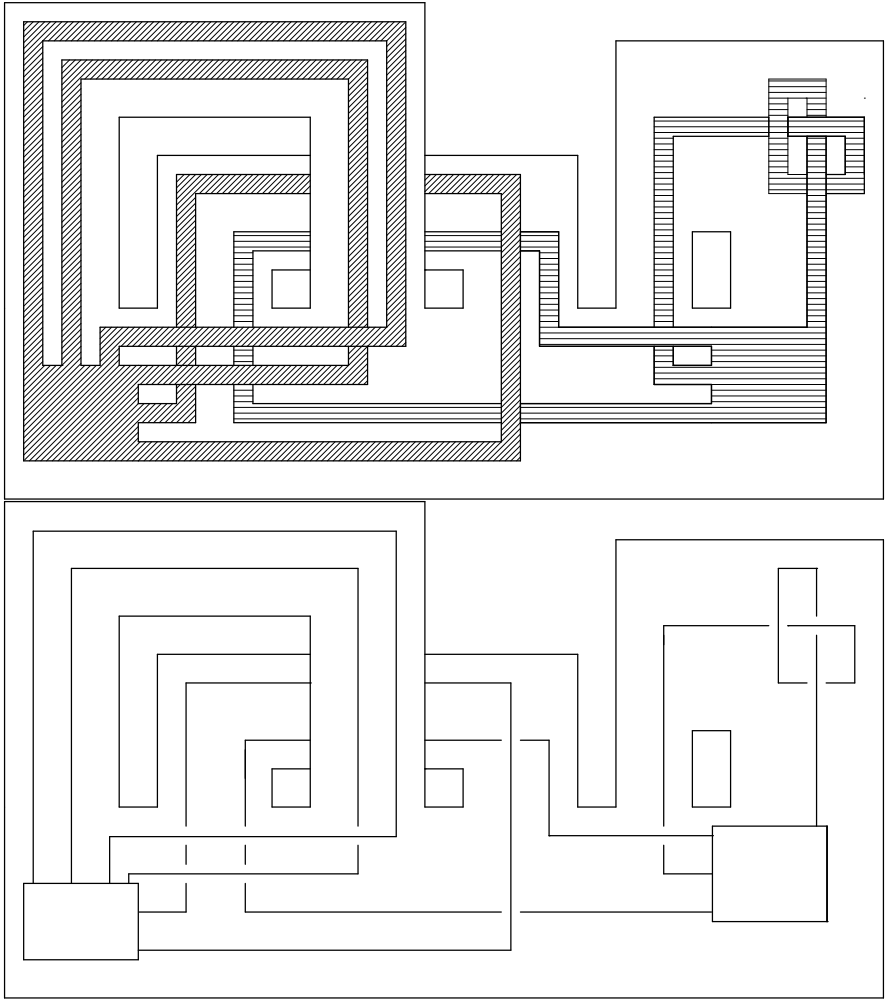}
{\unitlength 0.1in%
\begin{picture}(46.0000,26.0000)(4.0000,-30.0000)%
%
\special{pn 8}%
\special{pa 400 400}%
\special{pa 2600 400}%
\special{fp}%
\special{pa 2600 1800}%
\special{pa 2600 1800}%
\special{fp}%
\special{pa 2600 400}%
\special{pa 2600 2000}%
\special{fp}%
\special{pa 400 2000}%
\special{pa 400 400}%
\special{fp}%
\special{pa 1000 2000}%
\special{pa 1000 1000}%
\special{fp}%
\special{pa 1000 1000}%
\special{pa 2000 1000}%
\special{fp}%
\special{pa 2000 1000}%
\special{pa 2000 2000}%
\special{fp}%
\special{pa 1200 2000}%
\special{pa 1200 1200}%
\special{fp}%
\special{pa 1200 1200}%
\special{pa 2000 1200}%
\special{fp}%
\special{pa 2600 1200}%
\special{pa 3400 1200}%
\special{fp}%
\special{pa 3400 1200}%
\special{pa 3400 2000}%
\special{fp}%
\special{pa 2800 2000}%
\special{pa 2800 1800}%
\special{fp}%
\special{pa 2600 1800}%
\special{pa 2800 1800}%
\special{fp}%
\special{pa 2000 1800}%
\special{pa 1800 1800}%
\special{fp}%
\special{pa 1800 1800}%
\special{pa 1800 2000}%
\special{fp}%
\special{pa 3600 2000}%
\special{pa 3600 600}%
\special{fp}%
\special{pa 3600 600}%
\special{pa 5000 600}%
\special{fp}%
\special{pa 5000 600}%
\special{pa 5000 3000}%
\special{fp}%
\special{pa 5000 3000}%
\special{pa 400 3000}%
\special{fp}%
\special{pa 400 3000}%
\special{pa 400 2000}%
\special{fp}%
\special{pa 1000 2000}%
\special{pa 1200 2000}%
\special{fp}%
\special{pa 1800 2000}%
\special{pa 2000 2000}%
\special{fp}%
\special{pa 2600 2000}%
\special{pa 2800 2000}%
\special{fp}%
\special{pa 3400 2000}%
\special{pa 3600 2000}%
\special{fp}%
\special{pa 4000 2000}%
\special{pa 4000 1600}%
\special{fp}%
\special{pa 4200 1600}%
\special{pa 4000 1600}%
\special{fp}%
\special{pa 4000 2000}%
\special{pa 4200 2000}%
\special{fp}%
\special{pa 4200 2000}%
\special{pa 4200 1600}%
\special{fp}%
%
\special{pn 8}%
\special{pa 1100 2400}%
\special{pa 500 2400}%
\special{fp}%
\special{pa 500 2400}%
\special{pa 500 2800}%
\special{fp}%
\special{pa 500 2800}%
\special{pa 1100 2800}%
\special{fp}%
\special{pa 1100 2400}%
\special{pa 1100 2800}%
\special{fp}%
%
\special{pn 8}%
\special{pa 550 2400}%
\special{pa 550 555}%
\special{fp}%
\special{pa 550 555}%
\special{pa 2450 555}%
\special{fp}%
\special{pa 2450 555}%
\special{pa 2450 2155}%
\special{fp}%
\special{pa 2450 2155}%
\special{pa 950 2155}%
\special{fp}%
\special{pa 950 2155}%
\special{pa 950 2400}%
\special{fp}%
%
\special{pn 8}%
\special{pa 750 2400}%
\special{pa 750 750}%
\special{fp}%
\special{pa 750 750}%
\special{pa 2250 750}%
\special{fp}%
\special{pa 2250 750}%
\special{pa 2250 2100}%
\special{fp}%
\special{pa 2250 2200}%
\special{pa 2250 2350}%
\special{fp}%
\special{pa 2250 2350}%
\special{pa 1050 2350}%
\special{fp}%
\special{pa 1050 2350}%
\special{pa 1050 2400}%
\special{fp}%
%
\special{pn 8}%
\special{pa 1100 2750}%
\special{pa 3050 2750}%
\special{fp}%
\special{pa 3050 2750}%
\special{pa 3050 1350}%
\special{fp}%
\special{pa 3050 1350}%
\special{pa 2600 1350}%
\special{fp}%
\special{pa 2000 1350}%
\special{pa 2000 1350}%
\special{fp}%
\special{pa 2000 1350}%
\special{pa 1350 1350}%
\special{fp}%
\special{pa 1350 1350}%
\special{pa 1350 2100}%
\special{fp}%
\special{pa 1350 2200}%
\special{pa 1350 2300}%
\special{fp}%
\special{pa 1350 2400}%
\special{pa 1350 2550}%
\special{fp}%
\special{pa 1350 2550}%
\special{pa 1105 2550}%
\special{fp}%
%
\special{pn 8}%
\special{pa 4705 2100}%
\special{pa 4105 2100}%
\special{fp}%
\special{pa 4105 2100}%
\special{pa 4105 2600}%
\special{fp}%
\special{pa 4105 2600}%
\special{pa 4705 2600}%
\special{fp}%
\special{pa 4705 2600}%
\special{pa 4705 2100}%
\special{fp}%
\special{pa 4705 2600}%
\special{pa 4705 2600}%
\special{fp}%
\special{pa 4705 2600}%
\special{pa 4705 2100}%
\special{fp}%
\special{pa 4705 2100}%
\special{pa 4705 2600}%
\special{fp}%
%
\special{pn 8}%
\special{pa 4100 2550}%
\special{pa 3100 2550}%
\special{fp}%
\special{pa 3000 2550}%
\special{pa 1660 2550}%
\special{fp}%
\special{pa 1660 2400}%
\special{pa 1660 2550}%
\special{fp}%
\special{pa 1660 2300}%
\special{pa 1660 2200}%
\special{fp}%
\special{pa 1660 2100}%
\special{pa 1660 1700}%
\special{fp}%
\special{pa 1660 1650}%
\special{pa 2000 1650}%
\special{fp}%
\special{pa 2600 1650}%
\special{pa 3000 1650}%
\special{fp}%
\special{pa 3100 1650}%
\special{pa 3250 1650}%
\special{fp}%
\special{pa 3250 2150}%
\special{pa 3250 2150}%
\special{fp}%
\special{pa 3250 2150}%
\special{pa 3250 1650}%
\special{fp}%
\special{pa 3250 2150}%
\special{pa 4110 2150}%
\special{fp}%
%
\special{pn 8}%
\special{pa 3850 1050}%
\special{pa 3850 1150}%
\special{fp}%
%
\special{pn 8}%
\special{pa 1660 1650}%
\special{pa 1660 1850}%
\special{fp}%
%
\special{pn 8}%
\special{pa 3850 2350}%
\special{pa 3850 2200}%
\special{fp}%
\special{pa 3850 2100}%
\special{pa 3850 1100}%
\special{fp}%
\special{pa 3850 1050}%
\special{pa 4400 1050}%
\special{fp}%
\special{pa 4500 1050}%
\special{pa 4500 1050}%
\special{fp}%
\special{pa 4500 1050}%
\special{pa 4850 1050}%
\special{fp}%
\special{pa 4850 1350}%
\special{pa 4700 1350}%
\special{fp}%
\special{pa 4600 1350}%
\special{pa 4450 1350}%
\special{fp}%
\special{pa 4450 1350}%
\special{pa 4450 750}%
\special{fp}%
\special{pa 4650 750}%
\special{pa 4650 750}%
\special{fp}%
\special{pa 4450 750}%
\special{pa 4650 750}%
\special{fp}%
\special{pa 4650 750}%
\special{pa 4650 1000}%
\special{fp}%
\special{pa 4650 1100}%
\special{pa 4650 2100}%
\special{fp}%
%
\special{pn 8}%
\special{pa 4105 2350}%
\special{pa 3850 2350}%
\special{fp}%
\special{pa 4850 1350}%
\special{pa 4850 1050}%
\special{fp}%
\end{picture}}%
\caption{A diagram}
\label{fig_lemm_BL_standard_step2}
\end{figure}

\item We set $\surface_\mathrm{st}$ as the closure of $\surface$ except for closed disks corresponding to angles for any crossing of this diagram. For example, see Figure \ref{fig_lemm_BL_standard_step3}.

\begin{figure}
\input{fig_lemm_BL_standard_step3}
\caption{$\surface_\mathrm{st}$}
\label{fig_lemm_BL_standard_step3}
\end{figure}

\end{itemize}

Furthermore, we set $\tilde{e}: S \to \tilde{\surface}_\mathrm{st}$ as this diagram except for the neighborhood of the crossings, where we set $\tilde{e}$ as Figure \ref{fig_lemm_BL_standard_step4}. Then the composite 
$e_\mathrm{st} \circ \tilde{e}$
is isotopic to $e$. This proves the lemma.

\begin{figure}
\input{fig_lemm_BL_standard_step4}
\caption{$\tilde{e} :S \to \partial (\surface_\mathrm{st} \times I) \backslash
e_\mathrm{st}^{-1} (\partial \surface \times I)$}
\label{fig_lemm_BL_standard_step4}
\end{figure}

\end{proof}


To prove $\mathcal{IC} (\surface)=\mathcal{IC}_2 (\surface)$, we need the following lemma, which we call Lickorish's trick \cite{Lickorish1962}.


\begin{lemm}
\label{lemm_Lickorish's_trick_standard}

Let $S$ be a compact oriented surface such that each component of $S$ has one or two boundary components, $e_\mathrm{st}:\surface_\mathrm{st} \times I \to\surface \times I$ a standard embedding, $\tilde{e}:S \to \tilde{\surface}_\mathrm{st}$ an embedding, and $L_{e_\mathrm{st}\circ \tilde{e}(S)}$ a boundary link whose Seifert surface is $e\circ \tilde{e} (S)$. We write $c_\partial$ for a closed curve in $S$ parallel to a boundary $\partial$. For a label $\lambda: \pi_0 (\partial S)=
\pi_0 (\partial e_\mathrm{st}\circ \tilde{e}(S))
\to \shuugou{\pm 1}$ of $e\circ \tilde{e} (S)$, we have
\begin{equation*}
(\surface \times I)( e_\mathrm{st}, \prod_{[\partial] \in \pi_0 (\partial S)}
t_{\tilde{e} (c_\partial)}^{-\lambda ([\partial])})=
(\surface \times I)(L_{e_\mathrm{st}\circ \tilde{e}(S)}(\lambda)).
\end{equation*}

\end{lemm}

\begin{prop}
\label{prop_standard_embedding}

We have $\mathcal{IC} (\surface)=\mathcal{IC}_2 (\surface)$. In other words, for any homology cylinder $(M,\alpha)$, there exists a standard embedding 
$e_\mathrm{st}:\surface_\mathrm{st} \times I \to \surface \times I$
and exists an element $\xi$ of $\torelli' ( \tilde{\surface}_\mathrm{st})$ such that the diffeomorphic class $(\surface \times I)(e_\mathrm{st}, \xi)$ is that of the pair
$(M, \alpha)$.

\end{prop}

\begin{proof}

By Proposition \ref{prop_Habegger2000}, it is enough to show that $\mathcal{IC}_1(\surface) \subset \mathcal{IC}_2 (\surface)$. Let $L$ be a boundary link in $\surface \times I$ and $e:S\to \surface \times I$ a Seifert surface of $L$. By Lemma \ref{lemm_standard_embedding}, there exists an embedding 
$e_S:S \to \tilde{\surface}_\mathrm{st} =
\overline{\partial (\surface_\mathrm{st}
\times I) \backslash e^{-1} (\partial \surface \times I)}$
 such that the composite
$e_\mathrm{st} \circ e_S$
 is isotopic to this embedding $e$. By Lemma \ref{lemm_Lickorish's_trick_standard}, for a label
$\lambda:\pi_0 (\partial S) \to \shuugou{ \pm 1}$
 of $S$, we have 
\begin{equation*}
(\surface \times I)( e_\mathrm{st}, \prod_{[\partial] \in \pi_0 (\partial S)}
t_{e_S (c_\partial)}^{-\lambda ([\partial])})=
(\surface \times I)(L(\lambda)),
\end{equation*}
 where we denote by $c_\partial$ the simple closed curve parallel to the boundary $\partial$. Hence we obtain $(\surface \times I)(L(\lambda)) \in \mathcal{IC}_2 (\surface)$ as desired.


\end{proof}

\section{The main result}
\label{section_main_theorem_qbtskein}


For $\bullet, * \in \partial \surface$, we consider the action
$\Phi:\mathcal{IC} (\surface) \to \Aut (\cGLM (\surface, \bullet, *))$.
We obtain 
\begin{align*}
& \Xi_* (I_{\GLM (\surface, \bullet)}^n \cGLM (\surface, \bullet,*))
=I_{\GLM (\surface, \bullet)}^n \cGLM (\surface, \bullet,*), \\
&(\Xi_*-\id )(\cGLM (\surface,\bullet,*)) \subset
I_{\GLM (\surface, \bullet)}^2 \cGLM (\surface, \bullet,\star)
\end{align*}
 for any $\Xi$. By Prop \ref{proposition_center_of_the_Goldman}, there exists a unique element 
$\tilde{\zeta} (\Xi)$
of $\filt{3} \cGL (\surface)$ satisfying 
\begin{equation*}
\exp (\sigma (\tilde{\zeta}_{\GL} (\Xi)))=
\Xi_*: \cGLM  (\surface,\bullet, *) \to \cGLM (\surface, \bullet, *).
\end{equation*}
 In this section, we introduce two ways to compute $\tilde{\zeta}:\mathcal{IC} (\surface) \to \filt{3} \cGL (\surface)$. The first way is to use a sequence of the Goldman Lie algebra. The second one is to use the skein algebra $\tskein'$.

To state the main theorems, we introduce an extension of $\tskein' (\surface)$ and set the filtration of it as 
\begin{align*}
&\Loc \tskein' (\surface)\defeq
 \sum_{* \in \Zlarger{0}}
\frac{1}{h^*}F^{3*} \tskein' (\surface)
\subset \Q [h, \frac{1}{h}] \otimes_{\Q}
\tskein' (\surface), \\
&\filt{n} \Loc \tskein' (\surface)=
\sum_{* \in \Zlarger{0}}
\frac{1}{h^*}F^{2*+\max (n ,*)} \tskein' (\surface). \\
\end{align*}
 We consider the completion 
\begin{equation*}
\cLoc \tskein' (\surface)
\defeq
\comp{n} \Loc \tskein' (\surface)/
\filt{n} \Loc \tskein' (\surface)
\end{equation*}
in terms of the filtration. Using this extension, we can obtain a formula about the Baker-Campbel-Hausdorff series as follows.

\begin{rem}
Let $(R, \filtn{\filt{n}R})$ be a filtered completed ring satisfying $\filt{j}R \filt{k} R \subset \filt{j+k} R$. In this paper, we set the bracket $\shuugou{x,y}$ by $xy-yx$. Then, for $x_0, x_1 \in \filt{1}R$, $\bch' (x_0,x_1)\defeq \log (\exp x_0, \exp x_1)$ equals 
\begin{equation*}
\sum_{j} \sum_{\epsilon_1, \cdots, \epsilon_j \in \shuugou{0,1}}
b_{\epsilon_1, \cdots, \epsilon_j} \{ x_{\epsilon_1}, \{ x_{\epsilon_2}, \cdots
\{ x_{\epsilon_{j-1}},x_{\epsilon_j} \} \} \},
\end{equation*}
 where $b_{\epsilon_1, \cdots, \epsilon_j}$ is a rational number. We return to the skein algebra $\tskein' (\surface)$. We set $[x,y]$ as $\frac{1}{h} (xy-yx)$. Hence we obtain 
\begin{align*}
&\bch (x_0,x_1) \\
&=\sum_{j} \sum_{\epsilon_1, \cdots, \epsilon_j \in \shuugou{0,1}}
b_{\epsilon_1, \cdots, \epsilon_j} [ x_{\epsilon_1}, [ x_{\epsilon_2}, \cdots
[ x_{\epsilon_{j-1}},x_{\epsilon_j} ]]] \\
&=h\sum_{j} \sum_{\epsilon_1, \cdots, \epsilon_j \in \shuugou{0,1}}
b_{\epsilon_1, \cdots, \epsilon_j} \{ \frac{1}{h}x_{\epsilon_1}, \{ \frac{1}{h}
x_{\epsilon_2}, \cdots
\{ \frac{1}{h}x_{\epsilon_{j-1}},\frac{1}{h}x_{\epsilon_j} \} \} \} \\
&=h\bch' (\frac{1}{h}x_0,\frac{1}{h}x_1)
=h \log  (\exp (\frac{1}{h}x_0 \frac{1}{h}x_1))
\end{align*}
 for any two elements $x_0,x_1 \in \filt{3} \ctskein' (\surface)$. In other words, we have 
\begin{align*}
\exp (\frac{1}{h} \bch (x_0, x_1))
= \exp (\frac{1}{h} x_0) \exp (\frac{1}{h}x_1).
\end{align*}
\end{rem}


\subsection*{``A construction using a standard embedding''}


We will introduce how to compute $\tilde{\zeta}:\mathcal{IC} (\surface) \to \filt{3} \cGL (\surface)$ using a sequence of $\cGL (\surface)$. We fix a standard embedding 
$e_\mathrm{st}:\surface_\mathrm{st} \times I \to \surface \times I$ and denote by 
$\tilde{\surface}_\mathrm{st}$
 the closure of $\partial (\surface_\mathrm{st} \times I)
\backslash e^{-1} (\partial \surface \times I)$.

The embeddings
\begin{equation*}
\iota_1:\surface_\mathrm{st}  \to \tilde{
\surface}_\mathrm{st},p \mapsto (p,0), \ 
\iota_0:\surface_\mathrm{st}  \to \tilde{
\surface}_\mathrm{st},p \mapsto (p,1), 
\end{equation*}
 induce linear maps
\begin{equation*}
\iota_{1*}: \GL (\surface_\mathrm{st}) \to
\GL(\tilde{\surface}_\mathrm{st}), \
\iota_{0*}: \GL (\surface_\mathrm{st}) \to
\GL(\tilde{\surface}_\mathrm{st}). \
\end{equation*} 
The first embedding $\iota_1$ is an orientation preserving map, but the second one 
$\iota_0$ an orientation reversing one. So the first linear map
$\iota_{1*}: \GL (\surface_\mathrm{st}) \to
\GL(\tilde{\surface}_\mathrm{st})$
 is a Lie algebra homomorphism, but the second one 
$\iota_{0*}: \GL (\surface_\mathrm{st}) \to
\GL(\tilde{\surface}_\mathrm{st})$
is not. The embeddings
\begin{align*}
&\iota_1
:\surface_\mathrm{st}  \times I \to \tilde{
\surface}_\mathrm{st} \times I,(p,t) \mapsto ((p,0),t), \\
&\iota_0
:\surface_\mathrm{st} \times I \to \tilde{
\surface}_\mathrm{st} \times I ,(p,t) \mapsto ((p,1),1-t) \\ 
\end{align*}
 also induce linear maps
\begin{equation*}
\iota_{1*}: \tskein' (\surface_\mathrm{st}) \to
\tskein' (\tilde{\surface}_\mathrm{st}), \
\iota_{0*}: \tskein'  (\surface_\mathrm{st}) \to
\tskein' (\tilde{\surface}_\mathrm{st}). \
\end{equation*}
 Then we have 
\begin{align*}
&\PPsi{\GL}{\tskein'} \circ \iota_{1*} 
=\iota_{1*} \circ \PPsi{\GL}{\tskein'}, \\
&\PPsi{\GL}{\tskein'} \circ \iota_{0*} 
=\iota_{0*} \circ \PPsi{\GL}{\tskein'}. \\
\end{align*} 
 We can extend the linear maps to the completions.

We set the linear map
$\kappa_*:\GL (\tilde{\surface}_\mathrm{st} )\to \GL (\surface_\mathrm{st})$
 as the composite of 
$\GL (\tilde{\surface}_\mathrm{st} ) \to \GL (\surface_\mathrm{st} \times I)$
 induced by the embedding 
$\tilde{\surface}_\mathrm{st} \to
\surface_\mathrm{st} \times I$
and the natural isomorphism
$\GL (\surface_\mathrm{st} \times I) \to 
\GL (\surface_\mathrm{st}).$
 It is not a Lie algebra homomorphism. 
Let 
$\kappa : \tilde{\surface}_\mathrm{st} \times I \to \surface_\mathrm{st} \times I$
 be the tubular neighborhood satisfying the two conditions. 
\begin{itemize}
\item For any $p \in \tilde{\surface}_\mathrm{st}$,
$e_\partial (p,1)=p.$
\item For any $p \in \tilde{\surface}_\mathrm{st} \cap
e_\mathrm{st}^{-1} (\partial \surface \times I)$,
\begin{align*}
&e_\mathrm{st} \circ e_\partial ((p,1),t)=
(e'_\mathrm{st} (p), \frac{2-\epsilon +\epsilon t}{3}) \\
&e_\mathrm{st} \circ e_\partial ((p,0),t)=
(e'_\mathrm{st} (p), \frac{1 +\epsilon-\epsilon t}{3}). \\
\end{align*}
\end{itemize}
The embedding $\kappa$ induces the linear map
$\kappa_* : \tskein' (\tilde{\surface}_\mathrm{st}) \to
\tskein' (\surface_\mathrm{st}).$
Then we have 
\begin{equation*}
\kappa_* \circ \PPsi{\GL}{\tskein'}
=\PPsi{\GL}{\tskein'} \circ \kappa_*
:\GL (\tilde{\surface}_\mathrm{st}) \to
\tskein' (\surface_\mathrm{st}).
\end{equation*}
 We can extend the linear maps to the completion.



Furthermore, since the embeddings 
$\id_{\surface \times I}$, $\kappa \circ \iota_{0}$, and $\kappa\circ \iota_1$
are isotopic, we have 
$\id_{\GL (\surface_\mathrm{st})}=
\kappa_* \circ \iota_{0*}=\kappa \circ \iota_{1*}$
and
$\id_{\tskein' (\surface_\mathrm{st})}=
\kappa_* \circ \iota_{0*}=\kappa \circ \iota_{1*}$. 


\begin{df}

For an element $x \in  \filt{3}\cGL (\tilde{\surface}_\mathrm{st} )$,
 we define a sequence 
$\shuugou{v_n (x)}_{n \in \Zlarger{1}} \subset \filt{3} \cGL (\surface_\mathrm{st})$
by
\begin{align*}
v_1 (x) \defeq \kappa (x), \ 
v_{n+1} (x) \defeq v_{n} (x) +\kappa_* (\bch (-\iota_{1*} (v_{n} (x)), x))
\end{align*}
 and an element $v (x) \in \filt{3} \cGL (\surface_\mathrm{st})$
as 
\begin{equation*}
v(x) \defeq \lim_{n \to \infty} v_n(x).
\end{equation*}

\end{df}


By the following proposition, $v(x)$ is well-defined as an element of the completed Goldman Lie algebra. We can define $v(x)$ using a computation of the skein algebra $\tskein'$. Furthermore, $v(x)$ is a unique solution to $\kappa_* (\bch (-\iota_{1*} (\cdot),x))=0$.


\begin{prop}
\label{proposition_definition_v}
%

Let 
$e_\mathrm{st}: \surface_\mathrm{st} \times I \to
\surface \times I $
 be a standard embedding and 
$\tilde{\surface}_\mathrm{st}$
 be the closure of 
$\partial (\surface_\mathrm{st} \times I)
\backslash e^{-1} (\partial \surface \times I).$
 Then we have the following. 
\begin{enumerate}
\item $v_{n+1}(x)-v_n (x) =\kappa_* (\bch ( -\iota_{1*} (v_n(x)),x))
\in \filt{2+n} \cGL (\surface_\mathrm{st})$.
\item For $x \in \filt{3} \cGL (\tilde{\surface}_\mathrm{st})$ and
$y \in \filt{3} \cGL (\surface_\mathrm{st})$ satisfying 
$\kappa_* (\bch (-\iota_{1*} (y), x))=0$, we have
\begin{equation*}
\exp (\frac{\PPsi{\GL}{\tskein'} (y)}{h})
= \kappa_* (\exp (\frac{\PPsi{\GL}{\tskein'} (x)}{h}))
\in \cLoc \tskein' (\surface_\mathrm{st}).
\end{equation*}
\end{enumerate}
Using them, we obtain two statements. 
\begin{itemize}
\item[(a)]
$v(x)$ is well-defined.
\item[(b)]
$v(x)$ is a unique solution to $\kappa_* (\bch (-\iota_{1*} (\cdot),x))=0$.
\end{itemize}

\end{prop}

\begin{proof}

First, we prove the statement $(1)$ by induction. If $n=1$, 
\begin{align*}
&v_2(x)-v_1 (x)
=\kappa_* (\bch (-\iota_{1*} (v_1(x)),x)) \\
&=\kappa_* (-\iota_{1*} (v_1(x))+x) \filt{4} \cGL (\surface_\mathrm{st}) \\
&=-\kappa_* \circ \iota_{1*} \kappa_* (x)+
\kappa_* (x)=0
\end{align*}
 holds. We assume 
\begin{equation*}
v_{k+1} (x)-v_k (x) =\kappa_* (\bch (-\iota_{1*} (v_k(x)), x))
\in\filt{2+k} \cGL (\surface_\mathrm{st} ).
\end{equation*}
 For $z,w, w' \in \filt{3} \cGL (\surface_\mathrm{st})$ satisfying
$w-w' \in \filt{m} \cGL (\surface_\mathrm{st})$,
 we have 
\begin{equation*}
-z-w +\bch (z, w)=
-z-w'+ \bch (z, w')\mod \filt{m+1} 
\cGL (\surface_\mathrm{st})
\end{equation*}
 by $[\filt{3} \cGL (\surface_\mathrm{st}), \filt{m} \cGL (\surface_\mathrm{st})]
\subset \filt{m+1} \cGL (\surface_\mathrm{st})$. Using it, we have 
\begin{align*}
&v_{k+2} (x)-v_{k+1} (x)=\kappa_* (\bch (-\iota_{1*} (v_{k+1} (x)), x)) \\
&=-\kappa_* \circ \iota_{1*} (v_{k+1} (x))+ \kappa_* (x)
+(\kappa_* \circ \iota_{1*} (v_{k+1} (x))-
\kappa_* (x)+\kappa_* (\bch (-\iota_{1*} (v_{k+1} (x)), x))) \\
&=-\kappa_* \circ \iota_{1*} (v_{k+1} (x))+ \kappa_* (x) \\
&+(\kappa_* \circ \iota_{1*} (v_{k} (x))-
\kappa_* (x)+\kappa_* (\bch (-\iota_{1*} (v_{k} (x)), x)))
\mod \filt{3+k} \cGL (\surface_\mathrm{st}) \\
&=-v_{k+1} (x)+\kappa_* ( x)
+(v_{k} (x)-\kappa_* (x)+\kappa (\bch (-\iota_{1*} (v_{k} (x)), x))) 
 \\
&=-v_{k+1} (x)+v_k (x) +\kappa (\bch (-\iota_{1*} (v_k(x)), x))=0
\end{align*}
 and prove (1).



By (1), we have 
\begin{equation*}
K \geq k \Rightarrow
v_K (x)-v_k (x) \in \filt{k+2} \cGL (\surface),
\end{equation*}
 which means that $\lim_{n \to \infty } v_n (x)$
 is well-defined. In other words, the statement (a) holds.



We will prove the statement (2). By definition of $\bch$, we have 
\begin{align*}
&\kappa_* (\exp (\frac{1}{h} \PPsi{\GL}{\tskein'} (x))) \\
&=\kappa_* (\exp (\frac{1}{h} \bch (\iota_{1*} (\PPsi{\GL}{\tskein'}(y)),
\bch (-\iota_{1*} (\PPsi{\GL}{\tskein'}(y)),\PPsi{\GL}{\tskein'}(x))))) \\
&=\kappa_* (\exp (\frac{1}{h} \iota_{1*} (\PPsi{\GL}{\tskein'}(y)))
\exp (\frac{1}{h}\PPsi{\GL}{\tskein'}(\bch (-\iota_{1*} (y),x)))). \\
\end{align*}
 Since 
\begin{equation*}
\kappa_* (x')=0 \Rightarrow \kappa_* (x'' x')=0
\end{equation*}
 for $x', x'' \in \tskein' (\tilde{\surface}_\mathrm{st})$, we have 
\begin{align*}
&\kappa_* (\exp (\frac{1}{h} \PPsi{\Uh}{\tskein'} (x))) 
=\kappa_* (\exp (\frac{1}{h} \iota_{1*} (\PPsi{\Uh}{\tskein'}(y)))). \\
\end{align*} 
Since $\iota_{1*}$ is an algebra homomorphism, we have 
\begin{align*}
&\kappa_* (\exp (\frac{1}{h} \PPsi{\GL}{\tskein'} (x))) 
=\kappa_* \circ \iota_{1*}(\exp (\frac{1}{h}  \PPsi{\GL}{\tskein'}(y))) 
=\exp (\frac{1}{h}  \PPsi{\GL}{\tskein'}(y)) \\
\end{align*} 
as desired.



Finally, we prove the statement (b). By (1), we have 
$\kappa_* (\bch (-\iota_{1 *} (v(x)),x)) \in \cap_{i=4}^\infty \filt{i}
\cGL (\surface)= \shuugou{0}$, which means that $v(x)$ is a solution to $\kappa_* (\bch (-\iota_{1*} (\cdot), x))=0$. By (2), we have 
\begin{equation*}
\PPsi{\GL}{\Uh} (v(x))= 
\PPsi{\tskein'}{\Uh}( h\log (\kappa_* (\exp (\frac{\PPsi{\GL}{\tskein'} (x)}{h}))))
\in  \cLoc \Uh (\surface).
\end{equation*}
 Since $\PPsi{\GL}{\Uh}$ is an injection, $v(x)$ is a unique solution.

The above statements prove the proposition.


\end{proof}



We state Theorem \ref{thm_qbtskein_main_standard_embedding} using the definition.


\begin{thm}
\label{thm_qbtskein_main_standard_embedding}

Let $e'_\mathrm{st}; \surface_\mathrm{st} \hookrightarrow
\surface$ be an embedding such that the map 
$e_\mathrm{st}: \surface_\mathrm{st} \times I \hookrightarrow
\surface \times I,(p,t) \mapsto (e' (p),\frac{1+t}{3})$
is a standard embedding and $\tilde{\surface}_\mathrm{st}$ the closure of
$\partial (\surface_\mathrm{st} \times I)
\backslash e^{-1} (\partial \surface \times I).$ For an element $\xi \in \torelli' (\tilde{\surface}_\mathrm{st})$, we have 
\begin{equation*}
\tilde{\zeta}_{\GL} ((\surface \times I) (e_\mathrm{st} , \xi))
=e'_{\mathrm{st} *} (v (\zeta_\GL (\xi)))
\in \filt{3} \cGL (\surface).
\end{equation*}
 In other words, for $\star_1, \star_2  \in \partial \surface$, we have 
\begin{equation*}
(\surface \times I) (e_\mathrm{st} , \xi)_*
= \exp (\sigma (e'_{\mathrm{st} *} (v (\zeta_{\GL} (\xi)))):
\cGLM (\surface, \star_1, \star_2)
\to \cGLM (\surface, \star_1, \star_2).
\end{equation*}

\end{thm}


We will prove the theorem in \S \ref{section_proof_main_theorem_qbtskein}.


\subsection*{``A construction using a boundary link''}


For a boundary link $L$ in $\surface \times I$ and a label $\lambda$ of a Seifert surface of $L$, we set an element $L_{\tskein'} (S, \lambda) \in \ctskein' (S)$ as 
\begin{equation*}
L_{\tskein'} (S, \lambda)
\defeq 
-\sum_{[\partial] \in \pi_0 (\partial S)}
\lambda ([\partial]) L_{\tskein'} (c_\partial),
\end{equation*}
where 
\begin{equation*}
L_{\tskein'} (c)
\defeq \PPsi{\GL}{\tskein'} (L_{\GL} (c))
=\PPsi{\GL}{\tskein'} (\zettaiti{\frac{1}{2} \log (\gamma_c)^2})
\end{equation*}
 for any simple closed curve $c$.
 Here we use two notations as follows. 
\begin{itemize}
\item  For a simple close curve $c$, $\gamma_c$ is an element, where the conjugacy class $\zettaiti{\gamma_c}$ equals $c$. 
\item  The simple closed curve $c_\partial$ is parallel to the boundary component $\partial$.
\end{itemize}
 The Seifert surface $S$ defines the embedding
$e_S :S \times I \to\surface \times I$
 inducing the module homomorphism
$e_{S*}:\ctskein' (S) \to \ctskein' (\surface)$. 


\begin{thm}
\label{thm_qbtskein_main_boundary_link}

For a boundary link $L$ in $\surface \times I$ and a label $\lambda$ of a Seifert surface $S$ of $L$, we have
\begin{align*}
&\exp (\frac{1}{h}
\PPsi{\GL}{\tskein'} (\tilde{\zeta}_{\GL} ((\surface \times I) (L(\lambda)))))
=e_{S*} (\exp (\frac{1}{h}L_{\tskein'} (S, \lambda))) \\
&=\sum_{i=0}^\infty \frac{1}{i!h^i}
e_{S*} (L_{\tskein'} (S, \lambda)^i) 
\in \cLoc \tskein' (\surface).
\end{align*}
 In other words, we have
\begin{equation*}
\PPsi{\tskein'}{\Uh} (h \log (e_{S*} (\exp (\frac{1}{h}L_{\tskein'} (S, \lambda)))))
=\PPsi{\GL}{\Uh} \tilde{\zeta}_{\GL} ((\surface \times I) (L(\lambda)))
\in \cUh (\surface).
\end{equation*}

\end{thm}


We will prove the theorem in \S \ref{section_proof_main_theorem_qbtskein}.



\section{Proof of the main result}
\label{section_proof_main_theorem_qbtskein}


In this section, we will prove the theorems in \S \ref{section_main_theorem_qbtskein}. The third statement in Proposition \ref{proposition_technical_homology_cylinder_GL} plays an important role.


\begin{prop}
\label{proposition_technical_homology_cylinder_GL}

Let $\surface$ be a compact connected oriented surface and $M^{(1)}$ a submanifold of $\partial \surface$. We choose base points $\star_1, \star_2 \in M^{(1)}$ and denote by $\tilde{\surface}$ the closure of $\partial (\surface \times I) \backslash
(M^{(1)} \times I)$. We consider linear maps 
\begin{align*}
&\kappa'_*:
\GL (\tilde{\surface})
\to \GL (\surface), \ 
\kappa'_*:
\GLM (\tilde{\surface}, (\star_1, i), (\star_2,i))
\to \GLM (\surface, \star_1,\star_2) \mathrm{\ for \ } i=0,1,\\
&\iota'_{0*}:
\GL (\surface) \to
\GL (\tilde{\surface}), \ 
\iota'_{0*}:
\GLM (\surface, \star_1, \star_2) \to
\GLM (\tilde{\surface}, (\star_1, 0), (\star_2,0)),\\
&\iota'_{1*}:
\GL (\surface) \to
\GL (\tilde{\surface}), \ 
\iota'_{1*}:
\GLM (\surface, \star_1, \star_2) \to
\GLM (\tilde{\surface}, (\star_1, 1), (\star_2,1)). \\
\end{align*}
The homomorphisms $\kappa'_*:
\GL (\tilde{\surface})
\to \GL (\surface)$ and
 $\kappa'_*:
\GLM (\tilde{\surface}, (\star_1, i), (\star_2,i))
\to \GLM (\surface, \star_1,\star_2)$
 are composites of the ones induced by the embedding $\tilde{\surface} \hookrightarrow \surface \times I$ and the natural maps
\begin{align*}
&\GL (\surface \times I)
\to \GL (\surface), \\
&\GLM (\surface \times I, (\star_1,i),(\star_2,i))
\to \GLM (\surface , \star_1, \star_2).
\end{align*}
 The embeddings 
\begin{align*}
\iota_1: \surface \to \tilde{\surface}, p \mapsto (p,1)
\iota_0: \surface \to \tilde{\surface}, p \mapsto (p,0) 
\end{align*}
induce $\iota_{1*}, \iota{0*}$. We can extend them to completions. Then we have the three statements.

\begin{enumerate}
\item

For $i=0,1$, $x \in \GL (\tilde{\surface})$, and $y \in \GLM (\tilde{\surface}, (\star_1, i), (\star_2,i))$, we have 
\begin{equation*}
\kappa'_* (\sigma (x)(y)) =
\kappa'_* (\sigma (x)(\iota'_{1*} \circ \kappa'_* (y)))
+\kappa' (\sigma (\iota'_{0*} \circ\kappa'_* (x))(y)).
\end{equation*}

\item 

For $x \in \GL (\tilde{\surface})$ and $y \in \GLM (\surface,\star_1, \star_2)$ satisfying $\kappa'_* (x)=0$, we have 
\begin{equation*}
\kappa'_* ((\sigma (x))^n (\iota'_{0*}(y)))=
\kappa'_* ((\sigma (x))^n (\iota'_{1*}(y)))
\end{equation*}
for any $n \in \Zlarger{0}$.

\item

We choose a diffeomorphism 
$\chi \in \Diff (\tilde{\surface}, \partial \tilde{\surface})$
representing an element of $\torelli' (\tilde{\surface})$. Let $\zeta_{\GL} (\chi)$ and $v(\zeta_{\GL} (\chi))$ be elements of $\filt{3} \cGL (\tilde{\surface})$ and $\filt{3} \cGL (\surface)$, respectively, satisfying 
\begin{align*}
&
\chi_*=\exp (\sigma (\zeta_{\GL} (\chi))):
\cGLM (\tilde{\surface}, \star'_1, \star'_2) \to
\cGLM (\tilde{\surface}, \star'_1, \star'_2),\\
&
\kappa'_* (\bch (-\iota'_{1*} (v(\zeta_{\GL} (\chi))), 
\zeta_{\GL} (\chi)))=0.
\end{align*}
 Then we have 
\begin{equation*}
\kappa'_* \circ \chi_*^{-1} \circ \iota'_{0*} (x)
=\kappa'_* \circ \chi_*^{-1} \circ \iota'_{1*} 
(\exp (\sigma (v(\zeta_{\GL} (\chi))))(x))
\end{equation*}
 for any $x \in  \cGLM (\surface, \star_1, \star_2)$.

\end{enumerate}
\end{prop}

\begin{proof}


(1)We can prove the statement directly. In this paper, we use the skein algebra $\tskein'$ to prove it. We consider three embeddings
$\kappa' :\tilde{\surface} \times I \to
\surface \times I$,
$\iota'_1:\surface \times I \to \tilde{\surface} \times I$, and
$\iota'_0: \surface \times I \to \tilde{\surface} \times I$.
\begin{itemize}
\item
 The first one satisfies 
\begin{itemize}
\item For $p \in \tilde{\surface}$,
$\kappa' (p,1)=p$.
\item For $p \in M^{(1)}$,
$
\kappa' ( (p,1),t)=(p,1-\epsilon+\epsilon t), \ 
\kappa' ((p,0),t)=(p,\epsilon-\epsilon t).
$
\end{itemize}
 for enough small $\epsilon >0$. 
\item
We define the second one and the third one as 
\begin{align*}
&\iota'_1 (p,t) = ((p,1),t), \ 
\iota'_0 (p,t)=((p,0),1-t).
\end{align*}
\end{itemize}
 There exist isomorphisms 
$\diamondsuit_{(t_0,t_1)}^{(1,0)}:
\tskein' (\surface' \times I, (\star,t_0), (\star', t_1))
\to \tskein' (\surface',\star,\star')$
for any $t_0,t_1 \in I$.
 These embeddings and these isomorphisms induce the linear maps
\begin{align*}
&\kappa'_*:
\tskein' (\tilde{\surface})
\to \tskein' (\surface) \\
&\kappa'_*:
\tskein' (\tilde{\surface}, (\star_1, i), (\star_2,i))
\to \tskein' (\surface, \star_1,\star_2), \\
&\iota'_{i*}:
\tskein' (\surface) \to
\tskein' (\tilde{\surface}), \\
&\iota'_{i*}:
\tskein' (\surface, \star_1, \star_2) \to
\tskein' (\surface, (\star_1, i), (\star_2,i)).\\
\end{align*}

We prove the statement using 
\begin{align*}
&\kappa'_* (xy)=\kappa'_* (x \iota'_{1*} \circ \kappa'_* (y))
=\kappa'_* (x \iota'_{0*} \circ \kappa'_* (y)), \\
&\kappa'_* (yx)=\kappa'_* (y \iota'_{1*} \circ \kappa'_* (x))
=\kappa'_* (y \iota'_{0*} \circ \kappa'_* (x)), \\
&\kappa'_* (x) \kappa'_* (y)
=\kappa'_* (\iota'_{1*} \circ \kappa'_* (x)y)
=\kappa'_* (\iota'_{0*} \circ \kappa'_* (y) x), \\ 
&\kappa'_* (y) \kappa'_* (x)
=\kappa'_* (\iota'_{1*} \circ \kappa'_* (y)x)
=\kappa'_* (\iota'_{0*} \circ \kappa'_* (x) y)
\end{align*}
 for any $x \in \tskein' (\tilde{\surface})$
and $y \in \tskein' (\tilde{\surface}, (\star_1,i),(\star_2,i))$.
 We set 
\begin{align*}
&e_\mathrm{over}:\surface' \times I \to \surface' \times I,(p,t) \mapsto
(p, \frac{2+t}{4}), \\
&e_\mathrm{under}:\surface' \times I \to \surface' \times I, (p,t) \mapsto
(p,\frac{t}{4})
\end{align*}
 for any surface $\surface'$. Since the embeddings 
$\kappa' \circ (e_\mathrm{over} \sqcup e_\mathrm{under})$,
$\kappa' \circ (e_\mathrm{over} \sqcup (e_\mathrm{under}
\circ \iota'_1 \circ \kappa'))$,
and
$\kappa' \circ (e_\mathrm{over} \sqcup (e_\mathrm{under}
\circ \iota'_0 \circ \kappa'))$
are isotopic preserving 
$\shuugou{\star_1, \star_2} \times I$,
 we have the first one and the second one.  Furthermore, since the embeddings
$e_\mathrm{over} \circ \kappa' \sqcup
e_\mathrm{under} \circ \kappa'$,
$\kappa' \circ ((e_\mathrm{over} \circ \iota'_1 \circ \kappa') \sqcup
e_\mathrm{under})$,
and
$\kappa' \circ (e_\mathrm{under}  \sqcup
(e_\mathrm{over} \circ \iota'_0 \circ \kappa'))$ are isotopic preserving 
$\shuugou{\star_1, \star_2} \times I$, 
we also have the third one and the fourth one.



Using the above formulas, we obtain 
\begin{align*}
&h \kappa'_* (\sigma (x)(y)) \\
&= \kappa'_* (x y-yx) \\
&=\kappa'_* (x \iota'_{1*} \circ \kappa'_* (y)-
y \iota'_{0*} \circ \kappa'_* (x)) \\
&=(\kappa'_* (x \iota'_{1*} \circ \kappa'_* (y))-
\kappa'_* (y) \kappa'_* (x))
+(\kappa'_* (y) \kappa'_* (x)-
\kappa'_* (y \iota'_{0*} \circ \kappa'_* (x))) \\
&=\kappa'_* (x \iota'_{1*} \circ \kappa'_* (y)-
\iota'_{1*} \circ \kappa'_* (y) x)
+\kappa'_* (\iota'_{0*} \circ \kappa'_* (x)y
-y \iota'_{0*} \circ \kappa'_* (x)) \\
&=h(\kappa'_* (\sigma (x)(\iota'_{1*} \circ \kappa'_* (y)))
+\kappa'_* (\sigma (\iota'_{0*} \circ \kappa'_* (x))(y)))
\end{align*}
 for any $x \in \tskein' (\tilde{\surface})$
 and $y \in \tskein' (\tilde{\surface}, (\star_1, i), (\star_2,i))$. By Theorem \ref{thm_psi_Uh_tskein'}, $\tskein' (\surface_\mathrm{st}, \star_1, \star_2)$ is a free 
$\Q[h]$-module. So we have 
\begin{equation*}
\kappa'_* (\sigma (x)(y)) =
\kappa'_* (\sigma (x)(\iota'_{1*} \circ \kappa'_* (y)))
+\kappa (\sigma (\iota'_{0*} \circ\kappa'_* (x))(y)).
\end{equation*}

%

For $x' \in \GLM (\tilde{\surface}_\mathrm{st})$ and $y' \in \GLM (\tilde{\surface}_\mathrm{st}, (\star_1, i), (\star_2,i))$, we set $x$ and $y$ as 
\begin{align*}
&x\defeq \PPsi{\GL}{\tskein'} (x'), \
y \defeq \PPsi{\GLM}{\tskein'}(y').
\end{align*}
 We have 
\begin{align*}
&\kappa'_* (\sigma (x')(y')) \\
&=
\PPsi{\tskein'}{\GLM (i,i)} (\kappa'_* (\sigma (x)(y))) \\
&=\PPsi{\tskein'}{\GLM (i,i)}(\kappa'_* (\sigma (x)(\iota'_{1*} \circ \kappa'_* (y)))
+\kappa'_* (\sigma (\iota'_{0*} \circ\kappa'_* (x))(y))) \\
&=\kappa'_* (\sigma (x')(\iota'_{1*} \circ \kappa'_* (y')))
+\kappa'_* (\sigma (\iota'_{0*} \circ\kappa'_* (x'))(y'))
\end{align*}
 as desired.



(2)We will prove the statement by induction on $n$. If $n=0$, 
$
\kappa'_* (\iota'_{0*}(y))=
\kappa'_* (\iota'_{1*}(y))$
 holds. We assume
\begin{equation*}
\kappa'_* ((\sigma (x))^{n-1} (\iota'_{0*}(y)))=
\kappa'_* ((\sigma (x))^{n-1} (\iota'_{1*}(y))).
\end{equation*} 
We have 
\begin{align*}
&\kappa'_* ((\sigma (x))^{n} (\iota'_{0*}(y))) \\
&=\kappa'_* (\sigma (x) (\iota'_{1*} \circ 
\kappa'_* ((\sigma (x))^{n-1} (\iota'_{0*}(y)))))
+\kappa'_* (\sigma (\iota'_{0*} \circ \kappa'_* (x))
(\sigma (x))^{n-1} (\iota'_{0*}(y))) \\
&=\kappa'_* (\sigma (x) (\iota'_{1*} \circ 
\kappa'_* ((\sigma (x))^{n-1} (\iota'_{0*}(y))))), \\
&\kappa'_* ((\sigma (x))^{n} (\iota'_{1*}(y))) \\
&=\kappa'_* (\sigma (x) (\iota'_{1*} \circ 
\kappa'_* ((\sigma (x))^{n-1} (\iota'_{1*}(y)))))
+\kappa'_* (\sigma (\iota'_{0*} \circ \kappa'_* (x))
(\sigma (x))^{n-1} (\iota'_{1*}(y))) \\
&=\kappa'_* (\sigma (x) (\iota'_{1*} \circ 
\kappa'_* ((\sigma (x))^{n-1} (\iota'_{1*}(y))))). \\
\end{align*}
 By this inductive assumption, we obtain 
\begin{equation*}
\kappa'_* ((\sigma (x))^{n} (\iota'_{0*}(y)))
=
\kappa'_* ((\sigma (x))^{n} (\iota'_{1*}(y)))
\end{equation*}
 as desired.


%

(3)We write $\zeta$, $v$, and $v'$ for 
$ \zeta_{\GL} (\chi),$
$ v (\zeta_{\GL} (\chi))$, and $ \bch (-\iota'_{1*} (v),\zeta)$
 respectively. Using 
\begin{equation*}
\chi_*=\exp (\sigma (\zeta )):
\cGL (\tilde{\surface}, \star'_1, \star'_2),
\end{equation*}
 we have 
\begin{align*}
&\kappa'_* \circ \chi_*^{-1} \circ \iota'_{0*} (x)\\
&=\kappa'_* (\exp (\sigma (-\zeta))(\iota'_{0*}(x)))\\
&=\kappa'_* (\exp (\sigma (-\bch (\iota'_{1*}(v),v'))(\iota'_{0*}(x)))\\
&=\kappa'_* (\exp (\sigma (-v')) \circ \exp(\sigma (-\iota'_{1*} (v)))(
\iota'_{0*} (x)))\\
&= \kappa'_* (\exp (\sigma (-v')) (\iota'_{0*} (x))),
\end{align*}

\begin{align*}
&\kappa'_* \circ \chi_*^{-1} \circ \iota'_{1*} 
(\exp (\sigma (v)(x)) \\
&=\kappa'_* (\exp (\sigma (-\zeta))(\iota'_{1*} 
(\exp (\sigma (v)(x)))) \\
&=\kappa'_* (\exp (\sigma (-\bch (\iota'_{1*}(v),v'))(\iota'_{1*} 
(\exp (\sigma (v)(x))))\\
&=\kappa'_* (\exp (\sigma (-v')) \circ \exp(\sigma (-\iota'_{1*} (v)))(
\iota'_{1*} (\exp (\sigma (v)(x))))\\
&=\kappa'_* (\exp (\sigma (-v')) \circ \exp(\sigma (-\iota'_{1*} (v)))\circ
\exp (\sigma (\iota'_{1*}(v))) (\iota'_{1*} (x))\\
&=\kappa'_* (\exp (\sigma (-v'))  (\iota'_{1*} (x)).\\
\end{align*}
 By (2), we obtain 
\begin{equation*}
\kappa'_* \circ \chi_*^{-1} \circ \iota'_{0*} (x)
=\kappa'_* \circ \chi_*^{-1} \circ \iota'_{1*} 
(\exp (\sigma (v(\zeta_{\GL} (\chi))))(x)),
\end{equation*}
 as desired. 

The above statements prove the proposition.


\end{proof}

Let $e'_\mathrm{st}; \surface_\mathrm{st} \hookrightarrow
\surface$
 be an embedding such that 
$e_\mathrm{st}: \surface_\mathrm{st} \times I \hookrightarrow
\surface \times I,(p,t) \mapsto (e'_\mathrm{st},\frac{1+t}{3})$ is a standard embedding and 
$\tilde{\surface}_\mathrm{st}$ be the closure of $\partial (\surface_\mathrm{st} \times I) \backslash (e'^{-1} (\partial \surface)\times I)$. We choose a diffeomorphism $\chi$ representing an element of $\torelli' (\tilde{\surface}_\mathrm{st})$. 
We set $(\surface \times I)(e_\mathrm{st}, \chi)$ as the 3-manifold that is the quotient of the set 
$
\overline{(\surface \times I) \backslash e ({\surface}_\mathrm{st} \times I)} \sqcup
{\surface}_\mathrm{st} \times I$
by the relation
$
e_\mathrm{st} \circ \chi(p) \sim p
$ for any $p \in \tilde{\surface}_\mathrm{st}$.
 We consider the embedding maps.
\begin{itemize}
\item
We denote by $\kappa$ the natural embedding $\tilde{\surface}_\mathrm{st} \to \surface_\mathrm{st} \times I$.
\item
The identity map 
 $\surface_\mathrm{st} \times I \to \surface_\mathrm{st}
\times I$
induces an embedding map
\begin{equation*}
e_\mathrm{st}^\flat:\surface_\mathrm{st} \times I \to
(\surface \times I)(e_\mathrm{st}, \chi).
\end{equation*}
\item
We recall that we set the embeddings 
$\iota_1,\iota_0: \surface_\mathrm{st} \to
\tilde{\surface}_\mathrm{st}$ and 
$\alpha_1, \alpha_0: \surface \to 
(\surface \times I)(e_\mathrm{st}, \chi)$
as 
\begin{align*}
&\iota_1: \surface_\mathrm{st} \to
\tilde{\surface}_\mathrm{st},p \mapsto  (p,1) \\
&\iota_0: \surface_\mathrm{st} \to
\tilde{\surface}_\mathrm{st},p \mapsto (p,0), \\
&\alpha_1: \surface \to (\surface \times I)(e_\mathrm{st}, \chi),
p \mapsto (p,1), \\
&\alpha_0: \surface \to (\surface \times I)(e_\mathrm{st}, \chi),
p \mapsto (p,0). \\
\end{align*}
\end{itemize}
We recall that we set the diffeomorphism $\alpha^{e_\mathrm{st}}:\partial (\surface \times I) \to
\partial ((\surface \times I)(e_\mathrm{st}, \chi))$
as 
\begin{equation*}
\alpha^{e_\mathrm{st}} (p,t)=
\begin{cases}
(p,t) \mathrm{\ if \ }p \in \surface, t \in \shuugou{0,1} \\
(p,t) \mathrm{\ if \ }p\in \partial \surface, t \in [0, \frac{1}{3}] \cup
[\frac{2}{3},1] \\
e_\mathrm{st}^{-1} (p,3t-1) \mathrm{\ if \ }
p \in \partial \surface , t \in [\frac{1}{3}, \frac{2}{3}].
\end{cases}
\end{equation*} 
For any $\star_1, \star_2 \in \partial \surface$ and $t_1, t_2 \in I$, we define
\begin{align*} 
&\diamondsuit_{t_1}^{t_2}:\pi_1 ((\surface \times I)(e_\mathrm{st}, \chi), 
\alpha^{e_\mathrm{st}-1}(\star_1,t_1), \alpha^{e_\mathrm{st}-1}(\star_2, t_1))  \\
&\to \pi_1 ((\surface \times I)(e_\mathrm{st}, \chi), 
\alpha^{e_\mathrm{st}-1}(\star_1,t_2),
\alpha^{e_\mathrm{st}-1} (\star_2, t_2))
\end{align*}
 by $( \cdot ) \mapsto \gamma_{\star_1,t_2, t_1} (\cdot)
\gamma_{\star_2, t_1, t_2}$. Here, for any $\star \in \partial \surface$ and $t,t' \in I$, the continuous map $I \to (\surface \times I )(e, \chi),s \mapsto 
\alpha^{e_\mathrm{st}-1}(\star,ts+t'(1-s))$ represents the path
$\gamma_{\star,t,t'} \in \pi_1 ((\surface \times I)(e, \chi), \alpha^{e_\mathrm{st}-1}(\star,t),
\alpha^{e_\mathrm{st}-1}(\star,t'))$.
 To prove Theorem \ref{thm_qbtskein_main_standard_embedding},
 we need the following lemma.


\begin{lemm}
\label{lemm_technical_homology_cylinder_GL}
\begin{enumerate}
\item 
For any $\star_1, \star_2 \in \partial \surface$, the homomorphism
\begin{equation*}
e'_{\mathrm{st}*}: \pi_1 (\surface_\mathrm{st}, {e'_\mathrm{st}}^{-1} (\star_1),
 {e'_\mathrm{st}}^{-1} (\star_2)) \to \pi_1 (\surface,\star_1, \star_2)
\end{equation*}
 induced by the embedding 
$e'_\mathrm{st}:\surface_\mathrm{st}\to \surface$
is surjective.
\item 

We fix $\star_1, \star_2 \in \partial \surface_\mathrm{st}$. The homomorphisms
\begin{align*}
&(\alpha_1 \circ e_\mathrm{st})_*:\pi_1 (\surface_\mathrm{st}, \star_1, \star_2)
\to \pi_1 ((\surface \times I)(e_\mathrm{st}, \chi), (e'_\mathrm{st} (\star_1),1),
(e'_\mathrm{st} (\star_2),1)), \\
&(\alpha_0 \circ e_\mathrm{st})_* :\pi_1 (\surface_\mathrm{st}, \star_1, \star_2)
\to \pi_1 ((\surface \times I)(e_\mathrm{st}, \chi), (e'_\mathrm{st} (\star_1),0),
(e'_\mathrm{st} (\star_2),0)), \\
&(e_\mathrm{st}^\flat \circ \tilde{\kappa} \circ \chi^{-1} \circ \iota_1)_*
:\pi_1 (\surface_\mathrm{st}, \star_1, \star_2)
\to \pi_1 ((\surface \times I)(e_\mathrm{st}, \chi), (e'_\mathrm{st} (\star_1),\frac{2}{3}),
(e'_\mathrm{st} (\star_2),\frac{2}{3})),  \\
&(e_\mathrm{st}^\flat \circ \tilde{\kappa} \circ \chi^{-1} \circ \iota_0)_*
:\pi_1 (\surface_\mathrm{st}, \star_1, \star_2)
\to \pi_1 ((\surface \times I)(e_\mathrm{st}, \chi), (e'_\mathrm{st} (\star_1),\frac{1}{3}),
(e'_\mathrm{st} (\star_2),\frac{1}{3})),  \\
\end{align*}
 induced by the embeddings
$\alpha_1 \circ e_\mathrm{st}$,
$\alpha_0 \circ e_\mathrm{st}$,
$e_\mathrm{st}^\flat \circ \tilde{\kappa} \circ \chi^{-1} \circ \iota_1$,
and $e_\mathrm{st}^\flat \circ \tilde{\kappa} \circ \chi^{-1} \circ \iota_0$
 satisfy 
\begin{align*}
&(\alpha_1 \circ e_\mathrm{st})_*=\diamondsuit_{\frac{2}{3}}^{1} \circ
(e_\mathrm{st}^\flat \circ \tilde{\kappa} \circ \chi^{-1} \circ \iota_1)_*, \\
&(\alpha_0 \circ e_\mathrm{st})_*=\diamondsuit_{\frac{1}{3}}^{0} \circ
(e_\mathrm{st}^\flat \circ \tilde{\kappa} \circ \chi^{-1} \circ \iota_0)_*.
\end{align*}
\item 

Let $\zeta_{\GL} (\chi)$ and $v(\zeta_{\GL} (\chi))$ be elements in 
$\filt{3} \cGL (\tilde{\surface}_\mathrm{st})$ and $\filt{3}\cGL (\surface_\mathrm{st})$
 satisfying the two properties, respectively. 
\begin{itemize}
\item
For any $\star'_1, \star'_2 \in \partial \tilde{\surface}_\mathrm{st}$,
 we have 
\begin{equation*}
\chi_*=\exp (\sigma (\zeta_{\GL} (\chi))):
\cGLM (\tilde{\surface}_\mathrm{st}, \star'_1, \star'_2) \to
\cGLM (\tilde{\surface}_\mathrm{st}, \star'_1, \star'_2).
\end{equation*} 
\item
We have $
\tilde{\kappa}_* (\bch (-\iota_{1*} (v(\zeta_{\GL} (\chi))), 
\zeta_{\GL} (\chi)))=0$.
\end{itemize} 
Then we have 
\begin{equation*}
\tilde{\kappa}_* \circ \chi_*^{-1} \circ \iota_{0*} (x)
=\diamondsuit_1^0 \circ \tilde{\kappa}_* \circ \chi_*^{-1} \circ \iota_{1*} 
(\exp (\sigma (v(\zeta_{\GL} (\chi))))(x))
\end{equation*}
 for any $x \in  \cGLM (\surface_\mathrm{st}, \star_1, \star_2)$.

\item

Under the assumption in (3), we have 
\begin{equation*}
\alpha_{0*} (x)
=\diamondsuit_1^0 \circ \alpha_{1*}
(\exp (\sigma (e'_{\mathrm{st}*}(v(\zeta_{\GL} (\chi)))))(x))
\end{equation*}
 for any $x \in  \cGLM (\surface, \star_1, \star_2)$.
\end{enumerate}
\end{lemm}

\begin{proof}


(1)Since $e_\mathrm{st}$ is a standard embedding, the closure of $\surface \backslash e'_\mathrm{st} (\surface_\mathrm{st})$ is a connected sum of closed disks. So there exists an embedding $e'': \surface \to \surface_\mathrm{st}$ satisfying the two conditions. 
\begin{itemize}
\item We have $ e'_\mathrm{st} \circ e'' (\star_1)=\star_1, 
 e'_\mathrm{st} \circ e'' (\star_2)=\star_2$.
\item The composite $ e'_\mathrm{st} \circ e'' $ is isotopic to
the identity map preserving $\shuugou{\star_1, \star_2}$.
\end{itemize} Then we have $e'_{\mathrm{st}*} \circ e''_* =\id_{\cGLM (\surface, \star_1, \star_2)}$. So the homomorphism 
\begin{equation*}
e'_{\mathrm{st}*}: \pi_1 (\surface_\mathrm{st}, {e'_\mathrm{st}}^{-1} (\star_1),
 {e'_\mathrm{st}}^{-1} (\star_2)) \to \pi_1 (\surface,\star_1, \star_2)
\end{equation*}
is surjective as desired.



(2)The embeddings 
$\alpha_1 \circ e_\mathrm{st}$ and $\alpha_0 \circ e_\mathrm{st}$
are isotopic to 
$e_\mathrm{st}^\flat \circ \tilde{\kappa} \circ \chi^{-1} \circ \iota_1$
and
$e_\mathrm{st}^\flat \circ \tilde{\kappa} \circ \chi^{-1} \circ \iota_0$
 preserving 
$\star_1 \times I  \cup \star_2 \times I$,
 respectively. It proves (2).



(3)For $i=0,1$, we denote by $\mathrm{Proj}_i$ the natural map 
$\GL (\surface \times I, (\star_1,i),(\star_2,i))
\to \GL (\surface , \star_1, \star_2)$
and have 
$\mathrm{Proj}_0 \circ \diamondsuit_1^{0} =
\mathrm{Proj}_1$. By Proposition \ref{proposition_technical_homology_cylinder_GL}, we obtain 
\begin{equation*}
\mathrm{Proj}_0 \circ \tilde{\kappa}_* \circ \chi_*^{-1} \circ \iota_{0*} (x)
=\mathrm{Proj}_1 \circ \tilde{\kappa}_* \circ \chi_*^{-1} \circ \iota_{1*} 
(\exp (\sigma (v(\zeta_{\GL} (\chi))))(x)).
\end{equation*} 
So we have 
\begin{equation*}
\mathrm{Proj}_0 \circ \tilde{\kappa}_* \circ \chi_*^{-1} \circ \iota_{0*} (x)
=\mathrm{Proj}_0 \circ \diamondsuit_1^{0}
 \circ \tilde{\kappa}_* \circ \chi_*^{-1} \circ \iota_{1*} 
(\exp (\sigma (v(\zeta_{\GL} (\chi))))(x)).
\end{equation*} 
 Since $\mathrm{Proj}_0$ is an isomorphism, we get 
\begin{equation*}
\tilde{\kappa}_* \circ \chi_*^{-1} \circ \iota_{0*} (x)
=\diamondsuit_1^0 \circ \tilde{\kappa}_* \circ \chi_*^{-1} \circ \iota_{1*} 
(\exp (\sigma (v(\zeta_{\GL} (\chi))))(x)),
\end{equation*}
 as desired.

(4)By (1), there exists an element $x' \in \GLM (\surface, {e'_\mathrm{st}}^{-1} (\star_1),
{e'_\mathrm{st}}^{-1} (\star_2))$ satisfying $e'_{\mathrm{st}*} (x')=x$. Using (2), we obtain 
\begin{align*}
&\diamondsuit_1^0 \circ \alpha_{1*}
(\exp (\sigma (e'_{\mathrm{st}*}(v(\zeta_{\GL} (\chi)))))(x)) \\
&=\diamondsuit_1^0 \circ \alpha_{1*}
(\exp (\sigma (e'_{\mathrm{st}*}(v(\zeta_{\GL} (\chi)))))(
e'_{\mathrm{st}*}(x'))) \\
&=\diamondsuit_1^0 \circ \alpha_{1*} \circ e'_{\mathrm{st}*}
(\exp (\sigma (v(\zeta_{\GL} (\chi))))(x')) \\
&=\diamondsuit_1^0 \circ \diamondsuit_{\frac{2}{3}}^1
\circ e^\flat_{\mathrm{st}*} \circ \tilde{\kappa}_* \circ
\chi^{-1}_* \circ \iota_{1*}
(\exp (\sigma (v(\zeta_{\GL} (\chi))))(x')) \\
&=\diamondsuit_{\frac{2}{3}}^0
\circ e^\flat_{\mathrm{st}*} \circ \tilde{\kappa}_* \circ
\chi^{-1}_* \circ \iota_{1*}
(\exp (\sigma (v(\zeta_{\GL} (\chi))))(x')) \\
&=\diamondsuit_{\frac{1}{3}}^0
\circ e^\flat_{\mathrm{st}*} (\diamondsuit_1^0 \circ\tilde{\kappa}_* \circ
\chi^{-1}_* \circ \iota_{1*}
(\exp (\sigma (v(\zeta_{\GL} (\chi))))(x'))),
\end{align*}

\begin{align*}
& \alpha_{0*}(x) \\
&=\alpha_{0*}(e'_\mathrm{st} (x')) \\
&=\diamondsuit_{\frac{1}{3}}^0
\circ e^\flat_{\mathrm{st}*} \circ \tilde{\kappa}_* \circ
\chi^{-1}_* \circ \iota_{0*}
(x') \\
&=\diamondsuit_{\frac{1}{3}}^0
\circ e^\flat_{\mathrm{st}*} (\tilde{\kappa}_* \circ
\chi^{-1}_* \circ \iota_{0*}
(x')).
\end{align*}
 By (3), we have 
\begin{equation*}
\alpha_{0*} (x)
=\diamondsuit_1^0 \circ \alpha_{1*}
(\exp (\sigma (e'_{\mathrm{st}*}(v(\zeta_{\GL} (\chi)))))(x)),
\end{equation*}
 as desired.



The above statements prove the lemma.

\end{proof}

\begin{proof}[proof of 
Theorem \ref{thm_qbtskein_main_standard_embedding}]
%
%

We use the notations in Lemma \ref{lemm_technical_homology_cylinder_GL}(3).
 We assume the diffeomorphism $\chi$ represents an element $\xi$ of $\torelli' (\surface)$. By Proposition
\ref{proposition_center_of_the_Goldman}, $\zeta_{\GL} (\chi)$ equals $\zeta_{\GL} (\xi)$. Furthermore, by Proposition \ref{proposition_definition_v}, $v(\zeta_{\GL} (\chi))$ equals $v(\zeta_{\GL} (\xi))$. Lemma \ref{lemm_technical_homology_cylinder_GL}(4) says 
\begin{equation*}
\alpha_{1*}^{-1} \circ  \diamondsuit_0^1 \circ \alpha_{0*} (x)
=\exp (\sigma (e'_{\mathrm{st}*}(v(\zeta_{\GL} (\xi)))))(x),
\end{equation*} which means
\begin{equation*}
(\surface \times I)(e_\mathrm{st}, \xi)_*
(x)
=\exp (\sigma (e'_{\mathrm{st}*}(v(\zeta_{\GL} (\xi)))))(x).
\end{equation*}
 The equation proves the theorem.


\end{proof}


Next, we will prove Theorem \ref{thm_qbtskein_main_boundary_link}. To do it, we need the following lemma. We recall that $\kappa : \tilde{\surface}_\mathrm{st} \times I \to \surface_\mathrm{st} \times I$ is the tubular neighborhood satisfying the two conditions for enough small $\epsilon >0$.
\begin{itemize}
\item For any $p \in \tilde{\surface}_\mathrm{st}$,
$e_\partial (p,1)=p.$
\item For any $p \in \tilde{\surface}_\mathrm{st} \cap
e_\mathrm{st}^{-1} (\partial \surface \times I)$,
\begin{align*}
&e_\mathrm{st} \circ e_\partial ((p,1),t)=
(e'_\mathrm{st} (p), \frac{2-\epsilon +\epsilon t}{3}), \\
&e_\mathrm{st} \circ e_\partial ((p,0),t)=
(e'_\mathrm{st} (p), \frac{1 +\epsilon-\epsilon t}{3}). \\
\end{align*}
\end{itemize}


\begin{lemm}
\label{lemm_qbtskein_main_boundary_link}

For a standard embedding 
$e_\mathrm{st}: \surface_\mathrm{st} \times I \to
\surface \times I$
and an element 
$\xi \in \torelli (\tilde{\surface}_\mathrm{st})$,
 we have 
\begin{equation*}
\exp (\frac{1}{h} \PPsi{\GL}{\tskein'}(
\tilde{\zeta}_{\GL} ((\surface \times I)(e_\mathrm{st},\xi))))
=e_{\mathrm{st} *} \circ \kappa_*(
\exp (\frac{1}{h} (\PPsi{\GL}{\tskein'}(\zeta_{\GL} (\xi))))).
\end{equation*}


\end{lemm}
\begin{proof}

Let $e'_\mathrm{st}:\surface_\mathrm{st} \to \surface$ be the embedding satisfying $e' \times \id_{\shuugou{\frac{1}{2}}} =e_{\surface \times 
\shuugou{\frac{1}{2}}}$. By Theorem \ref{thm_qbtskein_main_standard_embedding}, we have 
\begin{align*}
&\exp (\frac{1}{h} \PPsi{\GL}{\tskein'}(
\tilde{\zeta}_{\GL} ((\surface \times I)(e_\mathrm{st},\xi))))
=\exp (\frac{1}{h}(
\PPsi{\GL}{\tskein'}(e'_{\mathrm{st}*}( v(\zeta(\xi)))))  \\
&=\exp (\frac{1}{h}(
e_{\mathrm{st}*}(\PPsi{\GL}{\tskein'} (v(\zeta(\xi)))))) 
=e_{\mathrm{st}*} (\exp(\frac{1}{h} 
\PPsi{\GL}{\tskein'} (v(\zeta(\xi))))).
\end{align*}
 Using Proposition \ref{proposition_definition_v}, we have 
\begin{align*}
&e_{\mathrm{st}*} (\exp(\frac{1}{h} 
\PPsi{\GL}{\tskein'} (v(\zeta(\xi))))) 
=e_{\mathrm{st}*} \circ \kappa_* (\exp (\frac{1}{h}
\PPsi{\GL}{\tskein'} (\zeta (\xi)))).
\end{align*} 
The equation proves the lemma.


\end{proof}

\begin{proof}[proof of Theorem 
\ref{thm_qbtskein_main_boundary_link}]

Let $L$ be a boundary link in $\surface \times I$ and $e_S:S \to \surface \times I$ a Seifert surface of $L$. By Lemma \ref{lemm_standard_embedding}, there exists a standard embedding 
$e_{\mathrm{st}} : \surface_\mathrm{st} \times I
\to \surface \times I$
and an embedding
$e':S \to \tilde{\surface}_\mathrm{st}$
 such that the composite $e_{\mathrm{st}} \circ e'$ is isotopic to $e_S$. Using Lemma \ref{lemm_Lickorish's_trick_standard}, for any label $\lambda: \pi_0 (\partial S) \to
\shuugou{\pm{1}}$, we have 
\begin{equation*}
(\surface \times I) (L (\lambda))=
(\surface \times I) (e_\mathrm{st},
\prod_{[\partial] \in \pi_0 (\partial S)}
t_{e' (c_\partial )}^{-\lambda ([\partial])}).
\end{equation*}
 By Lemma \ref{lemm_qbtskein_main_boundary_link}, we have 
\begin{align*}
&\exp (\frac{1}{h} \PPsi{\GL}{\tskein'}(
\tilde{\zeta}_{\GL} ((\surface \times I)(e_\mathrm{st},\xi)))) \\
&=e_{\mathrm{st} *} \circ \kappa_*(
\exp (\frac{1}{h} (\PPsi{\GL}{\tskein'}(\zeta_{\GL} (\xi))))) \\
&=e_{\mathrm{st} *} \circ \kappa_*(
\exp (\frac{1}{h} (\PPsi{\GL}{\tskein'}(\zeta_{\GL} (
\prod_{[\partial] \in \pi_0 (\partial S)}
t_{e' (c_\partial )}^{-\lambda ([\partial])})))))  \\
&=e_{\mathrm{st} *} \circ \kappa_*(
\exp (\frac{1}{h} (\PPsi{\GL}{\tskein'}(\sum_{[\partial]
\in \pi_0 (\partial S)} -\lambda ([\partial ])e'_* (L_{\GL} (c_\partial)))))) \\
&=e_{\mathrm{st} *} \circ \kappa_* \circ e'_*(
\exp (\frac{1}{h} (\PPsi{\GL}{\tskein'}(\sum_{[\partial]
\in \pi_0 (\partial S)} -\lambda ([\partial ])L_{\GL} (c_\partial))))) \\
&=e_{S*}(
\exp (\frac{1}{h} (\PPsi{\GL}{\tskein'}(\sum_{[\partial]
\in \pi_0 (\partial S)} -\lambda ([\partial ])L_{\GL} (c_\partial))))). \\
\end{align*}
 The computation proves the theorem.


\end{proof}

\section{A lemma of $\tskein'$}


To give another proof of the formula \cite{KM2019} of Kuno and Massuyeau and to clarify it, we need a lemma of the filtration of the skein algebra $\tskein'$. In this section, we introduce the lemma and prove it.



Let $\surface$ be a compact connected oriented surface and $\star$ a point in $\partial \surface$. We fix an embedding 
$e: \surface \hookrightarrow
\tilde{\surface} \times I$
and a non-increasing sequence
$\mathbf{b} \defeq \filtn{ b_n}$
 of non-negative integers. For any $m \in \Zlarger{0}$, we set a $\filt{(e, \mathbf{b})m} \GL (\surface)\subset \GL (\surface)$ generated by 
\begin{equation*}
\zettaiti{\eta_1 \eta_2 \cdots \eta_k}
\end{equation*}
 satisfying the properties. 
\begin{itemize}
\item $\eta_1, \eta_2, \cdots, \eta_k \in \GLM (\surface, \star)$.
\item There exist non-negative integers $m_1, \cdots, m_k, n_1, \cdots, c_k$
 satisfying the conditions.
\begin{itemize}
\item For any $i$, $\eta_i \in \filt{m_i} \GLM (\surface, \star)$.
\item For any $i$, $e_* (\eta_i) \in \filt{m_i+n_i} \GLM (\surface, \star)$.
\item We have $\sum_{i=1}^k m_i \geq m.$
\item For any subset $\varsigma \subset \shuugou{1, \cdots, k}$,
$\sum_{i \in \varsigma } n_i \geq b_{k-\sharp \varsigma}$.
\end{itemize}
\end{itemize}



In this paper, we denote by $\mathbf{b}(N)$ the $N$-th element of 
$\mathbf{b} (N)$ for any sequence $\mathbf{b}$. Let $\mathbf{b}_1, \mathbf{b}_2, \cdots$  be some sequences.  We set a new one $\Lambda ( \mathbf{b}_1, \mathbf{b}_2)$ by
\begin{equation*}
\Lambda (\mathbf{b}_1, \mathbf{b}_2)(n)
\defeq \min \shuugou{\mathbf{b}_1 (i)+\mathbf{b}_2(j)| i+j+2=n}.
\end{equation*}
 Then the operator has the associativity, which means 
\begin{equation*}
\Lambda ( \mathbf{b}_1, \Lambda (\mathbf{b}_2, \mathbf{b}_3)) (n)
=
\Lambda (\Lambda (\mathbf{b}_1, \mathbf{b}_2), \mathbf{b}_3)(n)
= \min \shuugou{b_i+b'_j+b''_k|
i+j+k=n-4}.
\end{equation*}
 We denote 
\begin{equation*} 
\Delta (\mathbf{b}_1 , \cdots, \mathbf{b}_j)
=\Delta (\mathbf{b}_1, (\Delta (\mathbf{b}_2, \cdots,
\Delta(\mathbf{b}_{j-1}, \mathbf{b}_j) \cdots ))).
\end{equation*}


%

In this section, we will prove the following statement Statement $(N)$ for any $N\in \Zlarger{1}$. 
\begin{itemize}
\item[] 
Let $\surface_1, \cdots, \surface_N, \surface$ be compact connected oriented surfaces and 
$e= e_1 \sqcup \cdots \sqcup e_N:
\surface_1\times I \sqcup \cdots \sqcup \surface_N
\times I \hookrightarrow
\surface \times I$
 be an embedding that induces homomorphism
$e_*: \tskein' (\surface_1) \otimes \cdots \otimes \tskein' (\surface_N)
\to \tskein' (\surface)$. For any $m_1, \cdots, m_N \in \Zlarger{0}$ and non-increasing sequences $\mathbf{b}_1, \cdots, \mathbf{b}_N$ of non-negative integers, we have
\begin{equation*}
e_* (\otimes_{i=1}^j\PPsi{\GL}{\tskein'}(\filt{(e_i , \mathbf{b}_i)m_i}\GL (\surface_i)))
\subset \filt{(\sum_{i=1}^N m_i)+\Delta (\mathbf{b}_1, \cdots , \mathbf{b}_N)(0)}
\tskein' (\surface).
\end{equation*} 
\end{itemize}


\begin{lemm}
\label{lemm_qbtskein_filtration_N1}

If $N=1$, Statement $(1)$ holds.

\end{lemm}

\begin{proof}

Let $\eta_1, \cdots, \eta_k$ be elements of $\GLM (\surface_1)$ satisfying 
$e_* (\eta_j) \in I_{\GLM (\surface)}^{n_j+m_j}$, where 
\begin{equation*}
\sum_{j =1, \cdots, k} m_j \geq \mathbf{b}_1 (0).
\end{equation*} 
We remark that we can ignore the self-crossing in the skein algebra. We obtain 
\begin{equation*}
e_* (\PPsi{\GL}{\tskein'}
(\zettaiti{\eta_1 \cdots \eta_k}))
\in F^{\sum_{j=1, \cdots, k} (m_j+n_j)} \tskein' (\surface)
\subset 
F^{\sum_{j=1, \cdots, k} m_j + \mathbf{b}_1 (0)}
\tskein' (\surface)
\end{equation*}
 as desired.


\end{proof}

\begin{lemm}
\label{lemm_qbtskein_filtration_bracket}

Let $\surface$ and $\tilde{\surface}$ be compact connected oriented surfaces and $e: \surface \hookrightarrow \tilde {\surface} \times I$ an embedding. For non-increasing sequences $\mathbf{b},\mathbf{b}'$ of non-negative integers and integers
$m,m' \in \Zlarger{0}$, we have 
\begin{equation*}
[\filt{(e,\mathbf{b})m} \GL (\surface),\filt{ (e, \mathbf{b}' )m'} \GL (\surface)]
\subset \filt{ (e, \Lambda (\mathbf{b}, \mathbf{b'})) m+m'-2}
\GL (\surface).
\end{equation*}

\end{lemm}

\begin{proof}

Let $\eta_1, \cdots, \eta_k, \eta'_1, \cdots, \eta'_{k'}$ be elements of $\GLM (\surface,\star)$ satisfying the conditions.
\begin{itemize}
\item For any $j \in \shuugou{1, \cdots, k}$, there exists
a set $S_j$, 
exists $m_j \in \Zlarger{0}$, 
sxists $q_{\alpha_j} \in \Q$ for any $\alpha_j \in S_j$,
and  
 exists $\gamma^{(\alpha_j)}_i \in \pi_1 (\surface, \star)$
for any $\alpha_j \in S_j, i \in \shuugou{1, \cdots, m_j}$
satisfying
\begin{equation*}
\eta_j = \sum_{\alpha_j \in S_j} q_{\alpha_j} (\gamma^{(\alpha_j)}_1-1)
(\gamma^{(\alpha_j)}_2-1) \cdots
(\gamma^{(\alpha_j)}_{m_j}-1)
\in I_{\GLM (\surface, \star)}^{m_j}.
\end{equation*}
\item For any $j' \in \shuugou{1, \cdots, k'}$, there exists
a set $S'_{j'}$, 
exists $m'_{j'} \in \Zlarger{0}$, 
sxists $q'_{\alpha'_{j'}} \in \Q$ for any $\alpha'_{j'} \in S'_{j'}$,
and  
 exists $\gamma'^{(\alpha'_{j'})}_i \in \pi_1 (\surface, \star)$
for any $\alpha'_{j'} \in S'_{j'}, i \in \shuugou{1, \cdots, m'_{j'}}$
satisfying
\begin{equation*}
\eta'_{j'} = \sum_{\alpha'_{j'} \in S'_{j'}} q'_{\alpha'_{j'}} (\gamma'^{(\alpha'_{j'})}_1-1)
(\gamma'^{(\alpha'_{j'})}_2-1) \cdots
(\gamma'^{(\alpha'_{j'})}_{m'_{j'}}-1)
\in I_{\GLM (\surface, \star)}^{m'_{j'}}.
\end{equation*}
\item There exist integers $n_1, \cdots, n_k,n'_1, \cdots, n'_{k'} \in \Zlarger{0}$ such that
\begin{equation*}
e_* (\eta_k) \in I_{\GLM (\tilde{\surface} \times I)}^{m_k+n_k}
\mathrm{ \ and \ }
e_* (\eta'_{k'}) \in I_{\GLM (\tilde{\surface} \times I)}^{m'_{k'}+n'_{k'}}.
\end{equation*}
\item $\sum_{i \in \shuugou{1, \cdots, k}} m_i \geq m$.
\item $\sum_{i \in \shuugou{1, \cdots, k'}} m'_i \geq m'$.
\item For any subset $\varsigma \subset \shuugou{1, \cdots, k}$,
$\sum_{i \in \varsigma} n_i \geq b_{k-\sharp \varsigma}$.
\item For any subset $\varsigma' \subset \shuugou{1, \cdots, k'}$,
$\sum_{i \in \varsigma'} n'_i \geq b'_{k- \sharp \varsigma'}$.
\end{itemize}
Here we recall that $I_{\Q G} \subset \Q G$
is the augmentation ideal $\shuugou{\sum_{g\in G} q_g g|\sum_{g\in G} q_g=0}$.
for any group $G$.
 Then we have 
\begin{align*}
&[ \zettaiti{\eta_1 \eta_2 \cdots \eta_k}, 
\zettaiti{\eta'_1, \eta'_2, \cdots \eta'_{k'}}] \\
&=\sum_{j \in \shuugou{1, \cdots, k}} \sum_{j' \in \shuugou{1, \cdots, k'}}
\sum_{\alpha_j \in S_j} \sum_{\alpha'_{j'} \in S'_{j'}}
\sum_{i \in \shuugou{1, \cdots, m_k}} \sum_{i' \in \shuugou{1, \cdots, m'_{k'}}}
\sum_{p \in \gamma^{( \alpha_j)}_i \cap \gamma'^{( \alpha'_{j'})}_{i'}} \\
&q_{\alpha_j} q'_{\alpha'_{j'}}
\epsilon (p, \gamma^{(\alpha_j)}_i, \gamma'^{(\alpha'_{j'})}_{i'})
\zettaiti{
(\gamma^{(\alpha_j)}_1-1) \cdots (\gamma^{(\alpha_j)}_{i-1}-1)
(\gamma^{(\alpha_j)}_i)_{\star,p} (\gamma'^{(\alpha'_{j'})}_{i'})_{p, \star} \\
&(\gamma^{(\alpha'_{j'})}_{i'+1}-1) \cdots (\gamma^{(\alpha'_{j'})}_{m'_{j'}}-1)
\eta'_{j'+1} \cdots \eta'_{k'} \eta'_{1} \cdots \eta'_{j'-1}
(\gamma'^{(\alpha'_{j'})}_1-1) \cdots (\gamma'^{(\alpha'_{j'})}_{i-1}-1) \\
&(\gamma'^{(\alpha'_{j'})}_{i'})_{\star,p}(\gamma^{(\alpha_{j})}_{i})_{p, \star}
(\gamma^{(\alpha_{j})}_{i+1}-1) \cdots
(\gamma^{(\alpha_{j})}_{m_j}-1)
\eta_{j+1} \cdots \eta_{k} \eta_1 \cdots \eta_{j-1}}.
\end{align*}
 Now, for any subsets
$\varsigma \subset \shuugou{1, \cdots, j-1, j+1, \cdots, k}$
 and $\varsigma' \subset \shuugou{ 1, \cdots, j'-1, j'+1, \cdots, k'}$,
\begin{equation*}
\sum_{i \in \varsigma}n_i+ \sum_{i' \in \varsigma'} n'_{i'}
\geq b_{k- \sharp \varsigma} + b'_{k'-\sharp \varsigma'}
\geq \Delta (\mathbf{b}, \mathbf{b'})(k+k'-2-\sharp \varsigma -\sharp'
\varsigma')
\end{equation*}
holds. By definition, we obtain 
\begin{align*}
&\zettaiti{
(\gamma^{(\alpha_j)}_1-1) \cdots (\gamma^{(\alpha_j)}_{i-1}-1)
(\gamma^{(\alpha_j)}_i)_{\star,p} (\gamma'^{(\alpha'_{j'})}_{i'})_{p, \star} \\
&(\gamma^{(\alpha'_{j'})}_{i'+1}-1) \cdots (\gamma^{(\alpha'_{j'})}_{m'_{j'}}-1)
\eta'_{j'+1} \cdots \eta'_{k'} \eta'_{1} \cdots \eta'_{j'-1}
(\gamma'^{(\alpha'_{j'})}_1-1) \cdots (\gamma'^{(\alpha'_{j'})}_{i-1}-1) \\
&(\gamma'^{(\alpha'_{j'})}_{i'})_{\star,p}(\gamma^{(\alpha_{j})}_{i})_{p, \star}
(\gamma^{(\alpha_{j})}_{i+1}-1) \cdots
(\gamma^{(\alpha_{j})}_{m_j}-1) 
\eta_{j+1} \cdots \eta_{k} \eta_1 \cdots \eta_{j-1}} \\
&\in \filt{(e,\Delta (\mathbf{b}, \mathbf{b}'))m+m'-2}
\GL (\surface).
\end{align*}
 It proves the lemma.

\end{proof}

\begin{lemm}
\label{lemm_qbtskein_filtration_embedding_1}

Let $\surface, \surface_1, \surface_2,$ and
$\tilde{\surface}$ be compact connected oriented surfaces and 
$e :\surface \times I  \to 
\tilde{\surface} \times I$,
$e_1 :\surface_1 \to \surface \times I$,
and $e_2: \surface_2 \to \surface \times I$
 embeddings. We also denote by $e$ the restriction
$\surface \to \tilde{\surface} \times I, p 
\mapsto e(p,0)$. For any elements
$x \in \PPsi{\GL}{\tskein'} (\filt{(e\circ e_1, \mathbf{b}_1)m_1}
\GL (\surface_1))$
and
$y \in \PPsi{\GL}{\tskein'} (\filt{(e \circ e_2, \mathbf{b}_2)m_2}
\GL (\surface_2))$, we obtain 
\begin{equation*}
e_{1*} (x) e_{2*} (y)-e_{2*} (y) e_{1*}(x)
\in h \PPsi{\GL}{\tskein'} (\filt{(e, \Delta (\mathbf{b}_1, \mathbf{b}_2))m_1+m_2-2}
\GL (\surface)).
\end{equation*}
\end{lemm}

\begin{proof}

We remark that, by definition, we have 
\begin{align*}
e_{i*} (\filt{(e \circ e_i, \mathbf{b}_i)m_i} \GL (\surface_i))
\subset \filt{ (e, \mathbf{b}_i)m_i} \GL (\surface)
\end{align*}
 for $i=1,2$. By the relation of $\Uh (\surface)$, we have 
\begin{align*}
\PPsi{\tskein'}{\GL} (e_{1*} (x) e_{2*} (y)-e_{2*} (y) e_{1*}(x))
=h [ \PPsi{\tskein'}{\GL} (e_{1*} (x)), \PPsi{\tskein'}{\Uh}
(e_{2*} (y))].
\end{align*} Using 
Lemma \ref{lemm_qbtskein_filtration_bracket},
 we obtain 
\begin{align*}
h [ \PPsi{\tskein'}{\GL} (e_{1*} (x)), \PPsi{\tskein'}{\GL}
(e_{2*} (y))]
\in
h  \filt{(e, \Delta (\mathbf{b}_1, \mathbf{b}_2))m_1+m_2-2}
\GL (\surface).
\end{align*}
 It proves the lemma.

\end{proof}

\begin{lemm}
\label{lemm_qbtskein_filtration_embedding_2}

Let $\surface_0, \cdots, \surface_{N+1}, \surface$ be compact connected oriented surfaces and 
\begin{align*}
 e^+=e^+_1 \sqcup \cdots \sqcup e^+_{N+1}
&: \surface_1 \times I \sqcup \cdots \sqcup
\surface_{N-1}\times I 
\sqcup \surface_N \times I \sqcup \surface_{N+1}
\times I
\hookrightarrow \surface \times I \\
 e^-=e^-_1 \sqcup \cdots \sqcup e^-_{N+1}
&: \surface_1 \times I \sqcup \cdots \sqcup
\surface_{N-1}\times I 
\sqcup \surface_N \times I \sqcup \surface_{N+1}
\times I
\hookrightarrow \surface \times I \\
e^0 =e^0_1 \sqcup \cdots \sqcup e^0_{N-1} \sqcup e^0_0
&: \surface_1 \times I \sqcup \cdots \sqcup
\surface_{N-1} \times I \sqcup \surface_0 \times I
\hookrightarrow
\surface \times I
\end{align*}
 be embeddings satisfying the three conditions.
\begin{itemize}
\item  The images of $e^+,e^-,e^0$ equal to each other except for a closed ball $D^3$.
\item In $D^3$, the images of $e^+,e^-,e^0$ are submanifolds shown in the figure.

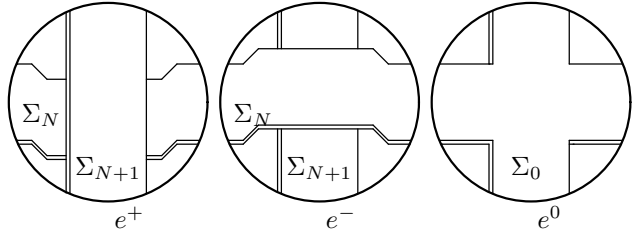
\begin{picture}(240,90)
\put(0,10){
{\unitlength 0.1in%
\begin{picture}(10.4000,10.4000)(1.2000,-11.2000)%
%
\special{pn 13}%
\special{ar 640 600 520 520 0.0000000 6.2831853}%
%
\special{pn 8}%
\special{pa 440 120}%
\special{pa 440 120}%
\special{fp}%
\special{pa 440 120}%
\special{pa 440 1080}%
\special{fp}%
\special{pa 840 1080}%
\special{pa 840 120}%
\special{fp}%
\special{pa 420 130}%
\special{pa 420 1076}%
\special{fp}%
%
\special{pn 8}%
\special{pa 160 800}%
\special{pa 160 800}%
\special{fp}%
\special{pa 160 800}%
\special{pa 240 800}%
\special{fp}%
\special{pa 240 800}%
\special{pa 320 880}%
\special{fp}%
\special{pa 320 880}%
\special{pa 420 880}%
\special{fp}%
\special{pa 840 880}%
\special{pa 920 880}%
\special{fp}%
\special{pa 920 880}%
\special{pa 1000 800}%
\special{fp}%
\special{pa 1000 800}%
\special{pa 1120 800}%
\special{fp}%
%
\special{pn 8}%
\special{pa 160 400}%
\special{pa 160 400}%
\special{fp}%
\special{pa 160 400}%
\special{pa 240 400}%
\special{fp}%
\special{pa 240 400}%
\special{pa 320 480}%
\special{fp}%
\special{pa 320 480}%
\special{pa 420 480}%
\special{fp}%
\special{pa 840 480}%
\special{pa 920 480}%
\special{fp}%
\special{pa 920 480}%
\special{pa 1000 400}%
\special{fp}%
\special{pa 1000 400}%
\special{pa 1120 400}%
\special{fp}%
%
\special{pn 8}%
\special{pa 166 820}%
\special{pa 240 820}%
\special{fp}%
\special{pa 240 820}%
\special{pa 320 900}%
\special{fp}%
\special{pa 320 900}%
\special{pa 420 900}%
\special{fp}%
\special{pa 840 900}%
\special{pa 920 900}%
\special{fp}%
\special{pa 920 900}%
\special{pa 1000 820}%
\special{fp}%
\special{pa 1000 820}%
\special{pa 1110 820}%
\special{fp}%
\end{picture}}%
}
\put(40,0){$e^+$}
\put(80,10){
{\unitlength 0.1in%
\begin{picture}(10.4000,10.4000)(1.2000,-11.2000)%
%
\special{pn 13}%
\special{ar 640 600 520 520 0.0000000 6.2831853}%
%
\special{pn 8}%
\special{pa 160 800}%
\special{pa 160 800}%
\special{fp}%
\special{pa 160 800}%
\special{pa 240 800}%
\special{fp}%
\special{pa 240 800}%
\special{pa 320 720}%
\special{fp}%
\special{pa 320 720}%
\special{pa 920 720}%
\special{fp}%
\special{pa 920 720}%
\special{pa 1000 800}%
\special{fp}%
\special{pa 1000 800}%
\special{pa 1120 800}%
\special{fp}%
\special{pa 1120 400}%
\special{pa 1000 400}%
\special{fp}%
\special{pa 1000 400}%
\special{pa 920 320}%
\special{fp}%
\special{pa 920 320}%
\special{pa 320 320}%
\special{fp}%
\special{pa 320 320}%
\special{pa 240 400}%
\special{fp}%
\special{pa 240 400}%
\special{pa 160 400}%
\special{fp}%
%
\special{pn 8}%
\special{pa 168 820}%
\special{pa 168 820}%
\special{fp}%
\special{pa 168 820}%
\special{pa 240 820}%
\special{fp}%
\special{pa 240 820}%
\special{pa 320 740}%
\special{fp}%
\special{pa 320 740}%
\special{pa 920 740}%
\special{fp}%
\special{pa 920 740}%
\special{pa 1000 820}%
\special{fp}%
\special{pa 1000 820}%
\special{pa 1114 820}%
\special{fp}%
%
\special{pn 8}%
\special{pa 440 120}%
\special{pa 440 120}%
\special{fp}%
\special{pa 440 120}%
\special{pa 440 320}%
\special{fp}%
\special{pa 840 120}%
\special{pa 840 320}%
\special{fp}%
\special{pa 840 740}%
\special{pa 840 1080}%
\special{fp}%
\special{pa 440 1080}%
\special{pa 440 740}%
\special{fp}%
%
\special{pn 8}%
\special{pa 420 128}%
\special{pa 420 128}%
\special{fp}%
\special{pa 420 128}%
\special{pa 420 320}%
\special{fp}%
\special{pa 420 740}%
\special{pa 420 1068}%
\special{fp}%
\end{picture}}%
}
\put(120,0){$e^-$}
\put(160,10){
{\unitlength 0.1in%
\begin{picture}(10.4000,10.4000)(1.2000,-11.2000)%
%
\special{pn 13}%
\special{ar 640 600 520 520 0.0000000 6.2831853}%
%
\special{pn 8}%
\special{pa 440 1080}%
\special{pa 440 800}%
\special{fp}%
\special{pa 440 800}%
\special{pa 160 800}%
\special{fp}%
\special{pa 160 400}%
\special{pa 440 400}%
\special{fp}%
\special{pa 440 120}%
\special{pa 440 400}%
\special{fp}%
\special{pa 840 400}%
\special{pa 840 120}%
\special{fp}%
\special{pa 840 400}%
\special{pa 1120 400}%
\special{fp}%
\special{pa 840 800}%
\special{pa 1120 800}%
\special{fp}%
\special{pa 840 800}%
\special{pa 840 1080}%
\special{fp}%
%
\special{pn 8}%
\special{pa 170 820}%
\special{pa 170 820}%
\special{fp}%
\special{pa 170 820}%
\special{pa 420 820}%
\special{fp}%
\special{pa 420 820}%
\special{pa 420 1070}%
\special{fp}%
%
\special{pn 8}%
\special{pa 420 400}%
\special{pa 420 400}%
\special{fp}%
\special{pa 420 400}%
\special{pa 420 132}%
\special{fp}%
%
\special{pn 8}%
\special{pa 860 820}%
\special{pa 860 820}%
\special{fp}%
\special{pa 1112 820}%
\special{pa 840 820}%
\special{fp}%
\end{picture}}%
}
\put(200,0){$e^0$}
\put(25,20){$\Sigma_{N+1}$}
\put(5,40){$\Sigma_{N}$}
\put(105,20){$\Sigma_{N+1}$}
\put(85,40){$\Sigma_{N}$}
\put(190,20){$\Sigma_{0}$}
\end{picture}

\item We have
\begin{align*}
&e_i^+=e_i^-=e_i^0 \mathrm{ \ for \ } i=1, \cdots, N-1, \\
&(\surface_i \times I) \backslash (e^{+})^{-1} (D^3)
=(\surface_i \times I) \backslash (e^{-})^{-1} (D^3)
\mathrm{ \ for \ } i=N,N+1, \\
&e^+_{i|(\surface_i \times I) \backslash (e^{+})^{-1} (D^3)}=
e^-_{i|(\surface_i \times I) \backslash (e^{-})^{-1} (D^3)}
\mathrm{ \ for \ } i=N,N+1, \\
\end{align*}
\end{itemize}

For non-increasing sequences 
$\mathbf{b}_1, \cdots \mathbf{b}_{N+1}$
of non-negative integers, integers
$m_1$, $\cdots$, $m_{N+1} \in \Zlarger{0}$,
 and 
\begin{align*}
&x_1 \in \filt{(e_1, \mathbf{b}_1)m_1} \GL (\surface_1),\cdots,
x_{N+1} \in \filt{(e_{N+1}, \mathbf{b}_{N+1})m_{N+1}}\GL (\surface_{N+1}),\\
\end{align*}
 the two statements hold. 
\begin{enumerate}
\item There exists $y \in \filt{(e_0, \Delta (\mathbf{b}_N, \mathbf{b}_{N+1}))+m_N+m_{N+1}-2}
\GL (\surface_0)$ satisfying 
\begin{align*}
&e^+_* (\PPsi{\GL}{\tskein'} (x_1) \otimes \cdots\otimes \PPsi{\GL}{\tskein'} (x_{N+1}))
-
e^-_* (\PPsi{\GL}{\tskein'} (x_1) \otimes \cdots \otimes \PPsi{\GL}{\tskein'} (x_{N+1})) \\
&= h
e^0_* (\PPsi{\GL}{\tskein'} (x_1) \otimes \cdots \PPsi{\Uh}{\tskein'} (x_{N-1})
\otimes \PPsi{\GL}{\tskein'} (y)).
\end{align*}
\item  If Statement$(N)$ holds, we have 
\begin{align*}
&e^+_* (\PPsi{\GL}{\tskein'} (x_1) \otimes \cdots\otimes \PPsi{\GL}{\tskein'} (x_{N+1}))
-
e^-_* (\PPsi{\GL}{\tskein'} (x_1) \otimes \cdots \otimes \PPsi{\GL}{\tskein'} (x_{N+1})) \\
&\in \filt{\sum_{i=1}^k m_i +\Delta (\mathbf{b}_1, \cdots, \mathbf{b}_{N+1})(0)}
\tskein' (\surface).
\end{align*}
\end{enumerate}

\end{lemm}

\begin{proof}

(1)We set embeddings
$e_{\mathrm{over}},e_{\mathrm{under}} : \surface_0 \times I \to \surface_0 \times I$
 as $ e_{\mathrm{over}}(p,t)= (p, \frac{t+1}{2})$
and
$e_{\mathrm{under}} (p,t)= (p, \frac{t}{2})$. Let 
$e_{N,0} : \surface_N \times I \to \surface_0 \times I$
 and  $e_{N+1,0}: \surface_{N+1} \times I \to \surface_0 \times I$
 be embeddings satisfying the condition. 
\begin{itemize}
\item The embedding $e^+=e^+_1 \sqcup \cdots \sqcup e^+_{N+1}$ is 
isotopic to
$e^+_1 \sqcup \cdots \sqcup e^+_{N-1} \sqcup e^0_0 \circ e_{\mathrm{under}} \circ e_{N,0}
\sqcup e^0_0 \circ e_{\mathrm{over}} \circ e_{N+1,0}$.
\item The embedding $e^-=e^-_1 \sqcup \cdots \sqcup e^-_{N+1}$
is isotopic to
$e^-_1 \sqcup \cdots \sqcup e^-_{N-1} \sqcup e^0_0 \circ e_{\mathrm{over}} \circ e_{N,0}
\sqcup e^0_0 \circ e_{\mathrm{under}} \circ e_{N+1,0}$.
\end{itemize}
Then we obtain 
\begin{align*}
&e^+_* (\PPsi{\GL}{\tskein'} (x_1) \otimes \cdots \otimes
\PPsi{\GL}{\tskein'} (x_{N+1}))-
e^-_* (\PPsi{\GL}{\tskein'} (x_1) \otimes \cdots \otimes\PPsi{\GL}{\tskein'} (x_{N+1})) \\
&=e^0_*(\PPsi{\GL}{\tskein'} (x_1) \otimes  \cdots \otimes\PPsi{\GL}{\tskein'} (x_{N-1}) \otimes \\
&(e_{N+1,0*} (\PPsi{\GL}{\tskein'} (x_N+1)) e_{N,0*} (\PPsi{\GL}{\tskein'} (x_{N})) 
-e_{N,0*} (\PPsi{\GL}{\tskein'} (x_{N}))e_{N+1,0*} (\PPsi{\GL}{\tskein'} (x_{N+1}))) ).\\
\end{align*} 
By Lemma \ref{lemm_qbtskein_filtration_embedding_1}, we have 
\begin{align*}
&e^+_* (\PPsi{\GL}{\tskein'} (x_1) \otimes \cdots\otimes \PPsi{\GL}{\tskein'} (x_{N+1}))
-
e^-_* (\PPsi{\GL}{\tskein'} (x_1) \otimes \cdots \otimes \PPsi{\GL}{\tskein'} (x_{N+1})) \\
&= h
e^0_* (\PPsi{\GL}{\tskein'} (x_1) \otimes \cdots \PPsi{\GL}{\tskein'} (x_{N-1})
\otimes \PPsi{\GL}{\tskein'} (y))
\end{align*}
 where we set 
\begin{equation*}
y \defeq [e_{N+1*} (x_{N+1}), e_{N*} (x_N)]
\in \filt{(e_0, \Delta (\mathbf{b}_N, \mathbf{b}_{N+1}))(0)+m_N+m_{N+1}-2}
\GL (\surface_{N+1}).
\end{equation*} 
It proves the statement.



(2)Using the statement (1) and Statement $(N)$, we obtain 
\begin{align*}
&e^0_* (\PPsi{\GL}{\tskein'} (x_1) \otimes \cdots \PPsi{\GL}{\tskein'} (x_{N-1})
\otimes \PPsi{\GL}{\tskein'} (y)) \\
&\in \filt{\Delta (\mathbf{b}_1, \cdots, \mathbf{b}_{N-1}, \Delta (\mathbf{b}_{N}, \mathbf{b}_{N+1}))(0)+
(\sum_{i=1}^{N+1} m_i) -2} \tskein' (\surface).
\end{align*} 
By definition, we have 
\begin{equation*}
\Delta (\mathbf{b}_1, \cdots, \mathbf{b}_{N-1}, \Delta (\mathbf{b}_{N}, \mathbf{b}_{N+1}))
=\Delta (\mathbf{b}_1, \cdots, \mathbf{b}_{N+1}).
\end{equation*} 
So we have 
\begin{align*}
&e^+_* (\PPsi{\GL}{\tskein'} (x_1) \otimes \cdots\otimes \PPsi{\GL}{\tskein'} (x_{N+1}))
-
e^-_* (\PPsi{\GL}{\tskein'} (x_1) \otimes \cdots \otimes \PPsi{\GL}{\tskein'} (x_{N+1})) \\
&= h
e^0_* (\PPsi{\GL}{\tskein'} (x_1) \otimes \cdots\otimes \PPsi{\GL}{\tskein'} (x_{N-1})
\otimes \PPsi{\GL}{\tskein'} (y)) \\
&\in \filt{\Delta (\mathbf{b}_1, \cdots, \mathbf{b}_{N+1})(0)+
(\sum_{i=1}^{N+1} m_i) } \tskein' (\surface).
\end{align*}
 It proves the statement.

\end{proof}

\begin{lemm}
\label{lemm_qbtskein_filtration_NN1}

We assume Statement $(N)$. Let 
\begin{align*}
 e=e_1 \sqcup \cdots \sqcup e_{N+1}
&: \surface_1 \times I \sqcup \cdots 
\sqcup \surface_{N+1}
\times I
\hookrightarrow \surface \times I \\
\end{align*}
 be an embedding. For non-increasing sequences 
$\mathbf{b}_1 , \cdots, \mathbf{b}_{N+1}$
of non-negative integers, integers
$m_1, \cdots, m_{N+1}$,
 and elements
\begin{align*}
&x_1 \in \filt{(e_1, \mathbf{b}_1)m_1} \GL (\surface_1),\cdots,
x_{N+1} \in \filt{(e_{N+1}, \mathbf{b}_{N+1})m_{N+1}}\GL (\surface_{N+1}),\\
\end{align*}
we have 
\begin{align}
&e_*(\PPsi{\GL}{\tskein'} (x_1)  \otimes \cdots \otimes
\PPsi{\GL}{\tskein'} (x_{N+1})) \\
&\notag -e_{1*} (\PPsi{\GL}{\tskein'} (x_1)) \cdots
e_{N+1*} (\PPsi{\GL}{\tskein'}(x_{N+1})) \\
&\notag \in \filt{\Delta (\mathbf{b}_1, \cdots, \mathbf{b}_{N+1})(0)+ \sum_{i=1}^{N+1} m_i}
\tskein' (\surface), \\
&e_{1*} (\PPsi{\GL}{\tskein'} (x_1)) \cdots
e_{N+1*} (\PPsi{\GL}{\tskein'}(x_{N+1})) \\
&\notag \in \filt{\Delta (\mathbf{b}_1, \cdots, \mathbf{b}_{N+1})(0)+ \sum_{i=1}^{N+1} m_i}
\tskein' (\surface), \\
&e_*(\PPsi{\GL}{\tskein'} (x_1)  \otimes \cdots \otimes
\PPsi{\GL}{\tskein'} (x_{N+1})) \\
&\notag \in \filt{\Delta (\mathbf{b}_1, \cdots, \mathbf{b}_{N+1})(0)+ \sum_{i=1}^{N+1} m_i}
\tskein' (\surface).
\end{align}

\end{lemm}

\begin{proof}

(1)Using Lemma \ref{lemm_qbtskein_filtration_embedding_2} (2) repeatedly,
 we obtain the equation (1).



(2)By definition, we have 
\begin{equation*}
e_{i*} (\PPsi{\GL}{\tskein'}(x_i)) \in 
\filt{m_i+\mathbf{b}_i (0)} \tskein' (\surface).
\end{equation*} So we obtain 
\begin{align*}
&e_{1*} (\PPsi{\GL}{\tskein'} (x_1))\cdots e_{N+1*} (\PPsi{\GL}{\tskein'} (x_{N+1})) \\
&\in \filt{\sum_{i=1}^{N+1} (m_i +\mathbf{b}_i (0))}
\tskein' (\surface)
\subset \filt{\sum_{i=1}^{N+1} m_i * + \Delta (\mathbf{b}_1, \cdots, \mathbf{b}_{N+1})(0)}
\tskein' (\surface).
\end{align*}

%

(3)Using the equations (1) and (2), we obtain the equation(3).


\end{proof}

\begin{thm}
\label{thm_qbtskein_filtration_assumption}

For any $N \in \Zlarger{1}$, Statement $(N)$ holds. In other words, for any $N \in \Zlarger{1}$, we have the following. Let $\surface_1, \cdots, \surface_N, \surface$ be compact connected oriented surfaces and $e= e_1 \sqcup \cdots \sqcup e_N:
\surface_1\times I \sqcup \cdots \sqcup \surface_N
\times I \hookrightarrow
\surface \times I$ be an embedding that induces homomorphism
$e_*: \tskein' (\surface_1) \otimes \cdots \otimes \tskein' (\surface_N)
\to \tskein' (\surface)$. For any integers $m_1, \cdots, m_N$ and non-increasing sequences 
$\mathbf{b}_1=\filtn{b_{(1)n}}, \cdots, \mathbf{b}_N=
\filtn{b_{(N)i}}$
of non-negative integers, we have 
\begin{align*}
&e_* (\otimes_{i=1}^j\PPsi{\GL}{\tskein'}(\filt{(e_i , \mathbf{b}_i)m_i}\GL (\surface_i)))
\subset \filt{(\sum_{i=1}^j m_i)+\Delta (\mathbf{b}_1, \cdots , \mathbf{b}_N)(0)}
\tskein' (\surface) \\
&= \filt{(\sum_{i=1}^N m_i) + \min \shuugou{b_{(1)k_1} +\cdots
b_{(N)k_N}|k_1 + \cdots + k_N \geq 2N-2}}
\tskein' (\surface).
\end{align*}

\end{thm}

\begin{proof}

By Lemma \ref{lemm_qbtskein_filtration_N1}, we have Statement $(1)$. Using Lemma \ref{lemm_qbtskein_filtration_NN1}(3), by induction, Statement $(N)$ holds for any $N \in \Zlarger{1}$. By definition, we can check 
$\Delta (\mathbf{b}_1, \cdots , \mathbf{b}_N)(k)
=\min \shuugou{b_{(1)k_1} +\cdots
b_{(N)k_N}|k_1 + \cdots + k_N \geq 2N-2+k}$.

\end{proof}
%

Using the lemma, we can prove the corollary.


\begin{cor}
\label{cor_qbtskein_filtration}

Let $\surface'$ and $\surface$ be compact connected oriented surfaces and $e: \surface' \times I \to \surface \times I$  an embedding. For an element $\gamma \in [\pi_1 (\surface'), \pi_1 (\surface')]$
satisfying $e_* (\gamma) \in 1+ I_{\GLM (\surface \times I)}^{m+2}$, we have 
\begin{equation*}
e_* (\prod_{i=1}^k \PPsi{\GL}{\tskein'} (\zettaiti{(\gamma-1)^{m_i}}))
\in \filt{2(\sum_{i=1}^k m_i)+m(\sum_{i=1}^k m_i-(2k-2))}
\tskein' (\surface).
\end{equation*}

\end{cor}

\begin{proof}

For $i=1, \cdots, k$, we set an embedding
$e_i: \surface' \times I \hookrightarrow 
\surface \times I$
 as $e_i (p,t)=e(p,\frac{2k-2i+t}{2k})$ and consider the embedding
$\tilde{e} =e_1 \sqcup \cdots \sqcup e_k: \coprod \surface' \times I
\hookrightarrow \surface \times I$. Then we have 
\begin{equation*}
e_* ( \prod_{i=1}^k \PPsi{\GL}{\tskein'} (\zettaiti{(\gamma-1)^{m_i}}))
=\tilde{e}_* (\otimes_{i=1}^k \PPsi{\GL}{\tskein'} (\zettaiti{(\gamma-1)^{m_i}})).
\end{equation*}
 Setting a sequence
$\mathbf{b}^{(i)} = \filtn{b^{(i)}_n}$
 as 
\begin{equation*}
b^{(i)}_n
=\max (m (m_i-n), 0)
\end{equation*}
for any $i =1, \cdots, k$, 
we obtain
\begin{equation*}
\zettaiti{(\gamma-1)^{m_i}}
\in \filt{(e_i, \mathbf{b}^{(i)})2m_i}
\GL (\surface).
\end{equation*} 
Using Lemma \ref{thm_qbtskein_filtration_assumption}, we have 
\begin{align*}
&\tilde{e}_* (\otimes_{i=1}^k \PPsi{\GL}{\tskein'} (\zettaiti{(\gamma-1)^{m_i}})) 
\in \filt{2(\sum_{i=1}^k m_i) + \min \shuugou{b_{(1)j_1} +\cdots
b_{(k)j_k}|j_1 + \cdots + j_k \geq 2k-2}}
\tskein' (\surface).
\end{align*} 
Since 
\begin{equation*}
m (\sum_{i=1}^k m_i-2k+2) \leq
\min \shuugou{b_{(1)j_1} +\cdots
b_{(k)j_k}|j_1 + \cdots + j_k \geq 2k-2}
\end{equation*}
 we have 
\begin{align*}
&\tilde{e}_* (\otimes_{i=1}^k \PPsi{\GL}{\tskein'} (\zettaiti{(\gamma-1)^{m_i}})) 
\in \filt{2(\sum_{i=1}^k m_i) + m (\sum_{i=1}^k m_i-2k+2)}
\tskein' (\surface).
\end{align*}
 It proves the corollary.


\end{proof}

\section{Some formulas of the action of some homology cylinders}

%

In this section, we will give another proof of the theorem (Theorem \ref{thm_KM} in this paper) of Kuno and Massuyeau \cite{KM2019} and clarify it. Kuno and Massuyeau obtain this theorem by observing how paths behave in the neighborhood of a Seifert surface. On the other hand, we will prove it by using Theorem \ref{thm_qbtskein_main_boundary_link} in this paper.



We apply our knowledge of the skein algebra $\tskein'$ to compute the map $\tilde{\zeta}:\mathcal{IC} (\surface) \to (\filt{3}\cGL (\surface), \bch)$ in the situation of the formula of Kuno and Massuyeau. Let $\surface$ be a compact connected oriented surface. For a knot $K$ in $\surface \times I$, the two properties are equivalent. 
\begin{enumerate}
\item The homology class of $K$ equals 0, and the framing number $w(K)$ of $K$ $0$. 
\item The knot $K$ is a boundary link, which means a boundary knot.
\end{enumerate}


\begin{thm}[\cite{KM2019}]
\label{thm_KM}

Let $\surface$ be a compact connected oriented surface, $\star$ a base point in $\partial \surface$, and $K$ a boundary knot in $\surface \times I$. We denote by $\lambda_\epsilon, \pi_0 (K) \to \shuugou{\pm 1},
[K] \mapsto \epsilon$ the label of a Seifert surface of $K$ and by $(\surface \times I) (K (\epsilon))$ the element $(\surface \times I)(K(\lambda_\epsilon))$ of $\mathcal{IC} (\surface)$. If there exists a path $\gamma \in \pi_1 (\surface) \cap (1+I_{\GLM (\surface)}^m)$ whose conjugacy class equals the homotopy class of $K$, we have 
\begin{align*}
&\tilde{\zeta}_{\GL} ((\surface \times I) (K (-\epsilon)))= \epsilon
\zettaiti{\frac{1}{2} (\log (\gamma))^2}
\mod \filt{2m+2} \cGL (\surface) \\
&\epsilon \zettaiti{\frac{1}{2} (\log (\gamma))^2} \in \filt{2m} \cGL (\surface).
\end{align*}
 In other words, for $\star_1, \star_2 \in \partial \surface$, we have 
\begin{align*}
&((\surface \times I) (K(-\epsilon))_* =
\id \\
&:\GLM (\surface, \star_1, \star_2)/\filt{2m-1} \GLM (\surface, \star_1
,\star_2)
\to \GLM (\surface, \star)/\filt{2m-1} \GLM (\surface, \star_1, 
\star_2), \\
&((\surface \times I) (K(-\epsilon))_* =
\exp (\sigma (\epsilon \frac{1}{2} \zettaiti{(\log \gamma)^2})) \\
&:\GLM (\surface, \star1, \star_2)/\filt{2m+1} \GLM (\surface, \star_1,
 \star_2)
\to \GLM (\surface, \star)/\filt{2m+1} \GLM (\surface, \star_1, 
\star_2). \\
\end{align*}

\end{thm}

\begin{rem}
In the above theorem,
Kuno and Massuyeau proved that 
$\tilde{\zeta}_{\GL} ((\surface \times I) (K (-\epsilon)))
 \mod \filt{2m+2} \cGL (\surface)$
depends only on the homotopy type of the knot $K$.
In this paper \S \ref{section_an_application_KMT}, we will verify 
that the number $2m+2$ is the best possible.

\end{rem}

\begin{proof}

Let $\surface'$ be a compact connected oriented surface and $e :\surface' \times I \to \surface \times I$ an embedding such that the image $e (\partial \surface' \times I)$ represents $K$. Using Theorem \ref{thm_qbtskein_main_boundary_link}, we have 
\begin{equation*}
\PPsi{\GL}{\Uh} (\tilde{\zeta}_{\GL} ((\surface \times I)(K(-\epsilon))))
=h\PPsi{\tskein'}{\Uh} (\log (e_*( \exp(\frac{\epsilon}{h}\PPsi{\GL}{\tskein'} (\frac{1}{2} (\zettaiti{
(\log (\gamma_\partial))^2})))))),
\end{equation*} 
where $\gamma_\partial$ is an element of $\pi_1 (\surface')$ whose conjugacy class equals the boundary $\partial \surface'$. We remark $e_* (\zettaiti{\gamma_\partial^n})=\zettaiti{\gamma^n}$ and
$e_* (\gamma_\partial) \in 1+I_{\GLM (\surface)}^m$. So we have $\epsilon \zettaiti{\frac{1}{2} (\log (\gamma))^2} \in \filt{2m} \cGL (\surface)$.

By Corollary \ref{cor_qbtskein_filtration}, since $4N+(m-2)(2N-2(N-1))=4(N-1)+2m$, we obtain 
\begin{align*}
&e_*( \exp(\frac{\epsilon}{h} \PPsi{\GL}{\tskein'} (\frac{1}{2} (\zettaiti{
(\log (\gamma_\partial))^2})))) \\
&=\sum_{i=0}^\infty \frac{\epsilon^i}{i! h^i} 
e_* (\PPsi{\GL}{\tskein'} ((\frac{1}{2} \zettaiti{(\log (\gamma_\partial ))^2})^i)) \\
&=1+\frac{\epsilon}{h}e_* (\PPsi{\GL}{\tskein'} ((\frac{1}{2} \zettaiti{(\log (\gamma_\partial ))^2})))
\mod \filt{2m} \cLoc \tskein' (\surface). \\
\end{align*}
 Using 
\begin{align*}
&e_*( \exp(\frac{\epsilon}{h} \PPsi{\GL}{\tskein'} (\frac{1}{2} (\zettaiti{
(\log (\gamma_\partial))^2}))))-1
\in \filt{2m-2} \cLoc \tskein' (\surface), \\
\end{align*} 
we have 
\begin{align*}
&h(\log (e_*( \exp(\frac{\epsilon}{h} \PPsi{\GL}{\tskein'} (\frac{1}{2} (\zettaiti{
(\log (\gamma_\partial))^2})))))) \\
&=\sum_{i=1}^\infty \frac{(-1)^{i-1}}{i}
(e_*( \exp(\frac{\epsilon}{h} \PPsi{\GL}{\tskein'} (\frac{1}{2} (\zettaiti{
(\log (\gamma_\partial))^2}))))-1)^i \\
&=\epsilon e_* (\PPsi{\GL}{\tskein'} ((\frac{1}{2} \zettaiti{(\log (\gamma_\partial ))^2})))
\mod \filt{2m+2} \cLoc \tskein' (\surface).
\end{align*} It proves
\begin{align*}
&\tilde{\zeta}_{\GL} ((\surface \times I) (K(- \epsilon))) \\
&=\PPsi{\tskein'}{\GL} 
(\epsilon e_* (\PPsi{\GL}{\tskein'} ((\frac{1}{2} \zettaiti{(\log (\gamma_\partial ))^2})))) 
\mod \filt{2m+2} \cGL (\surface) \\
&=\epsilon e_* (\frac{1}{2} \zettaiti{(\log (\gamma_\partial ))^2}) 
=\epsilon \frac{1}{2} \zettaiti{(\log (\gamma))^2} \\
\end{align*}
 as desired.


\end{proof}

Using the computation in the above proof, we obtain a more accurate formula. Specifically, we have 
\begin{align*}
&e_*( \exp(\frac{\epsilon}{h} \PPsi{\GL}{\tskein'} (\frac{1}{2} (\zettaiti{
(\log (\gamma_\partial))^2})))) \\
&=\sum_{i=0}^\infty \frac{\epsilon^i}{i! h^i} 
e_* (\PPsi{\GL}{\tskein'} ((\frac{1}{2} \zettaiti{(\log (\gamma_\partial ))^2})^i)) \\
&=1+\frac{\epsilon}{h}e_* (\PPsi{\GL}{\tskein'} ((\frac{1}{2} \zettaiti{(\log (\gamma_\partial ))^2}))) 
+\frac{\epsilon^2}{2h^2}e_* ((\PPsi{\GL}{\tskein'} ((\frac{1}{2} \zettaiti{(\log (\gamma_\partial ))^2})))^2) \\
&+\frac{\epsilon^3}{6h^3}e_* ((\PPsi{\GL}{\tskein'} ((\frac{1}{2} \zettaiti{(\log (\gamma_\partial ))^2})))^3)
\mod \filt{2m+4} \cLoc \tskein' (\surface). \\
\end{align*} 
Using the notation
\begin{equation*}
L^{(i)}\defeq e_* ((\PPsi{\GL}{\tskein'} ((\frac{1}{2} \zettaiti{(\log (\gamma_\partial ))^2})))^i),
\end{equation*} 
we get 
\begin{align*}
&\PPsi{\GL}{\Uh} (\tilde{\zeta}_{\GL} ((\surface \times I)(K(-\epsilon)))) \\
&=\PPsi{\tskein'}{\Uh} (h\log(1+\frac{\epsilon}{h}L^{(1)}+\frac{\epsilon^2}{2h^2}L^{(2)}+\frac{\epsilon^3}{6h^3}L^{(3)}))
\mod \filt{2m+6}\cGL (\surface) \\
&=\PPsi{\tskein'}{\Uh} (\epsilon L^{(1)} +\frac{1}{2h} (L^{(2)}-L^{(1)2}) 
+\frac{\epsilon}{12h^2}(2L^{(3)}-3L^{(2)}L^{(1)}-3L^{(1)}L^{(2)}+4L^{(1)3})) \\
&\mod \PPsi{\tskein'}{\Uh} (\filt{2m+6} \cLoc \tskein' (\surface)). \\
\end{align*}
 Here we can prove 
$L^{(2)}-L^{(1)2} \in h \ctskein' (\surface) \subset \ctskein' (\surface)$
and
$2L^{(3)}-3L^{(2)}L^{(1)}-3L^{(1)}L^{(2)}+4L^{(1)3}
\in h^2 \ctskein' (\surface) \subset \ctskein' (\surface)$ 
using the following lemma, where 
$x=y=z=\PPsi{\GL}{\tskein'} (\frac{1}{2} 
\zettaiti{\log (\gamma_\partial)^2})$.


\begin{lemm}

Let $e: \surface' \times I \hookrightarrow
\surface \times I$ be an embedding. Then we have 
\begin{align*}
&e_* (xy)-e_*(x)e_*(y) \in
h \tskein' (\surface), \\
&e_* (xy)-e_*(y)e_*(x) \in
h \tskein' (\surface), \\
&2e_* (xyz)- (e_* (x)e_* (yz)+2e_*(y)e_* (xz))-(e_*(yz)e_*(x)+
2e_* (xy)e_* (z) ) \\
&+e_*(x)e_*(y)e_*(z)+e_* (y) e_* (z)e_* (x)+2 e_* (y)e_*(x)e_*(z)
\in h^2 \tskein' (\surface)
\end{align*}
 for any $x,y,z \in \tskein' (\surface')$.

\end{lemm}

\begin{proof}
%

We will prove the equations by the three steps. We assume $e$ is a tubular neighborhood $e_\partial : \partial (\surface \times I) \times I \to \surface \times I$, an embedding $e_\mathrm{st} \circ e_{\mathrm{st}\partial}: \partial ( \surface_\mathrm{st} \times I) \times I \to \surface \times I$ defined by a standard embedding $e_\mathrm{st}:\surface_\mathrm{st} \times I \to \surface \times I$, or a general embedding in (Step 1),  (Step 2), or (Step 3), respectively.

(Step 1)
For enough small $\epsilon >0$,
we set embeddings $\kappa_0$ and $\kappa_1$ as 
\begin{align*}
&\kappa_0: \surface \times I \to \partial (\surface \times I) \times I,
(p,t) \mapsto e_\partial^{-1} (p, \epsilon t), \\
&\kappa_1: \surface \times I \to \partial (\surface \times I) \times I,
(p,t) \mapsto e_\partial^{-1} (p, 1-\epsilon +\epsilon t), \\
\end{align*}
 which induce 
\begin{align*}
&\kappa_{0*} : \tskein' (\surface) \to \tskein' (\partial (\surface \times I)), \\
&\kappa_{1*}: \tskein' (\surface) \to \tskein' (\partial (\surface \times I)).
\end{align*}
 Since 
\begin{align*}
&\id \simeq e_\partial \circ \kappa_0 \simeq e_\partial \circ \kappa_1, \\
&\kappa_1 \circ e^{\surface}_\mathrm{over} 
\sqcup \kappa_1 \circ e^{\surface}_\mathrm{under}
\simeq
e^{\partial (\surface \times I)}_\mathrm{over}  \circ \kappa_1
\sqcup
e^{\partial (\surface \times I)}_\mathrm{under}  \circ \kappa_1,
 \\
&\kappa_0 \circ e^{\surface}_\mathrm{over} 
\sqcup \kappa_0 \circ e^{\surface}_\mathrm{under}
\simeq
e^{\partial (\surface \times I)}_\mathrm{under}  \circ \kappa_0
\sqcup
e^{\partial (\surface \times I)}_\mathrm{over}  \circ \kappa_0,
 \\
&e_\partial \circ e_\mathrm{over}^{\partial (\surface \times I)}
\sqcup 
e_\partial \circ e_\mathrm{under}^{\partial (\surface \times I)} \\
&\simeq
e_\partial \circ e_\mathrm{over}^{\partial (\surface \times I)}
\sqcup 
e_\partial \circ e_\mathrm{under}^{\partial (\surface \times I)}
\circ \kappa_1 \circ e_\partial \\
&\simeq
e_\partial \circ e_\mathrm{over}^{\partial (\surface \times I)}
\sqcup 
e_\partial \circ e_\mathrm{under}^{\partial (\surface \times I)}
\circ \kappa_0 \circ e_\partial, \\
& e_\mathrm{over}^{\surface}\circ e_\partial 
\sqcup 
e_\mathrm{under}^{\surface} \circ e_\partial \\
&\simeq
e_\partial \circ e_\mathrm{over}^{\partial (\surface \times I)}
\circ \kappa_1 \circ e_\partial
\sqcup 
e_\partial \circ e_\mathrm{under}^{\partial (\surface \times I)} \\
&\simeq
e_\partial \circ e_\mathrm{under}^{\partial (\surface \times I)}
\sqcup 
e_\partial \circ e_\mathrm{over}^{\partial (\surface \times I)}
\circ \kappa_0 \circ e_\partial, \\
\end{align*}
 we have 
\begin{align*}
&\id = e_{\partial * } \circ \kappa_{0*} =e_{\partial *} \circ \kappa_{1*}, \\
& \kappa_{1*} (z'_1 z'_2)=\kappa_{1*} (z'_1) \kappa_{1*} (z'_2), \\
&  \kappa_{0*} (z'_1 z'_2)= \kappa_{0*} (z'_2)\kappa_{0*} (z'_1), \\
&e_{\partial *} (z_1 z_2) =e_{\partial *} (z_1 \kappa_{0*}  \circ e_{\partial *} (z_2))
= e_{\partial *} (z_1 \kappa_{1*}  \circ e_{\partial *} (z_2)) \\
&e_{\partial *} (z_1)e_{\partial *} (z_2)
 =e_{\partial *} ( \kappa_{1*}  \circ e_{\partial *} (z_1)z_2)
= e_{\partial *} ( \kappa_{0*}  \circ e_{\partial *} (z_2)z_1) \\
\end{align*}
 for $z_1 , z_2 \in \tskein' (\partial (\surface \times I)),
z'_1, z'_2 \in \tskein' (\surface)$. We obtain 
\begin{align*}
&e_{\partial *} (x y)-
e_{\partial *} (x) e_{\partial *}(y)
=e_{\partial *} (x  \kappa_{0*} \circ e_{\partial *} (y)-
\kappa_{0*} \circ e_{\partial *} (y) x) \in h \tskein' (\surface), \\
&e_{\partial *} (x y)-
e_{\partial *} (y) e_{\partial *}(x) \\
&=
(e_{\partial *} (x y)-
e_{\partial *} (x) e_{\partial *}(y))
+(e_{\partial *} (x) e_{\partial *}(y)-
e_{\partial *} (y) e_{\partial *}(x))
\in h \tskein' (\surface),
\end{align*}
 which proves the first equation and the second equation.



Using the first equation and the second equation, we will prove 
\begin{align*}
& \tag{a} e_{\partial *} (xyz) -e_{\partial *} (x \kappa_{0*} (e_{\partial *} (y) e_{\partial *}
(z))) \\
& \notag -e_{\partial *}(x) e_{\partial *} (yz)
+e_{\partial *} (x) e_{\partial *} (y) e_{\partial *} (z) 
\in h^2 \tskein' (\surface),  \\
& \tag{b} e_{\partial *} (xyz) -e_{\partial *} (x \kappa_{0*} (e_{\partial *} (y) e_{\partial *}
(z))) \\
& \notag -e_{\partial *} (yz)e_{\partial *}(x) 
+e_{\partial *} (y) e_{\partial *} (z) e_{\partial *} (x) 
\in h^2 \tskein' (\surface), \\
& \tag{c} e_{\partial *} (x \kappa_{0*} (e_{\partial *} (y) e_{\partial *}
(z)))-e_{\partial *} (y) e_{\partial *} (xz) \\
&\notag -e_{\partial *} (xy) e_{\partial *}(z)
+e_{\partial *} (y) e_{\partial *} (x) e_{\partial *}(z)
\in h^2 \tskein' (\surface)
\end{align*}
for $x,y,z \in \tskein' (\partial (\surface \times I))$.



We will prove the equations (a) and (b). Using the first and the second equation, we obtain 
\begin{align*}
&e_{\partial *} (xyz) -e_{\partial *} (x \kappa_{0*} (e_{\partial *} (y) e_{\partial *}
(z)))-e_{\partial *}(x) e_{\partial *} (yz)
+e_{\partial *} (x) e_{\partial *} (y) e_{\partial *} (z) \\
&=e_{\partial *} (x \kappa_{0*} (
e_{\partial *} (yz) -e_{\partial *} (y) e_{\partial *}(z)))
-e_{\partial *} (x) e_{\partial *}( \kappa_{0*} (
e_{\partial *} (yz) -e_{\partial *} (y) e_{\partial *}(z))) \\
& \in h^2 \tskein' (\surface) \\
&e_{\partial *} (xyz) -e_{\partial *} (x \kappa_{0*} (e_{\partial *} (y) e_{\partial *}
(z)))-e_{\partial *} (yz)e_{\partial *}(x) 
+e_{\partial *} (y) e_{\partial *} (z) e_{\partial *} (x) \\
&=e_{\partial *} (xyz) -e_{\partial *} (x \kappa_{1*} (e_{\partial *} (y) e_{\partial *}
(z)))-e_{\partial *} (yz)e_{\partial *}(x) 
+e_{\partial *} (y) e_{\partial *} (z) e_{\partial *} (x) \\
&=e_{\partial *} (x \kappa_{1*} (
e_{\partial *} (yz) -e_{\partial *} (y) e_{\partial *}(z)))
-e_{\partial *}( \kappa_{1*} (e_{\partial *} (yz) -e_{\partial *} (y) e_{\partial *}(z)))
 e_{\partial *} (x) \\
& \in h^2 \tskein' (\surface), \\
\end{align*} 
which verify the equations (a) and (b).

We will prove the equation (c). Since 
\begin{align*}
z'_1 z'_2 = e_{\partial*} ( \kappa_{1*} (z'_1 z'_2))
=e_{\partial *} (\kappa_{1*} (z'_1) \kappa_{1*} ( z'_2))
=e_{\partial *} (\kappa_{1*} (z'_1) \kappa_{0*} ( z'_2))
\end{align*}
 for $z'_1, z'_2 \in \tskein' (\surface)$,
 we obtain 
\begin{align*}
&e_{\partial *} (x \kappa_{0*} (e_{\partial *} (y) e_{\partial *}
(z))) \\
&=e_{\partial *} (x \kappa_{1*} (e_{\partial *} (y) e_{\partial *}
(z))) \\
&=e_{\partial *} (x \kappa_1 \circ e_{\partial *} \circ
\kappa_{1*} (e_{\partial *} (y) e_{\partial *}
(z))) \\
&=e_{\partial *} (x \kappa_1 \circ e_{\partial *} (
(\kappa_{1*} \circ e_{\partial *} (y)) (\kappa_{0*} \circ e_{\partial *}
(z)))) \\
&=e_{\partial *} (x (
(\kappa_{1*} \circ e_{\partial *} (y)) (\kappa_{0*} \circ e_{\partial *}
(z)))). \\
\end{align*} 
By definition, we have 
\begin{align*}
&e_{\partial *} (y) e_{\partial *} (xz)
=e_{\partial *} ((\kappa_{1 *} \circ e_{\partial *} (y))x(\kappa_{0*} \circ e_{\partial *}
(z))), \\
&e_{\partial *} (xy) e_{\partial *} (z)
=e_{\partial *} ((\kappa_{0 *} \circ e_{\partial *} (z))x(\kappa_{1*} \circ e_{\partial *}
(y))). \\
\end{align*} We remark that these elements 
$\kappa_{0*} (z_1)$ and $\kappa_{1*} (z_2)$
are commutative for $z_1, z_2 \in \tskein' (\surface)$. Using it, we have 
\begin{align*}
&e_{\partial *} (y) e_{\partial *} (x) e_{\partial *}(z) 
=e_{\partial *} ((\kappa_{0 *} \circ e_{\partial *} (z))(\kappa_{1*} \circ e_{\partial *}
(y))x).
\end{align*}
 Hence, we obtain 
\begin{align*}
&e_{\partial *} (x \kappa_{0*} (e_{\partial *} (y) e_{\partial *}
(z)))-e_{\partial *} (y) e_{\partial *} (xz) -e_{\partial *} (xy) e_{\partial *}(z)
+e_{\partial *} (y) e_{\partial *} (x) e_{\partial *}(z) \\
&=e_{\partial *} (x (\kappa_{1*} \circ e_{\partial *}(y))(\kappa_{1*} \circ_{\partial *}(z))
- (\kappa_{1*} \circ e_{\partial *}(y))x(\kappa_{1*} \circ e_{\partial *}(z)) \\
&- (\kappa_{1*} \circ e_{\partial *}(z))x(\kappa_{1*} \circ e_{\partial *}(y))
+ (\kappa_{1*} \circ e_{\partial *}(y))(\kappa_{1*} \circ e_{\partial *}(z))x
) \\
&=e_{\partial *} (\kappa_{1*} \circ e_{\partial *}(y)(
(\kappa_{1*} \circ e_{\partial *}(z))x-x(\kappa_{1*} \circ e_{\partial *}(z))) \\
&-(
(\kappa_{1*} \circ e_{\partial *}(z))x-x(\kappa_{1*} \circ e_{\partial *}(z))
\kappa_{1*} \circ e_{\partial *}(y))) \\
& \in h^2 \tskein' (\surface),
\end{align*} 
which verifies the equation (c).



Using the equations (a), (b), and (c), we have 
\begin{align*}
&e_{\partial *} (xyz) -e_{\partial *} (x \kappa_{0*} (e_{\partial *} (y) e_{\partial *}
(z)))-e_{\partial *}(x) e_{\partial *} (yz)
+e_{\partial *} (x) e_{\partial *} (y) e_{\partial *} (z) \\
&+e_{\partial *} (xyz) -e_{\partial *} (x \kappa_{0*} (e_{\partial *} (y) e_{\partial *}
(z)))-e_{\partial *} (yz)e_{\partial *}(x) 
+e_{\partial *} (y) e_{\partial *} (z) e_{\partial *} (x) \\
&+2(e_{\partial *} (x \kappa_{0*} (e_{\partial *} (y) e_{\partial *}
(z)))-e_{\partial *} (y) e_{\partial *} (xz) -e_{\partial *} (xy) e_{\partial *}(z)
+e_{\partial *} (y) e_{\partial *} (x) e_{\partial *}(z)) \\
&=2e_{\partial *} (xyz)- (e_{\partial *} (x)e_{\partial *} (yz)+2e_{\partial *}(y)e_{\partial *} (xz))-
(e_{\partial *}(yz)e_{\partial *}(x)+
2e_{\partial *} (xy)e_{\partial *} (z) ) \\
&+e_{\partial *}(x)e_{\partial *}(y)e_{\partial *}(z)+e_{\partial *} (y) e_{\partial *} (z)e_{\partial *} (x)+
2 e_{\partial *} (y)e_{\partial *}(x)e_{\partial *}(z)
\in h^2 \tskein' (\surface).
\end{align*}
 In other words, we obtain the third equation under the assumption of (Step 1).

(Step 2)
Let $e_\mathrm{st}:\surface_\mathrm{st} \times I \to \surface \times I$ be a standard embedding, $e_{\mathrm{st} \partial}: \partial (\surface_\mathrm{st} \times I) \to \surface_\mathrm{st} \times I$ a tubular neighborhood. Using (Step 1), we obtain
\begin{align*}
&e_{\mathrm{st} \partial *} (xy)-e_{\mathrm{st} \partial *}(x)e_{\mathrm{st} \partial *}(y) \in
h \tskein' (\surface_\mathrm{st}), \\
&e_{\mathrm{st} \partial *} (xy)-e_{\mathrm{st} \partial *}(y)e_{\mathrm{st} \partial *}(x) \in
h \tskein' (\surface_\mathrm{st}), \\
&2e_{\mathrm{st} \partial *} (xyz)- (e_{\mathrm{st} \partial *} (x)e_{\mathrm{st} \partial *} (yz)+
2e_{\mathrm{st} \partial *}(y)e_{\mathrm{st} \partial *} (xz)) \\
&-(e_{\mathrm{st} \partial *}(yz)e_{\mathrm{st} \partial *}(x)
+2e_{\mathrm{st} \partial *} (xy)e_{\mathrm{st} \partial *} (z) ) \\
&+e_{\mathrm{st} \partial *}(x)e_{\mathrm{st} \partial *}(y)e_{\mathrm{st} \partial *}(z)
+e_{\mathrm{st} \partial *} (y) e_{\mathrm{st} \partial *} (z)e_{\mathrm{st} \partial *} (x) \\
&+2 e_{\mathrm{st} \partial *} (y)e_{\mathrm{st} \partial *}(x)e_{\mathrm{st} \partial *}(z)
\in h^2 \tskein' (\surface_\mathrm{st}).
\end{align*}
 We can prove the equations under the assumption $e=e_\mathrm{st} \circ e_{\mathrm{st} \partial}$ of (Step 2) using that $e_{\mathrm{st}*}: \tskein' (\surface_\mathrm{st}) \to \tskein' (\surface)$ is an algebra isomorphism.



(Step 3)
By Lemma \ref{lemm_standard_embedding}, there exists a standard embedding $e_\mathrm{st}:\surface_\mathrm{st} \times I \to \surface \times I$ and exists an embedding $e': \surface' \to \partial (\surface_\mathrm{st} \times I)$ satisfying $e_\mathrm{st} \circ e' \simeq e(\cdot,1)$. We can extend $e': \surface' \to \partial (\surface_\mathrm{st})$ to $e''=e' \times \id_I: \surface' \times I \to \partial (\surface_\mathrm{st}) \times I$. By definition, $e_\mathrm{st} \circ e_{\partial \mathrm{st}} \circ e''$ is also isotopic to $e$. By the second step, we obtain 
\begin{align*}
&e_{\mathrm{st}*}\circ e_{\mathrm{st} \partial *} (x'y')-e_{\mathrm{st}*}\circ 
e_{\mathrm{st} \partial *}(x')e_{\mathrm{st}*}\circ e_{\mathrm{st} \partial *}(y') \in
h \tskein' (\surface_\mathrm{st}), \\
&e_{\mathrm{st}*}\circ e_{\mathrm{st} \partial *} (x'y')-e_{\mathrm{st}*}\circ 
e_{\mathrm{st} \partial *}(y')e_{\mathrm{st}*}\circ e_{\mathrm{st} \partial *}(x') \in
h \tskein' (\surface_\mathrm{st}), \\
&2e_{\mathrm{st}*}\circ e_{\mathrm{st} \partial *} (x'y'z') \\
&- (e_{\mathrm{st}*}\circ 
e_{\mathrm{st} \partial *} (x')e_{\mathrm{st}*}\circ e_{\mathrm{st} \partial *} (y'z')+
2e_{\mathrm{st}*}\circ e_{\mathrm{st} \partial *}(y')e_{\mathrm{st}*}\circ e_{\mathrm{st} \partial *} (x'z')) \\
&-(e_{\mathrm{st}*}\circ e_{\mathrm{st} \partial *}(y'z')
e_{\mathrm{st}*}\circ e_{\mathrm{st} \partial *}(x')+
2e_{\mathrm{st}*}\circ e_{\mathrm{st} \partial *} (x'y')e_{\mathrm{st}*}\circ e_{\mathrm{st} \partial *} (z') ) \\
&+e_{\mathrm{st}*}\circ e_{\mathrm{st} \partial *}(x')
e_{\mathrm{st}*}\circ e_{\mathrm{st} \partial *}(y')
e_{\mathrm{st}*}\circ e_{\mathrm{st} \partial *}(z') \\
&+e_{\mathrm{st}*}\circ e_{\mathrm{st} \partial *} (y') 
e_{\mathrm{st}*}\circ e_{\mathrm{st} \partial *} (z')
e_{\mathrm{st}*}\circ e_{\mathrm{st} \partial *} (x') \\
&+
2 e_{\mathrm{st}*}\circ e_{\mathrm{st} \partial *} (y')
e_{\mathrm{st}*}\circ e_{\mathrm{st} \partial *}(x')
e_{\mathrm{st}*}\circ e_{\mathrm{st} \partial *}(z')
\in h^2 \tskein' (\surface_\mathrm{st})
\end{align*}
for $x',y',z' \in \tskein' (\tilde{\surface}_\mathrm{st})$.
 We can prove the equations under the assumption of the (Step 3) using that 
$e''_* :\tskein' (\surface') \to \tskein' (\tilde{\surface}_\mathrm{st})$
 is an algebra isomorphism.


\end{proof}


To state an accurate formula, we use the following notation.


\begin{df}

Let $K$ be a null homologous unoriented framed knot in $\surface \times I$ satisfying $w(K)=0$  and $e'_K : S_K \to\surface \times I$ a Seifert surface of $K$. We choose an element $\gamma_\partial \in \pi_1 (S_K)$ whose conjugacy class is $\partial S_K$. We can extend $e'_K$ to $e_K:S_K \times I \to \surface \times I$ such that $e_K( \cdot, 0)=e'_K (\cdot)$. We use the notation
$L_{\tskein'} (c_\partial ) \defeq  \PPsi{\GL}{\tskein'} (
L_{\GL} (c_\partial ))
=\PPsi{\GL}{\tskein'} (
\zettaiti{\frac{1}{2} (\log \gamma_\partial )^2})$.
 Then we set $L_1(K), L_2(K), L_3(K) \in \cGL (\surface)$ as 
\begin{align*}
L_1 (K) \defeq & \PPsi{\tskein'}{\GL}
(e_{K*} (L_{\tskein'} (c_\partial ))), \\
L_2 (K) \defeq & \PPsi{\tskein'}{\GL} (\frac{1}{2h}
(e_{K*} ((L_{\tskein'} (c_\partial))^2)-
(e_{K*} (L_{\tskein'} (c_\partial )))^2)), \\
L_3 (K) \defeq & \PPsi{\tskein'}{\GL} (
\frac{1}{12h^2} (2 e_{K*} ((L_{\tskein'} (c_\partial))^3)
-3e_{K*} ((L_{\tskein'} (c_\partial))^2)e_{K*} (L_{\tskein'} (c_\partial)) \\
&-3e_{K*} (L_{\tskein'} (c_\partial))e_{K*} ((L_{\tskein'} (c_\partial))^2)
+4(e_{K*} ((L_{\tskein'} (c_\partial))))^3) ).\\
\end{align*}

\end{df}


Using the notation, we state the theorem.


\begin{thm}
\label{thm_KMT}

In the situation of Theorem \ref{thm_KM}, we have 
\begin{align*}
&\tilde{\zeta}_{\GL} ((\surface \times I)(K(-\epsilon))) \\
&=\epsilon L_1 (K) + L_2 (K) +\epsilon L_3 (k) \mod 
\filt{2m+6} \cGL (\surface),
\end{align*}
 where $L_1 (K)$, $L_2 (K)$, and $L_3 (K)$ are elements of $\filt{2m} \cGL (\surface)$, $\filt{2m+2} \cGL (\surface)$, and $\filt{2m+4} \cGL (\surface)$, respectively. In other words, for $\star_1, \star_2 \in \partial \surface$, we have 
\begin{align*}
&((\surface \times I) (K(-\epsilon))_*
= \id  \\
&: \GLM (\surface,\star_1, \star_2)/
\filt{2m-1} \GLM (\surface, \star_1, \star_2)
\to
\GLM (\surface, \star_1, \star_2)/
\filt{2m-1} \GLM (\surface, \star_1, \star_2), \\
&((\surface \times I) (K(-\epsilon))_*
= \exp (\sigma (\epsilon L_1 (K)))  \\
&: \GLM (\surface, \star_1, \star_2)/
\filt{2m+1} \GLM (\surface,\star_1, \star_2)
\to
\GLM (\surface, \star_1, \star_2)/
\filt{2m+1} \GLM (\surface, \star_1, \star_2),  \\
&((\surface \times I) (K(-\epsilon))_*
= \exp (\sigma (\epsilon L_1 (K)+L_2(K))) \\
&: \GLM (\surface, \star_1, \star_2)/
\filt{2m+3} \GLM (\surface, \star_1, \star_2)
\to
\GLM (\surface, \star_1, \star_2)/
\filt{2m+3} \GLM (\surface, \star_1, \star_2), \\
&((\surface \times I) (K(-\epsilon))_*
= \exp (\sigma (\epsilon L_1 (K) +L_2(K)+ \epsilon L_3 (K)))  \\
&: \GLM (\surface, \star_1, \star_2)/
\filt{2m+5} \GLM (\surface, \star_1, \star_2)
\to
\GLM (\surface, \star_1, \star_2)/
\filt{2m+5} \GLM (\surface, \star_1, \star_2). \\
\end{align*}


\end{thm}

\begin{proof}

Using the notation
$L^{(N)} \defeq e_{K*} ((L_{\tskein'} (c_\partial))^N)$,
 by the above computation, we obtain 
\begin{align*}
&\PPsi{\GL}{\Uh} (\tilde{\zeta}_{\GL} ((\surface \times I)(K(-\epsilon)))) \\
&=\PPsi{\tskein'}{\Uh} (h\log(1+\frac{\epsilon}{h}L^{(1)}+\frac{\epsilon^2}{2h^2}L^{(2)}+\frac{\epsilon^3}{6h^3}L^{(3)}))
\mod \filt{2m+6}\cGL (\surface) \\
&=\PPsi{\tskein'}{\Uh} (\epsilon L^{(1)} +\frac{1}{2h} (L^{(2)}-L^{(1)2}) 
+\frac{\epsilon}{12h^2}(2L^{(3)}-3L^{(2)}L^{(1)}-3L^{(1)}L^{(2)}+4L^{(1)3})) \\
&\mod \PPsi{\tskein'}{\Uh} (\filt{2m+6} \cLoc \tskein' (\surface)) \\
&=\PPsi{\GL}{\Uh} (\epsilon L_1 (K) + L_2 (K) +\epsilon L_3 (k)) \mod 
\filt{2m+6} \cGL (\surface).
\end{align*} 
By Corollary \ref{cor_qbtskein_filtration}, we have 
$L_1 (K ) \in \filt{2m} \cGL (\surface)$,
$L_2 (K) \in \filt{2m+2} \cGL (\surface)$,
and $L_3 (K) \in \filt{2m+4} \cGL (\surface)$. It proves the first formula. Since $\sigma (\filt{N} \cGL (\surface))(\GLM (\surface, \star))
\subset \filt{N-1} \cGLM (\surface,\star)$ for any $N \in \Zlarger{0}$, we have 
\begin{align*}
&((\surface \times I) (K(-\epsilon))_*
= \id  \\
&: \GLM (\surface,\star_1, \star_2)/
\filt{2m-1} \GLM (\surface, \star_1, \star_2)
\to
\GLM (\surface, \star_1, \star_2)/
\filt{2m-1} \GLM (\surface, \star_1, \star_2), \\
&((\surface \times I) (K(-\epsilon))_*
= \exp (\sigma (\epsilon L_1 (K)))  \\
&: \GLM (\surface, \star_1, \star_2)/
\filt{2m+1} \GLM (\surface,\star_1, \star_2)
\to
\GLM (\surface, \star_1, \star_2)/
\filt{2m+1} \GLM (\surface, \star_1, \star_2),  \\
&((\surface \times I) (K(-\epsilon))_*
= \exp (\sigma (\epsilon L_1 (K)+L_2(K))) \\
&: \GLM (\surface, \star_1, \star_2)/
\filt{2m+3} \GLM (\surface, \star_1, \star_2)
\to
\GLM (\surface, \star_1, \star_2)/
\filt{2m+3} \GLM (\surface, \star_1, \star_2), \\
&((\surface \times I) (K(-\epsilon))_*
= \exp (\sigma (\epsilon L_1 (K) +L_2(K)+ \epsilon L_3 (K)))  \\
&: \GLM (\surface, \star_1, \star_2)/
\filt{2m+5} \GLM (\surface, \star_1, \star_2)
\to
\GLM (\surface, \star_1, \star_2)/
\filt{2m+5} \GLM (\surface, \star_1, \star_2) \\
\end{align*}
 The above equations prove the theorem.

\end{proof}

\section{An application}
\label{section_an_application_KMT}
%
%
%
%
%

In this section, we will use Theorem \ref{thm_KMT} in a concrete example. Let $\surface_{1,1}$ be a compact connected oriented surface of genus one having a connected boundary. We fix a generating set $\shuugou{\gamma_\alpha, \gamma_\beta}$ of $\pi_1 (\surface_{1,1})$ as the figure and denote the element $\gamma_\alpha \gamma_\beta 
\gamma_\alpha^{-1} \gamma_\beta^{-1}$ by $\gamma_\partial$.

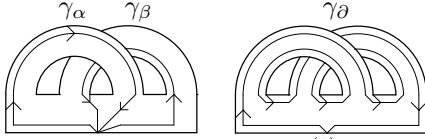
\begin{picture}(200,60)
\put(0,0){
{\unitlength 0.1in%
\begin{picture}(22.0200,5.6200)(0.8000,-7.6200)%
%
\special{pn 8}%
\special{ar 440 560 360 360 3.1415927 6.2831853}%
%
\special{pn 8}%
\special{ar 440 560 200 200 3.1415927 6.2831853}%
%
\special{pn 8}%
\special{ar 720 560 200 200 4.8120576 6.2831853}%
%
\special{pn 8}%
\special{ar 720 560 200 200 3.1415927 3.9269908}%
%
\special{pn 8}%
\special{ar 720 560 360 360 4.3175979 6.2831853}%
%
\special{pn 8}%
\special{ar 720 560 360 360 3.1415927 3.7295953}%
%
\special{pn 8}%
\special{pa 360 560}%
\special{pa 360 560}%
\special{fp}%
\special{pa 240 560}%
\special{pa 360 560}%
\special{fp}%
\special{pa 520 560}%
\special{pa 640 560}%
\special{fp}%
\special{pa 800 560}%
\special{pa 920 560}%
\special{fp}%
\special{pa 80 560}%
\special{pa 80 760}%
\special{fp}%
\special{pa 80 760}%
\special{pa 1080 760}%
\special{fp}%
\special{pa 1080 760}%
\special{pa 1080 560}%
\special{fp}%
\special{pa 1080 560}%
\special{pa 1080 560}%
\special{fp}%
%
\special{pn 4}%
\special{ar 440 560 320 320 3.1415927 6.2831853}%
%
\special{pn 4}%
\special{pa 560 760}%
\special{pa 560 760}%
\special{fp}%
\special{pa 560 760}%
\special{pa 520 720}%
\special{fp}%
\special{pa 520 720}%
\special{pa 120 720}%
\special{fp}%
\special{pa 120 720}%
\special{pa 120 560}%
\special{fp}%
\special{pa 560 760}%
\special{pa 760 560}%
\special{fp}%
%
\special{pn 4}%
\special{ar 720 560 240 240 4.6410815 6.2831853}%
%
\special{pn 4}%
\special{ar 720 560 240 240 3.1415927 3.9269908}%
%
\special{pn 4}%
\special{pa 560 760}%
\special{pa 560 760}%
\special{fp}%
\special{pa 560 760}%
\special{pa 680 720}%
\special{fp}%
\special{pa 680 720}%
\special{pa 960 720}%
\special{fp}%
\special{pa 960 720}%
\special{pa 960 560}%
\special{fp}%
\special{pa 560 760}%
\special{pa 560 640}%
\special{fp}%
\special{pa 560 640}%
\special{pa 480 560}%
\special{fp}%
\special{pa 440 240}%
\special{pa 400 200}%
\special{fp}%
\special{pa 440 240}%
\special{pa 400 280}%
\special{fp}%
%
\special{pn 4}%
\special{pa 960 560}%
\special{pa 960 560}%
\special{fp}%
\special{pa 960 560}%
\special{pa 920 600}%
\special{fp}%
\special{pa 960 560}%
\special{pa 1000 600}%
\special{fp}%
\special{pa 520 600}%
\special{pa 520 560}%
\special{fp}%
\special{pa 520 600}%
\special{pa 480 600}%
\special{fp}%
\special{pa 120 600}%
\special{pa 80 640}%
\special{fp}%
\special{pa 120 600}%
\special{pa 160 640}%
\special{fp}%
\special{pa 680 640}%
\special{pa 680 600}%
\special{fp}%
\special{pa 680 640}%
\special{pa 720 640}%
\special{fp}%
%
\special{pn 8}%
\special{ar 1642 562 360 360 3.1415927 6.2831853}%
%
\special{pn 8}%
\special{ar 1642 562 200 200 3.1415927 6.2831853}%
%
\special{pn 8}%
\special{ar 1922 562 200 200 4.8120576 6.2831853}%
%
\special{pn 8}%
\special{ar 1922 562 200 200 3.1415927 3.9269908}%
%
\special{pn 8}%
\special{ar 1922 562 360 360 4.3175979 6.2831853}%
%
\special{pn 8}%
\special{ar 1922 562 360 360 3.1415927 3.7295953}%
%
\special{pn 8}%
\special{pa 1562 562}%
\special{pa 1562 562}%
\special{fp}%
\special{pa 1442 562}%
\special{pa 1562 562}%
\special{fp}%
\special{pa 1722 562}%
\special{pa 1842 562}%
\special{fp}%
\special{pa 2002 562}%
\special{pa 2122 562}%
\special{fp}%
\special{pa 1282 562}%
\special{pa 1282 762}%
\special{fp}%
\special{pa 1282 762}%
\special{pa 2282 762}%
\special{fp}%
\special{pa 2282 762}%
\special{pa 2282 562}%
\special{fp}%
\special{pa 2282 562}%
\special{pa 2282 562}%
\special{fp}%
%
\special{pn 4}%
\special{ar 1922 562 240 240 4.6410815 6.2831853}%
%
\special{pn 4}%
\special{ar 1922 562 240 240 3.1415927 3.9269908}%
%
\special{pn 4}%
\special{ar 1642 562 320 320 3.1415927 6.2831853}%
%
\special{pn 4}%
\special{ar 1642 562 240 240 3.1415927 6.2831853}%
%
\special{pn 4}%
\special{ar 1922 562 320 320 4.4340893 6.2831853}%
%
\special{pn 4}%
\special{ar 1922 562 320 320 3.1415927 3.8163336}%
%
\special{pn 4}%
\special{pa 1762 762}%
\special{pa 1762 762}%
\special{fp}%
\special{pa 1762 762}%
\special{pa 1722 722}%
\special{fp}%
\special{pa 1722 722}%
\special{pa 1322 722}%
\special{fp}%
\special{pa 1322 722}%
\special{pa 1322 562}%
\special{fp}%
\special{pa 1402 562}%
\special{pa 1442 602}%
\special{fp}%
\special{pa 1442 602}%
\special{pa 1562 602}%
\special{fp}%
\special{pa 1562 602}%
\special{pa 1602 562}%
\special{fp}%
\special{pa 1682 562}%
\special{pa 1722 602}%
\special{fp}%
\special{pa 1722 602}%
\special{pa 1842 602}%
\special{fp}%
\special{pa 1842 602}%
\special{pa 1882 562}%
\special{fp}%
\special{pa 1962 562}%
\special{pa 2002 602}%
\special{fp}%
\special{pa 2002 602}%
\special{pa 2122 602}%
\special{fp}%
\special{pa 2122 602}%
\special{pa 2162 562}%
\special{fp}%
\special{pa 2242 562}%
\special{pa 2242 722}%
\special{fp}%
\special{pa 2242 722}%
\special{pa 1802 722}%
\special{fp}%
\special{pa 1802 722}%
\special{pa 1762 762}%
\special{fp}%
%
\special{pn 4}%
\special{pa 1322 642}%
\special{pa 1322 642}%
\special{fp}%
\special{pa 1322 642}%
\special{pa 1282 682}%
\special{fp}%
\special{pa 1322 642}%
\special{pa 1362 682}%
\special{fp}%
\special{pa 2042 602}%
\special{pa 2002 562}%
\special{fp}%
\special{pa 2042 602}%
\special{pa 2002 642}%
\special{fp}%
\special{pa 1762 602}%
\special{pa 1722 562}%
\special{fp}%
\special{pa 1762 602}%
\special{pa 1722 642}%
\special{fp}%
\special{pa 1482 602}%
\special{pa 1442 562}%
\special{fp}%
\special{pa 1482 602}%
\special{pa 1442 642}%
\special{fp}%
\special{pa 2242 642}%
\special{pa 2242 642}%
\special{fp}%
\special{pa 2242 642}%
\special{pa 2202 602}%
\special{fp}%
\special{pa 2242 642}%
\special{pa 2282 602}%
\special{fp}%
\end{picture}}%
}
\put(20,45){$\gamma_\alpha$}
\put(45,45){$\gamma_\beta$}
\put(120,45){$\gamma_\partial$}
\end{picture}

 We choose an embedding $e^{(1)}_{\surface_{1,1}}: \surface_{1,1} \times I 
\to \surface_{1, 1} \times I$ satisfying the two conditions. 
\begin{itemize}
\item
It is an embedding shown in the figure. 
\item
The induced map $\GL (\surface) \to \GL (\surface)$
is an isomorphism.
\end{itemize}

{\unitlength 0.1in%
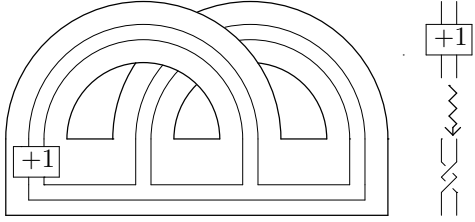
\begin{picture}(24.4000,11.2000)(0.8000,-13.2000)%
%
\special{pn 8}%
\special{ar 800 920 720 720 3.1415927 6.2831853}%
%
\special{pn 8}%
\special{ar 800 920 400 400 3.1415927 6.2831853}%
%
\special{pn 8}%
\special{ar 1360 920 400 400 4.8120576 6.2831853}%
%
\special{pn 8}%
\special{ar 1360 920 400 400 3.1415927 3.9269908}%
%
\special{pn 8}%
\special{ar 1360 920 720 720 4.3175979 6.2831853}%
%
\special{pn 8}%
\special{ar 1360 920 720 720 3.1415927 3.7295953}%
%
\special{pn 8}%
\special{pa 640 920}%
\special{pa 640 920}%
\special{fp}%
\special{pa 400 920}%
\special{pa 640 920}%
\special{fp}%
\special{pa 960 920}%
\special{pa 1200 920}%
\special{fp}%
\special{pa 1520 920}%
\special{pa 1760 920}%
\special{fp}%
\special{pa 80 920}%
\special{pa 80 1320}%
\special{fp}%
\special{pa 80 1320}%
\special{pa 2080 1320}%
\special{fp}%
\special{pa 2080 1320}%
\special{pa 2080 920}%
\special{fp}%
\special{pa 2080 920}%
\special{pa 2080 920}%
\special{fp}%
%
\special{pn 4}%
\special{ar 800 920 600 600 3.1415927 6.2831853}%
%
\special{pn 4}%
\special{ar 1360 920 600 600 4.4889124 6.2831853}%
%
\special{pn 4}%
\special{ar 1360 920 600 600 3.1415927 3.8414855}%
%
\special{pn 4}%
\special{ar 1360 920 520 520 4.6017318 6.2831853}%
%
\special{pn 4}%
\special{ar 1360 920 520 520 3.1415927 3.8966971}%
%
\special{pn 4}%
\special{pa 1880 920}%
\special{pa 1880 920}%
\special{fp}%
\special{pa 1880 920}%
\special{pa 1880 1160}%
\special{fp}%
\special{pa 1960 920}%
\special{pa 1960 1160}%
\special{fp}%
\special{pa 1400 1160}%
\special{pa 1400 920}%
\special{fp}%
\special{pa 1320 920}%
\special{pa 1320 1160}%
\special{fp}%
\special{pa 760 920}%
\special{pa 760 1160}%
\special{fp}%
\special{pa 840 920}%
\special{pa 840 1160}%
\special{fp}%
%
\special{pn 4}%
\special{ar 800 920 520 520 3.1415927 6.2831853}%
%
\special{pn 4}%
\special{pa 160 960}%
\special{pa 160 960}%
\special{fp}%
\special{pa 120 960}%
\special{pa 360 960}%
\special{fp}%
\special{pa 360 960}%
\special{pa 360 1120}%
\special{fp}%
\special{pa 360 1120}%
\special{pa 120 1120}%
\special{fp}%
\special{pa 120 1120}%
\special{pa 120 960}%
\special{fp}%
\special{pa 200 920}%
\special{pa 200 960}%
\special{fp}%
\special{pa 280 920}%
\special{pa 280 960}%
\special{fp}%
\special{pa 200 1120}%
\special{pa 200 1160}%
\special{fp}%
\special{pa 280 1120}%
\special{pa 280 1120}%
\special{fp}%
\special{pa 280 1160}%
\special{pa 280 1160}%
\special{fp}%
\special{pa 280 1160}%
\special{pa 280 1120}%
\special{fp}%
\special{pa 280 1160}%
\special{pa 760 1160}%
\special{fp}%
\special{pa 840 1160}%
\special{pa 1320 1160}%
\special{fp}%
\special{pa 1400 1160}%
\special{pa 1880 1160}%
\special{fp}%
\special{pa 1960 1160}%
\special{pa 1960 1240}%
\special{fp}%
\special{pa 200 1240}%
\special{pa 200 1160}%
\special{fp}%
\special{pa 200 1240}%
\special{pa 1960 1240}%
\special{fp}%
\put(1.6000,-10.8000){\makebox(0,0)[lb]{+1}}%
%
\special{pn 4}%
\special{pa 2160 480}%
\special{pa 2160 480}%
\special{fp}%
\special{pa 2280 480}%
\special{pa 2280 320}%
\special{fp}%
\special{pa 2280 320}%
\special{pa 2520 320}%
\special{fp}%
\special{pa 2520 320}%
\special{pa 2520 480}%
\special{fp}%
\special{pa 2520 480}%
\special{pa 2280 480}%
\special{fp}%
\special{pa 2360 480}%
\special{pa 2360 600}%
\special{fp}%
\special{pa 2440 600}%
\special{pa 2440 480}%
\special{fp}%
\special{pa 2360 320}%
\special{pa 2360 200}%
\special{fp}%
\special{pa 2440 200}%
\special{pa 2440 320}%
\special{fp}%
\put(23.2000,-4.4000){\makebox(0,0)[lb]{+1}}%
%
\special{pn 8}%
\special{pa 2400 640}%
\special{pa 2400 640}%
\special{fp}%
\special{pa 2400 640}%
\special{pa 2440 680}%
\special{fp}%
\special{pa 2440 680}%
\special{pa 2400 720}%
\special{fp}%
\special{pa 2400 720}%
\special{pa 2440 760}%
\special{fp}%
\special{pa 2440 760}%
\special{pa 2400 800}%
\special{fp}%
\special{pa 2400 800}%
\special{pa 2440 840}%
\special{fp}%
\special{pa 2420 860}%
\special{pa 2440 840}%
\special{fp}%
\special{pa 2420 860}%
\special{pa 2420 900}%
\special{fp}%
\special{pa 2420 900}%
\special{pa 2380 860}%
\special{fp}%
\special{pa 2420 900}%
\special{pa 2460 860}%
\special{fp}%
%
\special{pn 4}%
\special{pa 2360 920}%
\special{pa 2360 920}%
\special{fp}%
\special{pa 2360 920}%
\special{pa 2360 1040}%
\special{fp}%
\special{pa 2440 1040}%
\special{pa 2440 920}%
\special{fp}%
\special{pa 2360 1200}%
\special{pa 2360 1320}%
\special{fp}%
\special{pa 2440 1200}%
\special{pa 2440 1320}%
\special{fp}%
%
\special{pn 4}%
\special{pa 2440 1200}%
\special{pa 2440 1200}%
\special{fp}%
\special{pa 2360 1200}%
\special{pa 2440 1120}%
\special{fp}%
\special{pa 2360 1120}%
\special{pa 2440 1040}%
\special{fp}%
\special{pa 2440 1200}%
\special{pa 2420 1180}%
\special{fp}%
\special{pa 2380 1140}%
\special{pa 2360 1120}%
\special{fp}%
\special{pa 2440 1120}%
\special{pa 2420 1100}%
\special{fp}%
\special{pa 2380 1060}%
\special{pa 2360 1040}%
\special{fp}%
\end{picture}}%



To apply Theorem \ref{thm_KMT}, we need the lemma.


\begin{lemm}
\label{lemm_compute_KMT}

Using the above notation, we have 
\begin{align*}
&e^{(1)}_{\surface_{1,1}}((\PPsi{\GL}{\tskein'}(\zettaiti{\gamma_\partial-2+\gamma_\partial^{-1}}))^2)
-
(\PPsi{\GL}{\tskein'}(\zettaiti{\gamma_\partial-2+\gamma_\partial^{-1}}))^2\\
&=2h \PPsi{\GL}{\tskein'} (\zettaiti{
(\gamma_\partial-1)(\gamma_\partial^{-1}-1)
(\gamma_\partial -1)({\gamma'_\beta}^{-1}-1) {\gamma'_\beta} \\
&-
(\gamma_\partial-1)(\gamma_\partial^{-1}-1)
({\gamma'_\beta}^{-1}-1)(\gamma_\partial-1) {\gamma'_\beta} \\
&-
(\gamma_\partial-1)(\gamma_\partial^{-1}-1){\gamma'_\beta}^{-1} 
(\gamma_\partial^{-1}-1)( {\gamma'_\beta}-1) \\
&+
(\gamma_\partial-1)(\gamma_\partial^{-1}-1){\gamma'_\beta}^{-1}
({\gamma'_\beta} -1)(\gamma_\partial^{-1}-1)) )
}) \\
&+4h \PPsi{\GL}{\tskein'} (\zettaiti{
(\gamma_\partial-1)(\gamma_\partial -1)({\gamma'_\beta}^{-1}-1)({\gamma'_\beta}-1)  \\
&- 
(\gamma_\partial-1)({\gamma'_\beta}^{-1}-1)( \gamma_\partial-1)({\gamma'_\beta}-1) \\
&- 
(\gamma_\partial^{-1}-1)({\gamma'_\beta}^{-1}-1)(\gamma_\partial^{-1}-1)( {\gamma'_\beta}-1) \\
&+
(\gamma_\partial^{-1}-1)({\gamma'_\beta}^{-1}-1)({\gamma'_\beta} -1)(\gamma_\partial^{-1}-1)}).
 \\
\end{align*}

\end{lemm}

\begin{proof}
%
%
%
%

Using the notation
$l \defeq \PPsi{\GL}{\tskein'} (\zettaiti{\gamma_\partial} )$,
$l' \defeq \PPsi{\GL}{\tskein'}(\zettaiti{ \gamma_\partial^{-1}})$,
 we have 
\begin{align*}
&e^{(1)}_{\surface_{1,1}} ((l+l'+2)^2)
-((l+l'+2)^2) \\
&=(e^{(1)}_{\surface_{1,1}} (l^2)
-l^2)+2(e^{(1)}_{\surface_{1,1}} (ll')
-ll')+(e^{(1)}_{\surface_{1,1}} (l'^2)
-l'2).
\end{align*} 
We will compute the three elements
$(e^{(1)}_{\surface_{1,1}} (l^2)
-l^2)$, $(e^{(1)}_{\surface_{1,1}} (ll')
-ll')$, $(e^{(1)}_{\surface_{1,1}} (l'^2)
-l'2)$. In the proof, we denote by $L(X)$ the link presented by the link diagram. 

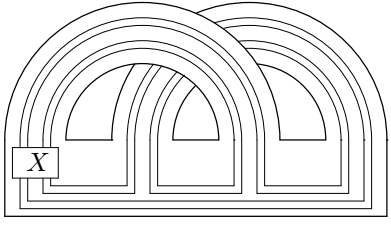
\begin{picture}(200,85)
\put(0,0){
{\unitlength 0.1in%
\begin{picture}(20.0000,11.2000)(0.8000,-13.2000)%
%
\special{pn 8}%
\special{ar 800 920 720 720 3.1415927 6.2831853}%
%
\special{pn 8}%
\special{ar 800 920 400 400 3.1415927 6.2831853}%
%
\special{pn 8}%
\special{ar 1360 920 400 400 4.8120576 6.2831853}%
%
\special{pn 8}%
\special{ar 1360 920 400 400 3.1415927 3.9269908}%
%
\special{pn 8}%
\special{ar 1360 920 720 720 4.3175979 6.2831853}%
%
\special{pn 8}%
\special{ar 1360 920 720 720 3.1415927 3.7295953}%
%
\special{pn 8}%
\special{pa 640 920}%
\special{pa 640 920}%
\special{fp}%
\special{pa 400 920}%
\special{pa 640 920}%
\special{fp}%
\special{pa 960 920}%
\special{pa 1200 920}%
\special{fp}%
\special{pa 1520 920}%
\special{pa 1760 920}%
\special{fp}%
\special{pa 80 920}%
\special{pa 80 1320}%
\special{fp}%
\special{pa 80 1320}%
\special{pa 2080 1320}%
\special{fp}%
\special{pa 2080 1320}%
\special{pa 2080 920}%
\special{fp}%
\special{pa 2080 920}%
\special{pa 2080 920}%
\special{fp}%
%
\special{pn 4}%
\special{ar 800 920 600 600 3.1415927 6.2831853}%
%
\special{pn 4}%
\special{ar 1360 920 600 600 4.4889124 6.2831853}%
%
\special{pn 4}%
\special{ar 1360 920 600 600 3.1415927 3.8414855}%
%
\special{pn 4}%
\special{ar 1360 920 520 520 4.6017318 6.2831853}%
%
\special{pn 4}%
\special{ar 1360 920 520 520 3.1415927 3.8966971}%
%
\special{pn 4}%
\special{ar 800 920 520 520 3.1415927 6.2831853}%
%
\special{pn 4}%
\special{ar 800 920 640 640 3.1415927 6.2831853}%
%
\special{pn 4}%
\special{ar 800 920 480 480 3.1415927 6.2831853}%
%
\special{pn 4}%
\special{ar 1360 920 640 640 4.4209322 6.2831853}%
%
\special{pn 4}%
\special{ar 1360 920 640 640 3.1415927 3.8056388}%
%
\special{pn 4}%
\special{ar 1360 920 480 480 4.6689381 6.2831853}%
%
\special{pn 4}%
\special{ar 1360 920 480 480 3.1415927 3.9269908}%
%
\special{pn 4}%
\special{pa 120 960}%
\special{pa 120 960}%
\special{fp}%
\special{pa 120 960}%
\special{pa 360 960}%
\special{fp}%
\special{pa 360 960}%
\special{pa 360 1120}%
\special{fp}%
\special{pa 360 1120}%
\special{pa 120 1120}%
\special{fp}%
\special{pa 120 1120}%
\special{pa 120 960}%
\special{fp}%
%
\special{pn 4}%
\special{pa 160 960}%
\special{pa 160 960}%
\special{fp}%
\special{pa 160 920}%
\special{pa 160 960}%
\special{fp}%
\special{pa 200 920}%
\special{pa 200 960}%
\special{fp}%
\special{pa 280 920}%
\special{pa 280 960}%
\special{fp}%
\special{pa 320 960}%
\special{pa 320 920}%
\special{fp}%
\special{pa 320 1120}%
\special{pa 320 1160}%
\special{fp}%
\special{pa 280 1120}%
\special{pa 280 1200}%
\special{fp}%
\special{pa 280 1200}%
\special{pa 760 1200}%
\special{fp}%
\special{pa 760 1200}%
\special{pa 760 920}%
\special{fp}%
\special{pa 720 920}%
\special{pa 720 1160}%
\special{fp}%
\special{pa 320 1160}%
\special{pa 720 1160}%
\special{fp}%
\special{pa 840 920}%
\special{pa 840 1200}%
\special{fp}%
\special{pa 840 1200}%
\special{pa 1320 1200}%
\special{fp}%
\special{pa 1320 1200}%
\special{pa 1320 920}%
\special{fp}%
\special{pa 1280 920}%
\special{pa 1280 920}%
\special{fp}%
\special{pa 1280 920}%
\special{pa 1280 1160}%
\special{fp}%
\special{pa 1280 1160}%
\special{pa 880 1160}%
\special{fp}%
\special{pa 880 1160}%
\special{pa 880 920}%
\special{fp}%
\special{pa 1400 920}%
\special{pa 1400 1200}%
\special{fp}%
\special{pa 1400 1200}%
\special{pa 1880 1200}%
\special{fp}%
\special{pa 1880 1200}%
\special{pa 1880 920}%
\special{fp}%
\special{pa 1840 920}%
\special{pa 1840 920}%
\special{fp}%
\special{pa 1840 920}%
\special{pa 1840 1160}%
\special{fp}%
\special{pa 1840 1160}%
\special{pa 1440 1160}%
\special{fp}%
\special{pa 1440 1160}%
\special{pa 1440 920}%
\special{fp}%
\special{pa 1960 920}%
\special{pa 1960 920}%
\special{fp}%
\special{pa 1960 920}%
\special{pa 1960 1240}%
\special{fp}%
\special{pa 2000 920}%
\special{pa 2000 1280}%
\special{fp}%
\special{pa 2000 1280}%
\special{pa 160 1280}%
\special{fp}%
\special{pa 160 1280}%
\special{pa 160 1120}%
\special{fp}%
\special{pa 200 1120}%
\special{pa 200 1240}%
\special{fp}%
\special{pa 200 1240}%
\special{pa 1960 1240}%
\special{fp}%
\end{picture}}%
}
\put(8,17){$X$}
\end{picture}

Here $X$ is one of the following.

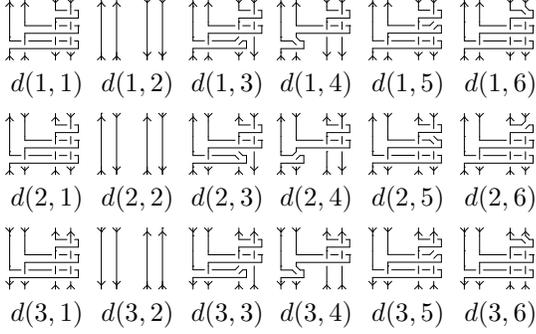
\input{fig_lemm_compute_KMT_2}
\bigskip

Using the notation
$x=\gamma_\alpha, y=\gamma_\beta$,
 we have 
\begin{align*}
&(e^{(1)}_{\surface_{1,1}} (l^2)
-l^2) \\
&=L(d (1,1))-L(d(1,2))
=h(-L(d(1,3))+L(d(1,4))-L(d(1,5))+L(d(1,6))) \\
&=h\PPsi{\GL}{\tskein'} (\zettaiti{-(xyx^{-1})^2y^{-2}+
xyx^{-1}y^{-1}xyx^{-1}y^{-1}-(xyx^{-1})^2y^{-2}+
xyx^{-1}y^{-1}xyx^{-1}y^{-1}}) \\
&=2h\PPsi{\GL}{\tskein'} (\zettaiti{-(xyx^{-1})^2y^{-2}+
xyx^{-1}y^{-1}xyx^{-1}y^{-1}}),
\end{align*}

\begin{align*}
&(e^{(1)}_{\surface_{1,1}} (ll')
-ll') \\
&=L(d (2,1))-L(d(2,2))
=h(L(d(2,3))-L(d(2,4))+L(d(2,5))-L(d(2,6))) \\
&=h\PPsi{\GL}{\tskein'} (\zettaiti{
xyx^{-1}y^{-1}xy^{-1}x^{-1} y-1+
xyx^{-1}y^{-1}xy^{-1}x^{-1} y-1})
 \\
&=2h\PPsi{\GL}{\tskein'} (\zettaiti{xyx^{-1}y^{-1}xy^{-1}x^{-1} y-1}),
\end{align*}

\begin{align*}
&(e^{(1)}_{\surface_{1,1}} (l'^2)
-l'^2) \\
&=L(d (3,1))-L(d(3,2))
=h(-L(d(3,3))+L(d(3,4))-L(d(3,5))+L(d(3,6))) \\
&=h\PPsi{\GL}{\tskein'} (\zettaiti{-(xyx^{-1})^{-2}y^{2}+
yxy^{-1}x^{-1}yxy^{-1}x^{-1}-(xyx^{-1})^{-2}y^{2}+
yxy^{-1}x^{-1}yxy^{-1}x^{-1}}) \\
&=2h\PPsi{\GL}{\tskein'} (\zettaiti{-(xyx^{-1})^{-2}y^{2}+
yxy^{-1}x^{-1}yxy^{-1}x^{-1}}).
\end{align*}
 So we obtain 
\begin{align*}
&e^{(1)}_{\surface_{1,1}} ((l+l'+2)^2)
-((l+l'+2)^2) \\
&=(e^{(1)}_{\surface_{1,1}} (l^2)
-l^2)+2(e^{(1)}_{\surface_{1,1}} (ll')
-ll')+(e^{(1)}_{\surface_{1,1}} (l'^2)
-l'2) \\
&=2h \PPsi{\GL}{\tskein'} (\zettaiti{
y'y^{-1} (y' y^{-1}-y^{-1} y'+y'^{-1} y-y y'^{-1}) \\
&+y y'^{-1} (y y'^{-1}-y'^{-1} y+y^{-1} y'-y' y^{-1})}) \\
&=2h \PPsi{\GL}{\tskein'} (\zettaiti{
(y'y^{-1}-y y'^{-1}) (y' y^{-1}-y^{-1} y'+y'^{-1} y-y y'^{-1})}), \\
\end{align*} 
where $y' =xyx^{-1}$.
Furthermore, using 
\begin{align*}
&y' y^{-1}-y^{-1} y'+y'^{-1} y-y y'^{-1} \\
&=(y' y^{-1} y'^{-1} -y'^{-1} y' y^{-1}) y'
+y'^{-1} (y y'^{-1} y'-y' y y'^{-1}) \\
&=((y'y^{-1} -1)(y'^{-1}-1) -(y'^{-1}-1)( y' y^{-1}-1)) y' \\
&+y'^{-1} ((y y'^{-1}-1)( y'-1)-(y' -1)(y y'^{-1}-1)), \\
\end{align*}
\begin{align*}
&y'y^{-1}-y y'^{-1} \\
&=(y' y^{-1}-1)(y y'^{-1}-1)+2y'y^{-1}-2 \\
&=-(y' y^{-1}-1)(y y'^{-1}-1)-2y y'^{-1}+2,
\end{align*} 
we obtain 
\begin{align*}
&2h \PPsi{\GL}{\tskein'} (\zettaiti{
(y'y^{-1}-y y'^{-1}) (y' y^{-1}-y^{-1} y'+y'^{-1} y-y y'^{-1})}) \\
&=2h \PPsi{\GL}{\tskein'} (\zettaiti{
(y'y^{-1}-y y'^{-1}) (((y'y^{-1} -1)(y'^{-1}-1) -(y'^{-1}-1)( y' y^{-1}-1)) y' \\
&+y'^{-1} ((y y'^{-1}-1)( y'-1)-(y' -1)(y y'^{-1}-1)) )}) \\
&=2h \PPsi{\GL}{\tskein'} (\zettaiti{
(y'y^{-1}-1)(y y'^{-1}-1) (((y'y^{-1} -1)(y'^{-1}-1) -(y'^{-1}-1)( y' y^{-1}-1)) y' \\
&-y'^{-1} ((y y'^{-1}-1)( y'-1)-(y' -1)(y y'^{-1}-1)) )}) \\
&+4h \PPsi{\GL}{\tskein'} (\zettaiti{
(y'y^{-1}-1) ((y'y^{-1} -1)(y'^{-1}-1) -(y'^{-1}-1)( y' y^{-1}-1)) y' \\
&-(y y'^{-1}-1)y'^{-1} ((y y'^{-1}-1)( y'-1)-(y' -1)(y y'^{-1}-1)) }) \\
&=2h \PPsi{\GL}{\tskein'} (\zettaiti{
(y'y^{-1}-1)(y y'^{-1}-1) (((y'y^{-1} -1)(y'^{-1}-1) -(y'^{-1}-1)( y' y^{-1}-1)) y' \\
&-y'^{-1} ((y y'^{-1}-1)( y'-1)-(y' -1)(y y'^{-1}-1)) )}) \\
&+4h \PPsi{\GL}{\tskein'} (\zettaiti{
(y'y^{-1}-1) ((y'y^{-1} -1)(y'^{-1}-1) -(y'^{-1}-1)( y' y^{-1}-1))  \\
&-(y y'^{-1}-1)((y y'^{-1}-1)( y'-1)-(y' -1)(y y'^{-1}-1)) }) \\
&+4h \PPsi{\GL}{\tskein'} (\zettaiti{
(y'y^{-1}-1) ((y'y^{-1} -1)(y'^{-1}-1) -(y'^{-1}-1)( y' y^{-1}-1)) (y'-1) \\
&-(y y'^{-1}-1)(y'^{-1}-1) ((y y'^{-1}-1)( y'-1)-(y' -1)(y y'^{-1}-1)) }) \\
\end{align*}
 Using
$\zettaiti{z_1 z_2 } =\zettaiti{z_2 z_1}$
for $z_1, z_2 \in \GLM (\surface)$, we have 
\begin{align*}
&e^{(1)}_{\surface_{1,1}}((\PPsi{\GL}{\tskein'}(\zettaiti{\gamma_\partial-2+\gamma_\partial^{-1}}))^2)
-
(\PPsi{\GL}{\tskein'}(\zettaiti{\gamma_\partial-2+\gamma_\partial^{-1}}))^2\\
&=2h \PPsi{\GL}{\tskein'} (\zettaiti{
(y'y^{-1}-1)(y y'^{-1}-1) (((y'y^{-1} -1)(y'^{-1}-1) -(y'^{-1}-1)( y' y^{-1}-1)) y' \\
&-y'^{-1} ((y y'^{-1}-1)( y'-1)-(y' -1)(y y'^{-1}-1)) )}) \\
&+4h \PPsi{\GL}{\tskein'} (\zettaiti{
(y'y^{-1}-1) ((y'y^{-1} -1)(y'^{-1}-1) -(y'^{-1}-1)( y' y^{-1}-1)) (y'-1) \\
&-(y y'^{-1}-1)(y'^{-1}-1) ((y y'^{-1}-1)( y'-1)-(y' -1)(y y'^{-1}-1)) }). \\
\end{align*}
 Recalling the notations
\begin{align*}
&y' y^{-1} =\gamma_\alpha \gamma_\beta \gamma_\alpha^{-1}
\gamma_\beta^{-1}= \gamma_\partial, \\
&y y'^{-1}= \gamma_\partial^{-1}, \\
&y=\gamma_\beta, \\
&y'=\gamma_\alpha \gamma_\beta \gamma_\alpha^{-1}={\gamma'_\beta},
\end{align*}
we get 
\begin{align*}
&e^{(1)}_{\surface_{1,1}}((\PPsi{\GL}{\tskein'}(\zettaiti{\gamma_\partial-2+\gamma_\partial^{-1}}))^2)
-
(\PPsi{\GL}{\tskein'}(\zettaiti{\gamma_\partial-2+\gamma_\partial^{-1}}))^2\\
&=2h \PPsi{\GL}{\tskein'} (\zettaiti{
(\gamma_\partial-1)(\gamma_\partial^{-1}-1) (((\gamma_\partial
 -1)({\gamma'_\beta}^{-1}-1) -({\gamma'_\beta}^{-1}-1)(\gamma_\partial-1)) {\gamma'_\beta} \\
&-{\gamma'_\beta}^{-1} ((\gamma_\partial^{-1}-1)( {\gamma'_\beta}-1)-({\gamma'_\beta} -1)(\gamma_\partial^{-1}-1)) )}) \\
&+4h \PPsi{\GL}{\tskein'} (\zettaiti{
(\gamma_\partial-1) ((\gamma_\partial -1)({\gamma'_\beta}^{-1}-1) -({\gamma'_\beta}^{-1}-1)
( \gamma_\partial-1)) ({\gamma'_\beta}-1) \\
&-(\gamma_\partial^{-1}-1)({\gamma'_\beta}^{-1}-1) 
((\gamma_\partial^{-1}-1)( {\gamma'_\beta}-1)-({\gamma'_\beta} -1)(\gamma_\partial^{-1}-1)) }) \\
\end{align*}
 Finally, we can expand it into 
\begin{align*}
&e^{(1)}_{\surface_{1,1}}((\PPsi{\GL}{\tskein'}(\zettaiti{\gamma_\partial-2+\gamma_\partial^{-1}}))^2)
-
(\PPsi{\GL}{\tskein'}(\zettaiti{\gamma_\partial-2+\gamma_\partial^{-1}}))^2\\
&=2h \PPsi{\GL}{\tskein'} (\zettaiti{
(\gamma_\partial-1)(\gamma_\partial^{-1}-1)
(\gamma_\partial -1)({\gamma'_\beta}^{-1}-1) {\gamma'_\beta} \\
&-
(\gamma_\partial-1)(\gamma_\partial^{-1}-1)
({\gamma'_\beta}^{-1}-1)(\gamma_\partial-1) {\gamma'_\beta} \\
&-
(\gamma_\partial-1)(\gamma_\partial^{-1}-1){\gamma'_\beta}^{-1} 
(\gamma_\partial^{-1}-1)( {\gamma'_\beta}-1) \\
&+
(\gamma_\partial-1)(\gamma_\partial^{-1}-1){\gamma'_\beta}^{-1}
({\gamma'_\beta} -1)(\gamma_\partial^{-1}-1)) )
}) \\
&+4h \PPsi{\GL}{\tskein'} (\zettaiti{
(\gamma_\partial-1)(\gamma_\partial -1)({\gamma'_\beta}^{-1}-1)({\gamma'_\beta}-1)  \\
&- 
(\gamma_\partial-1)({\gamma'_\beta}^{-1}-1)( \gamma_\partial-1)({\gamma'_\beta}-1) \\
&- 
(\gamma_\partial^{-1}-1)({\gamma'_\beta}^{-1}-1)(\gamma_\partial^{-1}-1)( {\gamma'_\beta}-1) \\
&+
(\gamma_\partial^{-1}-1)({\gamma'_\beta}^{-1}-1)({\gamma'_\beta} -1)(\gamma_\partial^{-1}-1)}),
 \\
\end{align*}
 as desired.

\end{proof}


Using the lemma, we will state the theorem.


\begin{thm}

We use the notation in Theorem \ref{thm_KM}. Let $e: \surface_{1,1} \times I \to \surface \times I$ be an embedding satisfying $e_* (\gamma_\alpha \gamma_\beta \gamma_\alpha^{-1} \gamma_\beta^{-1})
\in 1+ I_{\GLM (\surface \times I)}^N$, and $K_\partial$ the boundary knot in $\surface_{1,1} \times I$ whose Seifert surface is $\surface_{1,1} \times \shuugou{\frac{1}{2}}$, respectively. Then we have 
\begin{align*}
&\zeta_\GL ((\surface \times I) (e \circ e^{(1)}_{\surface_1,1}(K_\partial)(-\epsilon)))
=\zeta_\GL ((\surface \times I) (e (K_\partial)(-\epsilon))) \\
&= \epsilon \zettaiti{\frac{1}{2} (\log (e_* (\gamma_\partial))^2)}
\mod \filt{2N+2} \cGL (\surface), \\
&\zeta_\GL ((\surface \times I) (e \circ e^{(1)}_{\surface_1,1}(K_\partial)(-\epsilon)))
-\zeta_\GL ((\surface \times I) (e (K_\partial)(-\epsilon))) \\
&= \zettaiti{(e_* (\gamma_\partial)-1)(e_* (\gamma_\beta) -1)
(e_* (\gamma_\partial)-1)(e_* (\gamma_\beta) -1)
-(e_* (\gamma_\partial)-1)^2(e_* (\gamma_\beta) -1)^2
} \\
&\mod \filt{2N+3} \cGL (\surface). \\
\end{align*}
 In other words, for $\star_1, \star_2 \in \partial \surface$, we have 
\begin{align*}
& (\surface \times I) (e \circ e^{(1)}_{\surface_1,1}(K_\partial)(-\epsilon))_*
=(\surface \times I) (e(K_\partial)(-\epsilon))_*
=\exp (\sigma (\epsilon \zettaiti{\frac{1}{2} (\log (e_* (\gamma_\partial)))^2})) : \\
&\cGLM (\surface, \star_1, \star_2 )/\filt{2N+1} \cGLM (\surface, \star_1, \star_2)
 \to \cGLM (\surface, \star_1, \star_2) /\filt{2N+1} \cGLM (\surface, \star_1, \star_2) ,\\
& (\surface \times I) (e \circ e^{(1)}_{\surface_1,1}(K_\partial)(-\epsilon))_*
\circ((\surface \times I) (e(K_\partial)(-\epsilon))_*)^{-1} \\
&=
\exp (\sigma(\zettaiti{(e_* (\gamma_\partial)-1)(e_* (\gamma_\beta) -1)
(e_* (\gamma_\partial)-1)(e_* (\gamma_\beta) -1)
-(e_* (\gamma_\partial)-1)^2(e_* (\gamma_\beta) -1)^2
})): \\
&\cGLM (\surface, \star )/\filt{2N+2} \cGLM (\surface, \star)
 \to \cGLM (\surface, \star_1, \star_2) /\filt{2N+2} \cGLM (\surface, \star_1, \star_2). \\
\end{align*}

\end{thm}

\begin{proof}

The first formula and the third formula are special ones in Theorem \ref{thm_KM}. We will prove the second equation using Theorem \ref{thm_KMT}.

By Theorem \ref{thm_KMT}, we have 
\begin{align*}
&\zeta_\GL ((\surface \times I) (e \circ e^{(1)}_{\surface_1,1}(K_\partial)(-\epsilon)))
-\zeta_\GL ((\surface \times I) (e (K_\partial)(-\epsilon))) \\
&=(\epsilon L_1 (e \circ e^{(1)}_{\surface_1,1}(K_\partial))
+L_2 (e \circ e^{(1)}_{\surface_1,1}(K_\partial))) \\
&-(\epsilon L_1 (e (K_\partial))
+L_2 (e (K_\partial))) \mod \filt{2N+3} \cGL (\surface) \\
&=L_2 (e \circ e^{(1)}_{\surface_1,1}(K_\partial))
-L_2 (e (K_\partial)) \\
&= \PPsi{\tskein'}{\GL}(\frac{1}{2h}(
(e_* \circ e^{(1)}_{\surface_1,1 *} ((
\PPsi{\GL}{\tskein'} (
\zettaiti{\frac{1}{2} (\log \gamma_\partial )^2}))^2)) \\
&-
(e_* \circ e^{(1)}_{\surface_1,1 *} (
\PPsi{\GL}{\tskein'} (
\zettaiti{\frac{1}{2} (\log \gamma_\partial )^2})))^2 )\\
&-\PPsi{\tskein'}{\GL}(\frac{1}{2h}(
(e_*  ((
\PPsi{\GL}{\tskein'} (
\zettaiti{\frac{1}{2} (\log \gamma_\partial )^2}))^2))
-
(e_* (
\PPsi{\GL}{\tskein'} (
\zettaiti{\frac{1}{2} (\log \gamma_\partial )^2})))^2)). \\
\end{align*} 
Using 
\begin{equation*}
e_* \circ e^{(1)}_{\surface_1,1 *} (
\PPsi{\GL}{\tskein'} (
\zettaiti{\frac{1}{2} (\log \gamma_\partial )^2}))
=
e_*  (
\PPsi{\GL}{\tskein'} (
\zettaiti{\frac{1}{2} (\log \gamma_\partial )^2})),
\end{equation*}
 we have 
\begin{align*}
&\zeta_\GL ((\surface \times I) (e \circ e^{(1)}_{\surface_1,1}(K_\partial)(-\epsilon)))
-\zeta_\GL ((\surface \times I) (e (K_\partial)(-\epsilon))) \\
&= \PPsi{\tskein'}{\GL}(\frac{1}{2h}(
(e_* \circ e^{(1)}_{\surface_1,1 *} (
(\PPsi{\GL}{\tskein'} (
\zettaiti{\frac{1}{2} (\log \gamma_\partial )^2}))^2)))  \\
&-(e_* (
(\PPsi{\GL}{\tskein'} (
\zettaiti{\frac{1}{2} (\log \gamma_\partial )^2}))^2))) \mod
\filt{2N+3} \cGL (\surface). \\
\end{align*}
 Since $\frac{1}{2} (\log (X))^2 =\frac{1}{2} (X-2+X^{-1})$,
 by Corollary \ref{cor_qbtskein_filtration},
 we obtain
\begin{align*}
&\zeta_\GL ((\surface \times I) (e \circ e^{(1)}_{\surface_1,1}(K_\partial)(-\epsilon)))
-\zeta_\GL ((\surface \times I) (e (K_\partial)(-\epsilon))) \\
&=\PPsi{\tskein'}{\GL} (\frac{1}{8h}e_*(e^{(1)}_{\surface_{1,1}}((\PPsi{\GL}{\tskein'}(\zettaiti{\gamma_\partial-2+\gamma_\partial^{-1}}))^2)
-
(\PPsi{\GL}{\tskein'}(\zettaiti{\gamma_\partial-2+\gamma_\partial^{-1}}))^2)) \\
&\mod \filt{2N+3} \cGL (\surface).
\end{align*}

Using Lemma \ref{lemm_compute_KMT}, we continue the computation. By the lemma, we have 
\begin{align*}
&\zeta_\GL ((\surface \times I) (e \circ e^{(1)}_{\surface_1,1}(K_\partial)(-\epsilon)))
-\zeta_\GL ((\surface \times I) (e (K_\partial)(-\epsilon))) \\
&=\PPsi{\tskein'}{\GL} (\frac{1}{8}e_*( \\
&2h \PPsi{\GL}{\tskein'} (\zettaiti{
(\gamma_\partial-1)(\gamma_\partial^{-1}-1)
(\gamma_\partial -1)({\gamma'_\beta}^{-1}-1) {\gamma'_\beta} \\
&-
(\gamma_\partial-1)(\gamma_\partial^{-1}-1)
({\gamma'_\beta}^{-1}-1)(\gamma_\partial-1) {\gamma'_\beta} \\
&-
(\gamma_\partial-1)(\gamma_\partial^{-1}-1){\gamma'_\beta}^{-1} 
(\gamma_\partial^{-1}-1)( {\gamma'_\beta}-1) \\
&+
(\gamma_\partial-1)(\gamma_\partial^{-1}-1){\gamma'_\beta}^{-1}
({\gamma'_\beta} -1)(\gamma_\partial^{-1}-1)) )
}) \\
&+4h \PPsi{\GL}{\tskein'} (\zettaiti{
(\gamma_\partial-1)(\gamma_\partial -1)({\gamma'_\beta}^{-1}-1)({\gamma'_\beta}-1)  \\
&- 
(\gamma_\partial-1)({\gamma'_\beta}^{-1}-1)( \gamma_\partial-1)({\gamma'_\beta}-1) \\
&- 
(\gamma_\partial^{-1}-1)({\gamma'_\beta}^{-1}-1)(\gamma_\partial^{-1}-1)( {\gamma'_\beta}-1) \\
&+
(\gamma_\partial^{-1}-1)({\gamma'_\beta}^{-1}-1)({\gamma'_\beta} -1)(\gamma_\partial^{-1}-1)})
 ) \mod \filt{2N+3} \cGL (\surface) \\
&=\frac{1}{4} e_*(\zettaiti{
(\gamma_\partial-1)(\gamma_\partial^{-1}-1)
(\gamma_\partial -1)({\gamma'_\beta}^{-1}-1) {\gamma'_\beta} \\
&-
(\gamma_\partial-1)(\gamma_\partial^{-1}-1)
({\gamma'_\beta}^{-1}-1)(\gamma_\partial-1) {\gamma'_\beta} \\
&-
(\gamma_\partial-1)(\gamma_\partial^{-1}-1){\gamma'_\beta}^{-1} 
(\gamma_\partial^{-1}-1)( {\gamma'_\beta}-1) \\
&+
(\gamma_\partial-1)(\gamma_\partial^{-1}-1){\gamma'_\beta}^{-1}
({\gamma'_\beta} -1)(\gamma_\partial^{-1}-1)) )
}) \\
&-\frac{1}{2} e_*(\zettaiti{
(\gamma_\partial-1)(\gamma_\partial -1)({\gamma'_\beta}^{-1}-1)({\gamma'_\beta}-1)  \\
&- 
(\gamma_\partial-1)({\gamma'_\beta}^{-1}-1)( \gamma_\partial-1)({\gamma'_\beta}-1) \\
&- 
(\gamma_\partial^{-1}-1)({\gamma'_\beta}^{-1}-1)(\gamma_\partial^{-1}-1)( {\gamma'_\beta}-1) \\
&+
(\gamma_\partial^{-1}-1)({\gamma'_\beta}^{-1}-1)({\gamma'_\beta} -1)(\gamma_\partial^{-1}-1)})
 ) \mod \filt{2N+3} \cGL (\surface)
\end{align*}
where $\gamma'_\beta \defeq \gamma_\alpha \gamma_\beta \gamma_\alpha^{-1}$. Furthermore, since 
\begin{align*}
&e_* (\gamma_\partial -1) \in \filt{N} \GLM (\surface), \\
&e_* (\gamma_\partial-1) =e_* (\gamma_\partial^{-1}-1)
\mod \filt{2N} \GLM (\surface), \\
&e_* ({\gamma'_\beta}-1)=-e_*({\gamma'_\beta}^{-1}-1)
=e_* (\gamma_\beta-1)=-e_*(\gamma_\beta^{-1}-1)
\mod \filt{2} \GLM (\surface),
\end{align*}
 we obtain 
\begin{align*}
&\zeta_\GL ((\surface \times I) (e \circ e^{(1)}_{\surface_1,1}(K_\partial)(-\epsilon)))
-\zeta_\GL ((\surface \times I) (e (K_\partial)(-\epsilon))) \\
&=\frac{1}{2} e_*(\zettaiti{-
(\gamma_\partial-1)(\gamma_\partial -1)({\gamma_\beta}-1)({\gamma_\beta}-1)  \\
&+
(\gamma_\partial-1)({\gamma_\beta}-1)( \gamma_\partial-1)({\gamma_\beta}-1) \\
&+
(\gamma_\partial-1)({\gamma_\beta}-1)(\gamma_\partial-1)( {\gamma_\beta}-1) \\
&-
(\gamma_\partial -1)({\gamma_\beta}-1)({\gamma_\beta} -1)(\gamma_\partial-1)})
 ) \mod \filt{2N+3} \cGL (\surface) \\
&= \zettaiti{(e_* (\gamma_\partial)-1)(e_* (\gamma_\beta) -1)
(e_* (\gamma_\partial)-1)(e_* (\gamma_\beta) -1)
-(e_* (\gamma_\partial)-1)^2(e_* (\gamma_\beta) -1)^2
}. \\
\end{align*}
 It proves the second formula.

Finally, using the second formula, we verify the fourth formula. By definition of the Baker-Campbell-Hausdorff series, we obtain 
\begin{align*}
& (\surface \times I) (e \circ e^{(1)}_{\surface_1,1}(K_\partial)(-\epsilon))_*
\circ((\surface \times I) (e(K_\partial)(-\epsilon))_*)^{-1} \\
&=\exp (\sigma (\tilde{\zeta}_{\GL} (
(\surface \times I) (e \circ e^{(1)}_{\surface_1,1}(K_\partial)(-\epsilon)))))
\circ
\exp(\sigma(-\tilde{\zeta}_{\GL} (
(\surface \times I) (e(K_\partial)(-\epsilon))))) \\
&=\exp (\sigma (\tilde{\zeta}_{\GL} (
(\surface \times I) (e \circ e^{(1)}_{\surface_1,1}(K_\partial)(-\epsilon)))))
\circ
\exp(\sigma(-\tilde{\zeta}_{\GL} (
(\surface \times I) (e(K_\partial)(-\epsilon))))) \\
&=\exp (\sigma (\bch (\tilde{\zeta}_{\GL} ( (
(\surface \times I) (e \circ e^{(1)}_{\surface_1,1}(K_\partial)(-\epsilon)))),
-\tilde{\zeta}_{\GL} (
(\surface \times I) (e(K_\partial)(-\epsilon)))))).\\
\end{align*}
 By the second formula, we have 
\begin{align*}
&\bch (\tilde{\zeta}_{\GL} ( (
(\surface \times I) (e \circ e^{(1)}_{\surface_1,1}(K_\partial)(-\epsilon)))),
-\tilde{\zeta}_{\GL} (
(\surface \times I) (e(K_\partial)(-\epsilon)))) \\
&=\tilde{\zeta}_{\GL} ( (
(\surface \times I) (e \circ e^{(1)}_{\surface_1,1}(K_\partial)(-\epsilon))))
-\tilde{\zeta}_{\GL} (
(\surface \times I) (e(K_\partial)(-\epsilon)))
\mod \filt{4N-2} \cGL (\surface) \\
&= \zettaiti{(e_* (\gamma_\partial)-1)(e_* (\gamma_\beta) -1)
(e_* (\gamma_\partial)-1)(e_* (\gamma_\beta) -1)
-(e_* (\gamma_\partial)-1)^2(e_* (\gamma_\beta) -1)^2
} \\
&\mod \filt{2N+3} \cGL (\surface).
\end{align*}
 Since $\sigma (\filt{2N+3} \cGL (\surface)) \subset \filt{2N+2} \cGLM (\surface)$ for any $N \in \Zlarger{0}$, we get 
\begin{align*}
& (\surface \times I) (e \circ e^{(1)}_{\surface_1,1}(K_\partial)(-\epsilon))_*
\circ((\surface \times I) (e(K_\partial)(-\epsilon))_*)^{-1} \\
&=
\exp (\sigma(\zettaiti{(e_* (\gamma_\partial)-1)(e_* (\gamma_\beta) -1)
(e_* (\gamma_\partial)-1)(e_* (\gamma_\beta) -1)
-(e_* (\gamma_\partial)-1)^2(e_* (\gamma_\beta) -1)^2
})): \\
&\cGLM (\surface, \star )/\filt{2N+2} \cGLM (\surface, \star)
 \to \cGLM (\surface, \star_1, \star_2) /\filt{2N+2} \cGLM (\surface, \star_1, \star_2) \\
\end{align*}
 as desired.

\end{proof}

\begin{rem}
The theorem verifies that the number $2m+2$ in Theorem \ref{thm_KM} is the best possible. In other words, there exist two boundary knots $K,K'$ having the same homotopy type in $\zettaiti{\pi_1} (\surface)
\cap \zettaiti{1+I_{\GLM (\surface)}^m}$ such that
\begin{equation*}
\tilde{\zeta}_{\GL} ((\surface \times I) (K (-\epsilon))
- \tilde{\zeta}_{\GL} ((\surface \times I) (K' (-\epsilon))  \neq 0 
\mod \filt{2m+3} \cGLM (\surface).
\end{equation*}

\end{rem}


\begin{thebibliography}{99}
\large
\bibitem{GL}
S. Garoufalidis and J. Levine, 
\textit{Tree-level invariants of three-manifolds, Massey products and the Johnson homomorphism}, In: Graphs and patterns in mathematics and
theoretical physics, volume 73 of Proc. Sympos. Pure Math., Amer. Math. Soc., Providence 2005, 173–203.

\bibitem{Goldman}
W. M. Goldman,
\textit{Invariant functions on Lie groups and Hamiltonian flows of surface
groups representations}, 
Invent. Math. 85, 263-302(1986).
\bibitem{Habegger2000}
N. Habegger,
\textit{Milnor, Johnson, and tree level perturbative invariants},
preprint.
\bibitem{Johnson80}
D. Johnson,
\textit{An abelian quotient of the mapping class group $\mathcal{I}_g$},
Math. Ann. 249, 225-242(1980).
\bibitem{Kawazumi}
N. Kawazumi and Y. Kuno,
\textit{The logarithms of Dehn twists},
Quantum Topology, Vol. 5(2014), Issue 3, pp. 347–423.
\bibitem{KK}
N. Kawazumi and Y. Kuno,
\textit{Groupoid-theoretical methods in the mapping class groups of surfaces},
arXiv: 1109.6479 (2011), UTMS preprint: 2011–28.
\bibitem{KM2019}
Y.\, Kuno and G.\, Massuyeau,
\emph{Generalized Dehn twists on surfaces and homology cylinders}.
Preprint, arXiv:1902.02592.

\bibitem{Lickorish1962}
W. B. R. Lickorish, 
\textit{A representation of orientable combinatorial $3$-manifolds},
Ann. of Math (2) 76 (9152), 531-540.
\bibitem{MT}
G. Massuyeau and V. Turaev,
\textit{Fox pairings and generalized Dehn twists},
Ann. Inst. Fourier 63 (2013) 2403-2456.
\bibitem{Morita1989}
S. Morita, 
\textit{Casson's invariant for homology 3-spheres and characteristic classes of surface bundles. I}, Topology 28
(1989) 305–323.
\bibitem{Morita_Casson_core}
S. Morita,
\textit{On the structure of the Torelli group and the Casson invariant},
Topology, Volume 30(1991), 603-621.
\bibitem{Morita1999}
S. Morita,
\textit{Structures of the mapping class group of surface: a survey and a prospect},
Proceedings of the Kirbyfest Geom. Topol.Monogr., 2 (1998) 349-406.
\bibitem{Stallings}
J. Stallings,
\textit{Homology and central series of groups},
J. Algebra 2 (1965), 170-181.
\bibitem{TsujiHOMFLY-PTskein}
S. Tsuji,
\textit{
A formula for the action of Dehn twists on HOMFLY-PT
 skein modules and its applications},
preprint, arXiv:1801.00580.
\bibitem{Turaev}
V. G. Turaev, 
\textit{Skein quantization of Poisson algebras of loops on surfaces},
 Ann. Sci. Ecole Norm. Sup. (4) 24 (1991), no. 6, 635-704.



\end{thebibliography}
\end{document}